\documentclass[11pt,twoside,ngerman,english]{scrartcl}
\usepackage[T1]{fontenc}
\usepackage[latin9]{luainputenc}
\usepackage{geometry}
\usepackage{color}
\usepackage{babel}
\usepackage{float}
\usepackage{amsmath}
\usepackage{amsthm}
\usepackage{amssymb}
\usepackage{esint}
\usepackage[unicode=true,pdfusetitle,
 bookmarks=true,bookmarksnumbered=false,bookmarksopen=false,
 breaklinks=false,pdfborder={0 0 1},backref=false,colorlinks=false]
 {hyperref}

\makeatletter
\numberwithin{equation}{section}
\theoremstyle{plain}
\newtheorem{thm}{\protect\theoremname}[section]
  \theoremstyle{plain}
  \newtheorem{assumption}[thm]{\protect\assumptionname}
  \theoremstyle{plain}
  \newtheorem{lem}[thm]{\protect\lemmaname}
  \theoremstyle{plain}
  \newtheorem{cor}[thm]{\protect\corollaryname}
  \theoremstyle{plain}
  \newtheorem{prop}[thm]{\protect\propositionname}
  \theoremstyle{remark}
  
  \theoremstyle{remark}
  \newtheorem{rem}[thm]{\protect\remarkname}
  \theoremstyle{definition}
  \newtheorem{defn}[thm]{\protect\definitionname}

\allowdisplaybreaks[1]

\makeatother

  \addto\captionsenglish{\renewcommand{\assumptionname}{Assumption}}
  \addto\captionsenglish{\renewcommand{\corollaryname}{Corollary}}
  \addto\captionsenglish{\renewcommand{\definitionname}{Definition}}
  \addto\captionsenglish{\renewcommand{\lemmaname}{Lemma}}
  \addto\captionsenglish{\renewcommand{\notationname}{Notation}}
  \addto\captionsenglish{\renewcommand{\propositionname}{Proposition}}
  \addto\captionsenglish{\renewcommand{\remarkname}{Remark}}
  \addto\captionsenglish{\renewcommand{\theoremname}{Theorem}}
  \addto\captionsngerman{\renewcommand{\assumptionname}{Annahme}}
  \addto\captionsngerman{\renewcommand{\corollaryname}{Korollar}}
  \addto\captionsngerman{\renewcommand{\definitionname}{Definition}}
  \addto\captionsngerman{\renewcommand{\lemmaname}{Lemma}}
  \addto\captionsngerman{\renewcommand{\notationname}{Notation}}
  \addto\captionsngerman{\renewcommand{\propositionname}{Satz}}
  \addto\captionsngerman{\renewcommand{\remarkname}{Bemerkung}}
  \addto\captionsngerman{\renewcommand{\theoremname}{Theorem}}
  \providecommand{\assumptionname}{Assumption}
  \providecommand{\corollaryname}{Corollary}
  \providecommand{\definitionname}{Definition}
  \providecommand{\lemmaname}{Lemma}
  \providecommand{\notationname}{Notation}
  \providecommand{\propositionname}{Proposition}
  \providecommand{\remarkname}{Remark}
  \providecommand{\theoremname}{Theorem}

  \newcommand{\R}{\mathbb{R}}
    \newcommand{\eps}{\epsilon}

\begin{document}
\global\long\def\d{\, \mathrm{d}}
\global\long\def\gt{\Gamma_{t}\left(\delta\right)}
\global\long\def\we{\tilde{\mathbf{w}}_{1}}
\global\long\def\vt{\tilde{\mathbf{v}}_{A}^{\epsilon}}
\global\long\def\wei{\mathbf{w}_{1}^{\epsilon}}
\global\long\def\wzw{\mathbf{w}_{2}^{\epsilon}}
\global\long\def\yb{\left(\Psi_{1}^{0}\right)^{\bot}}
\global\long\def\dh{\d\mathcal{H}^{n-1}}
\global\long\def\twe{\tilde{\mathbf{w}}_{1}^{\epsilon}}
\global\long\def\tweh{\tilde{\mathbf{w}}_{1}^{\epsilon,H}}
\global\long\def\twz{\tilde{\mathbf{w}}_{2}^{\epsilon}}
\global\long\def\cte{\tilde{c}^{\epsilon}}
\global\long\def\mte{\tilde{\mu}^{\epsilon}}
\global\long\def\vte{\tilde{\mathbf{v}}^{\epsilon}}
\global\long\def\pte{\tilde{p}^{\epsilon}}
\global\long\def\di{\text{div}}
\global\long\def\vn{\mathbf{u}^{\epsilon,\mathbf{n}}}
\global\long\def\vt{\mathbf{u}^{\epsilon,\tau}}
\global\long\def\cae{\overline{c_{A}^{\epsilon}}}
\global\long\def\caeh{\overline{c_{A}^{\epsilon,H}}}
\global\long\def\wei{\mathbf{w}_{1}^{\epsilon}}
\global\long\def\rs{\mathbf{r}_{\mathrm{S}}^{\epsilon}}
\global\long\def\rdiv{r_{\mathrm{div}}^{\epsilon}}
\global\long\def\rc{r_{\mathrm{CH1}}^{\epsilon}}
\global\long\def\rh{r_{\mathrm{CH2}}^{\epsilon}}
\global\long\def\ra{\mathcal{R}_{\alpha}}
\global\long\def\rsi{\mathbf{r}_{\mathrm{S},I}^{\epsilon}}
\global\long\def\rso{\mathbf{r}_{\mathrm{S},O}^{\epsilon}}
\global\long\def\rdivi{r_{\mathrm{div},I}^{\epsilon}}
\global\long\def\rdivo{r_{\mathrm{div},O}^{\epsilon}}
\global\long\def\rci{r_{\mathrm{CH1},I}^{\epsilon}}
\global\long\def\rco{r_{\mathrm{CH1},O}^{\epsilon}}
\global\long\def\rhi{r_{\mathrm{CH2},I}^{\epsilon}}
\global\long\def\rhO{r_{\mathrm{CH2},O}^{\epsilon}}
\global\long\def\iome{\int_{\Omega_{T_{\epsilon}}}}
\global\long\def\rsb{\mathbf{r}_{\text{S},\mathbf{B}}^{\epsilon}}
\global\long\def\rdivb{r_{\text{div},\mathbf{B}}^{\epsilon}}
\global\long\def\rcb{r_{\text{CH1},\mathbf{B}}^{\epsilon}}
\global\long\def\rhb{r_{\text{CH2},\mathbf{B}}^{\epsilon}}
\global\long\def\trhb{\tilde{r}_{\text{CH2},\mathbf{B}}^{\epsilon}}

\title{Sharp Interface Limit of a Stokes/Cahn-Hilliard System, Part I:
Convergence Result}

\author{Helmut Abels\thanks{  \textit{Fakult\"at f\"ur Mathematik,   Universit\"at Regensburg,   93040 Regensburg,   Germany}   \textsf {helmut.abels@ur.de} } \  and Andreas Marquardt\thanks{\textit   {Fakult\"at f\"ur Mathematik,   Universit\"at Regensburg,   93040 Regensburg,   Germany}  } }
\maketitle
\begin{abstract}
We consider the sharp interface limit of a coupled Stokes/Cahn\textendash Hilliard
system in a two dimensional, bounded and smooth domain, i.e., we
consider the limiting behavior of solutions when a parameter $\epsilon>0$
corresponding to the thickness of the diffuse interface tends to zero.
We show that for sufficiently short times the solutions to the Stokes/Cahn\textendash Hilliard
system converge to solutions of a sharp interface model, where the
evolution of the interface is governed by a Mullins\textendash Sekerka
system with an additional convection term coupled to a two\textendash phase
stationary Stokes system with the Young-Laplace law for the jump of  an extra contribution to the stress
tensor, representing  capillary stresses.
We prove the convergence result by estimating the difference between the exact and an approximate
solutions. To this end we make use of modifications of spectral estimates shown
by X.\ Chen for the linearized Cahn-Hilliard operator. The treatment
of the coupling terms requires careful estimates, the use of the refinements of the latter spectral estimate and a suitable structure of the approximate solutions, which will be constructed in the second part of this contribution.
\end{abstract}

{\small\noindent
}

\section{Introduction and Overview}

Classically, the transition between two immiscible fluids was considered
to be sharp, in the sense of an appearance of a lower-dimensional
surface separating the phases. The behavior of a multiphase system
is then governed by the intricate interactions between the bulk regions
and the interface, mathematically expressed as equations of motion,
which hold in each fluid, complemented by boundary conditions at the
(free) surface. Models incorporating these ideas \textendash{} often
called \emph{sharp interface models} \textendash{} and the corresponding
free-boundary problems have been widely studied and used to great
success in describing a multitude of physical and biological phenomena.
However, fundamental problems arise in the analysis and numerical
simulation of such problems, whenever the considered interfaces develop
singularities. In fluid dynamics, topological changes such as the
pinch off of droplets or collisions are non-negligible features of
many systems, having a significant impact on the flow. Conversely,
\emph{diffuse interface models} turn out to provide a promising, alternative
approach to describe such phenomena and overcome the associated difficulties.
In these diffuse interface (or \emph{phase field)} methods, a partial
mixing of the two phases throughout a thin interfacial layer, heuristically
viewed to have a thickness proportional to a length scale parameter
$\epsilon>0$, is taken into account. Naturally, the question of the behavior for the limit  $\epsilon\rightarrow0$
arises. This so-called \emph{sharp interface limit} is in fact a question
about the connection of sharp and diffuse interface models. Concerning
the flow of two macroscopically immiscible, viscous, incompressible
Newtonian fluids with matched densities, a fundamental and broadly
accepted diffuse interface model is the so-called \emph{model H}.
This model consists of a Navier-Stokes system coupled with the Cahn-Hilliard
equation and was derived in \cite{hohenberg,Gurtin}. The sharp interface limit was studied with the method of formally matched asymptotics in \cite{agg}
and the existence of solutions for the model H was shown in \cite{abels_matched_densities,boyertheman}.
Regarding the formal sharp interface limit, short time existence of
strong solutions was shown in \cite{abelswilke} and existence of
weak solutions for long times in \cite{abelsroger}. Despite these
analytic results and the formal findings for the sharp interface limit,
there are only few attempts at rigorously discussing the sharp interface
limit for the model H. Using the notion of varifold solutions as discussed
in \cite{chenvarifold} such results for large times were shown in
\cite{abelsroger} for the model H and in \cite{abelslengeler} also
for the more general situation of fluids with different densities. But the notion of solution for the latter contributions is rather weak and no rates of convergence were obtained and convergence was only shown for a suitable subsequence. 

For the Allen-Cahn and Cahn-Hilliard equation another approach is based on the works \cite{schatz} and \cite{abc},
where the method of formally matched asymptotics is made rigorous.
However, in view of two-phase flow models in fluid mechanics and the
arising difficulties therein, the first and so far only convergence
result with convergence rates in strong norms is \cite{nsac}. More precisely, considering
a coupled Stokes/Allen-Cahn system in two dimensions, it is shown
that smooth solutions of the diffuse interface system converge for
short times to solutions of a sharp interface model, where the evolution
of the free surface is governed by a convective mean curvature flow
coupled to a two-phase Stokes system together with the Young-Laplace law for the jump of the stress tensor,
accounting for capillary forces. This contribution builds upon the
ideas introduced in \cite{nsac} and aims to establish the first rigorous
result in strong norms for a sharp interface limit of a two phase
flow model involving the Cahn-Hilliard equation with convergence rates.
In doing so, we hope to build another cornerstone on the way to rigorously
showing the sharp interface limit for model H.

More precisely we consider the Stokes/Cahn-Hilliard system
\begin{align}
-\Delta\mathbf{v}^{\epsilon}+\nabla p^{\epsilon} & =\mu^{\epsilon}\nabla c^{\epsilon} &  & \text{in }\Omega_{T},\label{eq:StokesPart}\\
\operatorname{div}\mathbf{v}^{\epsilon} & =0 &  & \text{in }\Omega_{T},\label{eq:StokesPart2}\\
\partial_{t}c^{\epsilon}+\mathbf{v}^{\epsilon}\cdot\nabla c^{\epsilon} & =\Delta\mu^{\epsilon} &  & \text{in }\Omega_{T},\label{eq:CH-Part1}\\
\mu^{\epsilon} & =-\epsilon\Delta c^{\epsilon}+\frac{1}{\epsilon}f'\left(c^{\epsilon}\right) &  & \text{in }\Omega_{T},\label{eq:CH-Part2}\\
c^{\epsilon}|_{t=0} & =c_{0}^{\epsilon} &  & \text{in }\Omega,\label{eq:CH-Anfang}\\
\left(-2D_{s}\mathbf{v}^{\epsilon}+p^{\epsilon}\mathbf{I}\right)\cdot\mathbf{n}_{\partial\Omega} & =\alpha_{0}\mathbf{v}^{\epsilon},\quad \mu^\epsilon =0,\quad c^\epsilon =-1 &  & \text{on }\partial \Omega\times (0,T).\label{eq:StokesBdry}
\end{align}
 Here $T>0$, $\Omega\subset\mathbb{R}^{2}$ is a bounded and smooth
domain, $\Omega_{T}:=\Omega\times\left(0,T\right)$ and $\alpha_{0}>0$
is fixed. $\mathbf{v}^{\epsilon}\colon \Omega_T \to \R^2$ and $p^{\epsilon}\colon \Omega_T\to \R$ represent the
mean velocity and pressure, $D_{s}\mathbf{v}^{\epsilon}:=\frac{1}{2}\big(\nabla\mathbf{v}^{\epsilon}+(\nabla\mathbf{v}^{\epsilon})^{T}\big)$,
$c^{\epsilon}\colon \Omega\to \R$ is an order parameter representing the concentration
difference of the fluids and $\mu^{\epsilon}\colon \Omega_T\to \R$ is the chemical potential
of the mixture. Moreover, $c_{0}^{\epsilon}\colon \Omega\to \R$ is a suitable initial
value, specified in Theorem \ref{Main} and $f\colon \mathbb{R}\rightarrow\mathbb{R}$
is a double well potential. The system corresponds to the model H if one would add the convection term $\partial_t \mathbf{v}^\epsilon+ \mathbf{v}^\epsilon\cdot \nabla \mathbf{v}^\epsilon$ on the left-hand side to \eqref{eq:StokesPart}. 

Existence of smooth solutions to (\ref{eq:StokesPart})\textendash (\ref{eq:StokesBdry})
can be shown with similar methods as in \cite{abels_matched_densities}.
A word is in order about the choice of boundary conditions
(\ref{eq:StokesBdry}). The reason
we prescribe such boundary conditions for $\mathbf{v}^{\epsilon}$
instead of periodic, no-slip or Navier boundary conditions, are major
difficulties which arise in the construction of the approximate solutions
for $\mathbf{v}^{\epsilon}$. A more detailed account is given in
\cite[Remark~3.9]{NSCH2}. Classically, the Cahn-Hilliard system is
complemented with Neumann boundary conditions for $c^{\epsilon}$
and $\mu^{\epsilon}$. While it is rather unproblematic to adapt the
present work to Neumann boundary conditions for $c^{\epsilon}$, major
issues arise when considering $\partial_{\mathbf{n}_{\partial\Omega}}\mu^\epsilon=0$
instead of $\mu^\epsilon=0$, see Remark \ref{rem:nodirnosol} below.
To circumvent these problems and as the focus of our interest and
analysis lies in the obstacles and difficulties occurring close to
the interface $\Gamma_{t}$, we decided on the present choice of boundary
conditions. We will show that the sharp interface limit of (\ref{eq:StokesPart})\textendash (\ref{eq:StokesBdry})
is given by the system
\begin{align}
-\Delta\mathbf{v}+\nabla p & =0 &  & \text{in }\Omega^{\pm}(t),t\in\left[0,T_{0}\right],\label{eq:S-SAC1}\\
\operatorname{div}\mathbf{v} & =0 &  & \text{in }\Omega^{\pm}(t),t\in\left[0,T_{0}\right],\label{eq:S-SAC2}\\
\Delta\mu & =0 &  & \text{in }\Omega^{\pm}(t),t\in\left[0,T_{0}\right],\label{eq:S-SAC3}\\
\left(-2D_{s}\mathbf{v}+p\mathbf{I}\right)\mathbf{n}_{\partial\Omega} & =\alpha_{0}\mathbf{v} &  & \text{on }\partial_{T_{0}}\Omega,\\
\mu & =0 &  & \text{on }\partial_{T_{0}}\Omega,\\
\left[2D_{s}\mathbf{v}-p\mathbf{I}\right]\mathbf{n}_{\Gamma_{t}} & =-2\sigma H_{\Gamma_{t}}\mathbf{n}_{\Gamma_{t}} &  & \text{on }\Gamma_{t},t\in\left[0,T_{0}\right],\label{eq:S-SAC4}\\
\mu & =\sigma H_{\Gamma_{t}} &  & \text{on }\Gamma_{t},t\in\left[0,T_{0}\right],\label{eq:S-SAC5}\\
-V_{\Gamma_{t}}+\mathbf{n}_{\Gamma_{t}}\cdot\mathbf{v} & =\frac{1}{2}\left[\mathbf{n}_{\Gamma_{t}}\cdot\nabla\mu\right] &  & \text{on }\Gamma_{t},t\in\left[0,T_{0}\right],\label{eq:S-SAC6}\\
\left[\mathbf{v}\right] & =0 &  & \text{on }\Gamma_{t},t\in\left[0,T_{0}\right],\\
\Gamma\left(0\right) & =\Gamma_{0}.\label{eq:S-SAC8}
\end{align}
 Here $T_{0}>0$, $\Omega$ is the disjoint union of smooth domains
$\Omega^{+}(t)$, $\Omega^{-}(t)$ and a curve
$\Gamma_{t}\subseteq \Omega$ for every $t\in\left[0,T_{0}\right]$, where $\Gamma_{t}=\partial\Omega^{+}(t)$,
$\mathbf{n}_{\Gamma_{t}}$ is the exterior normal with respect to
$\Omega^{-}(t)$, and $H_{\Gamma_{t}}$ and $V_{\Gamma_{t}}$
denote the mean curvature and normal velocity of the interface $\Gamma_{t}$.
Furthermore, $\partial_{T_{0}}\Omega:=\partial\Omega\times\left(0,T_{0}\right)$,
$\Gamma_{0}$ is a given initial surface and we use the definitions
\begin{align}
\left[g\right]\left(p,t\right) & :=\lim_{h\searrow0}\left(g(p+\mathbf{n}_{\Gamma_{t}}(p)h)-g(p-\mathbf{n}_{\Gamma_{t}}(p)h)\right)\text{ for }p\in\Gamma_{t},\nonumber \\
\sigma & :=\frac{1}{2}\int_{-\infty}^{\infty}\theta_{0}'(s)^{2}\d s,\label{eq:sigma}
\end{align}
where $\theta_{0}\colon\mathbb{R}\rightarrow\mathbb{R}$ is the so-called
optimal profile, i.e., the unique solution to the ordinary differential equation
\begin{equation}
-\theta_{0}''+f'\left(\theta_{0}\right)=0\quad\text{in }\mathbb{R},\quad \theta_{0}\left(0\right)=0,\;\lim_{\rho\rightarrow\pm\infty}\theta_{0}(\rho)=\pm1.\label{eq:optprofdef}
\end{equation}
Regarding the existence of local strong solutions of (\ref{eq:S-SAC1})\textendash (\ref{eq:S-SAC8}),
the proof in \cite{abelswilke} may be adapted, where a coupled Navier-Stokes/Mullins-Sekerka
system was treated. Regularity theory for parabolic equations and
the Stokes equation may then be used to show smoothness of the solution
for smooth initial values.

Assuming that suitable approximate solutions
$\left( c_{A}^{\epsilon},\mu_{A}^{\epsilon},\mathbf{v}_{A}^{\epsilon},p_{A}^{\epsilon}\right) _{\epsilon>0}$
to (\ref{eq:StokesPart})\textendash (\ref{eq:StokesBdry}) are constructed we show
the existence of some $T_{1}>0$ such that the difference between
$c^{\epsilon}$ and $c_{A}^{\epsilon}$ goes to zero in $L^{\infty}\left(0,T_{1};H^{-1}(\Omega)\right)$
with $H^{-1}(\Omega):=\left(H_{0}^{1}(\Omega)\right)'$,
$L^{2}\left(\Omega_{T_{1}}\right)$, $L^{2}\left(0,T_{1};H^{1}(\Omega)\right)$
and many other norms as $\epsilon\rightarrow0$ with explicit convergence
rates, for some small $T_{1}>0$. These rates will depend on the order
up to which the approximate solutions have been constructed. Moreover,
we will also present convergence rates for the error $\mathbf{v}^{\epsilon}-\mathbf{v}_{A}^{\epsilon}$
in $L^{1}\left(0,T_{1};L^{q}(\Omega)\right)$ for $q\in\left(1,2\right)$.
This result is stated in Theorem \ref{Main}. The key to this endeavors
will be a modification of the spectral estimate for the linearized
Cahn-Hilliard operator as given in \cite{chen}, see Theorem~\ref{specHillmod} below. As in \cite{nsac}, the main difficulties which arise
in the treatment of the Stokes/Cahn-Hilliard system are due to the
appearance of the capillary term $\mu^{\epsilon}\nabla c^{\epsilon}$
in (\ref{eq:StokesPart}) and the convective term $\mathbf{v^{\epsilon}}\cdot\nabla c^{\epsilon}$
in (\ref{eq:CH-Part1}). Although we may build upon the insights gained
in the cited article, several new and severe obstacles arise in the
context of system (\ref{eq:StokesPart})\textendash (\ref{eq:StokesBdry})
which have to be overcome. A novelty in this context is the introduction
of terms of fractional order in the asymptotic expansions. The necessity
of such terms is at its core a consequence of our treatment of the
convective term $\mathbf{v}^{\epsilon}\cdot\nabla c^{\epsilon}$.
Where \cite{nsac} relied on the intricate analysis of a second order,
parabolic, degenerate partial differential equation in the construction
of the highest order terms, the introduction of fractional order terms
renders such considerations unnecessary. The caveat being, that while
the produced fractional order terms are smooth, they may not be estimated
uniformly in $\epsilon$ in arbitrarily strong norms. This is the
cause for many technical subtleties in \cite{NSCH2}, where the construction
is discussed and where estimates for the remainder are shown. See also the second author's PhD-thesis \cite{ichPhD}, which contains the results of this contribution and \cite{NSCH2}.

Throughout this work we make the following assumptions: Let $\Omega\subset\mathbb{R}^{2}$
be a smooth domain, $\Gamma_{0}\subset\subset\Omega$ be a given,
smooth, non-intersecting, closed initial curve. Let moreover $\left(\mathbf{v},p,\mu,\Gamma\right)$
be a smooth solution to (\ref{eq:S-SAC1})\textendash (\ref{eq:S-SAC8})
and $\left(c^{\epsilon},\mu^{\epsilon},\mathbf{v}^{\epsilon},p^{\epsilon}\right)$
be a smooth solution to (\ref{eq:StokesPart})\textendash (\ref{eq:StokesBdry})
for some $T_{0}>0$. We assume that $\Gamma=\cup_{t\in\left[0,T_{0}\right]}\Gamma_{t}\times\left\{ t\right\} $
is a smoothly evolving hypersurface in $\mathbb{R}^{2}$, where $\left(\Gamma_{t}\right)_{t\in\left[0,T_{0}\right]}$
are compact, non-intersecting, closed curves in $\Omega$. We define
$\Omega^{+}(t)$ to be the inside of $\Gamma_{t}$ and
set $\Omega^{-}(t)$ such that $\Omega$ is the disjoint
union of $\Omega^{+}(t)$, $\Omega^{-}(t)$
and $\Gamma_{t}$. Moreover we define $\Omega_{T}^{\pm}=\cup_{t\in\left[0,T\right]}\Omega^{\pm}(t)\times\left\{ t\right\} $,
$\Omega_{T}:=\Omega\times\left(0,T\right)$ and also $\partial_{T}\Omega:=\partial\Omega\times\left(0,T\right)$
for $T\in\left[0,T_{0}\right]$. We define $\mathbf{n}_{\Gamma_{t}}(p)$
for $p\in\Gamma_{t}$ as the exterior normal with respect to $\Omega^{-}(t)$
and $V_{\Gamma_{t}}$, and $H_{\Gamma_{t}}$ as the normal velocity
and mean curvature of $\Gamma_{t}$ with respect to $\mathbf{n}_{\Gamma_{t}}$,
$t\in\left[0,T_{0}\right]$. Let 
\[
d_{\Gamma}:\Omega_{T_{0}}\rightarrow\mathbb{R},\;\left(x,t\right)\mapsto\begin{cases}
\mbox{dist}\left(\Omega^{-}(t),x\right) & \mbox{if }x\notin\Omega^{-}(t),\\
-\mbox{dist}\left(\Omega^{+}(t),x\right) & \mbox{if }x\in\Omega^{-}(t)
\end{cases}
\]
 denote the signed distance function to $\Gamma$ such that $d_{\Gamma}$
 is positive inside $\Omega_{T_{0}}^{+}$. We write
 $$
 \Gamma_{t}(\alpha):=\left\{ \left.x\in\Omega\right|\left|d_{\Gamma}\left(x,t\right)\right|<\alpha\right\}
 $$
for $\alpha>0$ and set $\Gamma(\alpha;T):=\bigcup_{t\in\left[0,T\right]}\Gamma_{t}\left(\alpha\right)\times\left\{ t\right\} $
for $T\in\left[0,T_{0}\right]$. Moreover, we assume that $\delta>0$
is a small positive constant such that $\text{dist}\left(\Gamma_{t},\partial\Omega\right)>5\delta$
for all $t\in\left[0,T_{0}\right]$ and such that the orthogonal projection $\operatorname{Pr}_{\Gamma_{t}}:\Gamma_{t}(3\delta)\rightarrow\Gamma_{t}$
is well-defined and smooth for all $t\in\left[0,T_{0}\right]$. In
the following we often use the notation $\Gamma (\alpha):=\Gamma (\alpha;T_{0})$
as a simplification. We also define a tubular neighborhood around
$\partial\Omega$: For this let $d_{\mathbf{B}}\colon \Omega\rightarrow\mathbb{R}$
be the signed distance function to $\partial\Omega$ such that $d_{\mathbf{B}}<0$
in $\Omega$. As for $\Gamma_{t}$ we define a tubular neighborhood
by $\partial\Omega\left(\alpha\right):=\left\{ x\in\Omega\left|-\alpha<d_{\mathbf{B}}(x)<0\right.\right\} $
and $\partial_{T}\Omega\left(\alpha\right):=\left\{ \left.\left(x,t\right)\in\Omega_{T}\right|d_{\mathbf{B}}(x)\in\left(-\alpha,0\right)\right\} $
for $\alpha>0$ and $T\in\left(0,T_{0}\right]$. Moreover, we denote
the outer unit normal to $\Omega$ by $\mathbf{n}_{\partial\Omega}$
and denote the normalized tangent by $\boldsymbol{\tau}_{\partial\Omega}$, which
is fixed by the relation
\[
\mathbf{n}_{\partial\Omega}(p)=\left(\begin{array}{cc}
0 & -1\\
1 & 0
\end{array}\right)\boldsymbol{\tau}_{\partial\Omega}(p)
\]
for $p\in\partial\Omega$. Finally we assume that $\delta>0$ is chosen
small enough such that the orthogonal projection $\operatorname{Pr}_{\partial\Omega}:\partial\Omega(\delta)\rightarrow\partial\Omega$
along the normal $\mathbf{n}_{\partial\Omega}$ is also well-defined
and smooth.

Concerning the potential $f$, we assume that it is a fourth order
polynomial, satisfying 
\begin{equation}
  f\left(\pm1\right)=f'\left(\pm1\right)=0,\,f''\left(\pm1\right)>0,\:f(s)=f\left(-s\right)>0\quad\text{for all }s\in\R
  \label{eq:f}
\end{equation}
and fulfilling $f^{\left(4\right)}>0$. Then
the ordinary differential equation (\ref{eq:optprofdef}) allows for
a unique, monotonically increasing solution $\theta_{0}:\mathbb{R}\rightarrow\left(-1,1\right)$.
This solution furthermore satisfies the decay estimate
\begin{equation}
\left|\theta_{0}^{2}(\rho)-1\right|+ |\theta_{0}^{\left(n\right)}(\rho)|\leq C_{n}e^{-\alpha\left|\rho\right|}\quad\text{for all }\rho\in\mathbb{R},\;n\in\mathbb{N}\backslash\left\{ 0\right\} \label{eq:optimopti}
\end{equation}
for constants $C_{n}>0$, $n\in\mathbb{N}\backslash\{ 0\}$,
and fixed $\alpha\in\big(0,\min\{ \sqrt{f''(-1)},\sqrt{f''(1)}\} \big)$.
We denote by $\xi\in C^{\infty}\left(\mathbb{R}\right)$ a
cut-off function such that
\begin{equation}
\xi(s)=1\text{ if }\left|s\right|\leq\delta,\,\xi(s)=0\text{ if }\left|s\right|>2\delta,\text{ and }0\geq s\xi'(s)\geq-4\quad \text{if }\delta\leq\left|s\right|\leq2\delta.\label{eq:cut-off}
\end{equation}

The following theorem is the main theorem of this article (for an
explanation of the used notations see the preliminaries section):
\begin{thm}
  \label{Main}
  Let $M\in\mathbb{N}$ with $M\geq 4$, $\xi$ be a cut-off function satisfying \eqref{eq:cut-off},  $\gamma(x):=\xi\left(4d_{\mathbf{B}}(x)\right)$
for all $x\in\Omega$ and let for $\epsilon\in\left(0,1\right)$ a
smooth function $\psi_{0}^{\epsilon}\colon \Omega\rightarrow\mathbb{R}$
be given, which satisfies $\left\Vert \psi_{0}^{\epsilon}\right\Vert _{C^{1}(\Omega)}\leq C_{\psi_{0}}\epsilon^{M}$
for some $C_{\psi_{0}}>0$ independent of $\epsilon$. Then there
are smooth functions $c_{A}^{\epsilon}\colon \Omega\times\left[0,T_{0}\right]\rightarrow\mathbb{R}$, $\mathbf{v}_{A}^{\epsilon}\colon\Omega\times\left[0,T_{0}\right]\rightarrow\mathbb{R}^{2}$
for $\epsilon\in\left(0,1\right)$ such that the following holds: 

Let  $\left(\mathbf{v}^{\epsilon},p^{\epsilon},c^{\epsilon},\mu^{\epsilon}\right)$
be smooth solutions to \eqref{eq:StokesPart}\textendash \eqref{eq:StokesBdry}
with initial value
\begin{align}
c_{0}^{\epsilon}(x) & =c_{A}^{\epsilon}\left(x,0\right)+\psi_{0}^{\epsilon}(x)\label{eq:canf}
\end{align}
for all $x\in\Omega$. Then there are some $\epsilon_{0}\in\left(0,1\right]$,
$K>0$, $T\in\left(0,T_{0}\right]$ such that 
\begin{subequations}\label{eq:Main}
  \begin{align}
    \left\Vert c^{\epsilon}-c_{A}^{\epsilon}\right\Vert _{L^{2}(\Omega_T)}+\left\Vert \nabla^{\Gamma}\left(c^{\epsilon}-c_{A}^{\epsilon}\right)\right\Vert _{L^{2}\left(\Gamma (\delta,T)\right)} & \le K\epsilon^{M-\frac{1}{2}},\label{eq:Main1}\\
   \epsilon\left\Vert \nabla\left(c^{\epsilon}-c_{A}^{\epsilon}\right)\right\Vert _{L^{2}\left(\Omega_T\setminus \Gamma(\delta,T)\right)}+\left\Vert c^{\epsilon}-c_{A}^{\epsilon}\right\Vert _{L^{2}\left(\Omega_T\setminus \Gamma(\delta,T)\right)} & \leq K\epsilon^{M+\frac{1}{2}},\label{eq:Main2}\\
   \epsilon^{\frac{3}{2}}\left\Vert \partial_{\mathbf{n}}\left(c^{\epsilon}-c_{A}^{\epsilon}\right)\right\Vert _{L^{2}\left(\Gamma(\delta,T)\right)}+\left\Vert c^{\epsilon}-c_{A}^{\epsilon}\right\Vert _{L^{\infty}\left(0,T;H^{-1}(\Omega)\right)} & \leq K\epsilon^{M},\label{eq:Main3}\\
   \int_{\Omega_{T}}\epsilon\left|\nabla\left(c^{\epsilon}-c_{A}^{\epsilon}\right)\right|^{2}+\epsilon^{-1}f''\left(c_{A}^{\epsilon}\right)\left(c^{\epsilon}-c_{A}^{\epsilon}\right)^{2}\d\left(x,t\right) & \leq K^{2}\epsilon^{2M},\label{eq:Main4}\\
   \left\Vert \gamma\left(c^{\epsilon}-c_{A}^{\epsilon}\right)\right\Vert _{L^{\infty}\left(0,T;L^{2}(\Omega)\right)}+\epsilon^{\frac{1}{2}}\left\Vert \gamma\Delta\left(c^{\epsilon}-c_{A}^{\epsilon}\right)\right\Vert _{L^{2}\left(\Omega_{T}\right)} & \leq K\epsilon^{M-\frac{1}{2}},\label{eq:Main5}\\
   \left\Vert \gamma\nabla\left(c^{\epsilon}-c_{A}^{\epsilon}\right)\right\Vert_{L^{2}\left(\Omega_{T}\right)}+\left\Vert \gamma\left(c^{\epsilon}-c_{A}^{\epsilon}\right)\nabla\left(c^{\epsilon}-c_{A}^{\epsilon}\right)\right\Vert _{L^{2}\left(\Omega_{T}\right)} & \leq K\epsilon^{M},\label{eq:Main6}
\end{align}
\end{subequations}
and for $q\in\left(1,2\right)$
\begin{equation}
\left\Vert \mathbf{v}^{\epsilon}-\mathbf{v}_{A}^{\epsilon}\right\Vert _{L^{1}\left(0,T;L^{q}(\Omega)\right)}\leq C\left(K,q\right)\epsilon^{M-\frac{1}{2}},\label{eq:Mainv}
\end{equation}
 hold for all $\epsilon\in\left(0,\epsilon_{0}\right)$ and some $C\left(K,q\right)>0$.
Moreover, we have
\begin{align}
\lim_{\epsilon\rightarrow0}c_{A}^{\epsilon} & =\pm1\text{ in }L^{\infty}\left(\Omega_{T}'\right)\label{eq:Maincconverge}
\end{align}
and
\begin{equation}
\lim_{\epsilon\rightarrow0}\mathbf{v}_{A}^{\epsilon}=\mathbf{v}^{\pm}\text{ in }L^{6}((s,t);H^{2}\left(\Omega'\right)^{2})\label{eq:Mainvconverge}
\end{equation}
for every $\Omega'\times (s,t)\subset\subset\Omega_{T}^{\pm}$.
\end{thm}

Throughout this work we will often consider the following assumptions. 
\begin{assumption}
\label{assu:Main-est}Let $M\in\mathbb{N}$ with $M\geq 4$ and $\gamma(x):=\xi (4d_{\mathbf{B}}(x))$
for all $x\in\Omega$. We assume that $c_{A}\colon \Omega\times\left[0,T_{0}\right]\rightarrow\mathbb{R}$
is a smooth function and that there are $\epsilon_{0}\in\left(0,1\right)$,
$K\geq1$ and a family $\left(T_{\epsilon}\right)_{\epsilon\in\left(0,\epsilon_{0}\right)}\subset\left(0,T_{0}\right]$
such that the following holds: If $c^{\epsilon}$ is given as in Theorem
\ref{Main} with $c_{0}^{\epsilon}(x)=c_{A}\left(x,0\right)$,
then it holds for $R:=c^{\epsilon}-c_{A}^{\epsilon}$
\begin{subequations}\label{eq:Main-est}
\begin{align}
\Vert R\Vert _{L^{2}\left(\Omega_{T_{\epsilon}}\right)}+\Vert \nabla^{\Gamma}R\Vert _{L^{2}\left(\Gamma(T_\eps,\delta)\right)}+\left\Vert \left(\tfrac{1}{\epsilon}R,\nabla R\right)\right\Vert _{L^{2}\left(\Omega_{T_{\epsilon}}\setminus \Gamma(T_\eps,\delta)\right)} & \le K\epsilon^{M-\frac{1}{2}},\label{eq:Main-est-a}\\
\epsilon^{\frac{3}{2}}\left\Vert \partial_{\mathbf{n}}R\right\Vert _{L^{2}\left(\Gamma(T_\eps,\delta)\right)}+\left\Vert R\right\Vert _{L^{\infty}\left(0,T_{\epsilon};H^{-1}(\Omega)\right)} & \leq K\epsilon^{M},\label{eq:Main-est-b}\\
\int_{\Omega_{T_{\epsilon}}}\epsilon\left|\nabla R\right|^{2}+\tfrac1{\epsilon}f''(c_{A}^{\epsilon})R^{2}\d\left(x,t\right) & \leq K^{2}\epsilon^{2M},\label{eq:Main-est-c}\\
\epsilon^{\frac{1}{2}}\left\Vert \gamma R\right\Vert _{L^{\infty}\left(0,T_{\epsilon};L^{2}(\Omega)\right)}+\left\Vert \left(\epsilon\gamma\Delta R,\gamma\nabla R,\gamma R\left(\nabla R\right)\right)\right\Vert _{L^{2}\left(\Omega_{T_{\epsilon}}\right)} & \leq K\epsilon^{M}\label{eq:Main-est-d}
\end{align}
\end{subequations}for all $\epsilon\in\left(0,\epsilon_{0}\right)$.
Moreover, we assume that there exist $\epsilon_{0}>0$ and a constant
$C_{0}>0$ independent of $\epsilon$, such that
\begin{equation}
E^{\epsilon}\left(c_{0}^{\epsilon}\right)+ \left\Vert c_{0}^{\epsilon}\right\Vert _{L^{\infty}(\Omega)}\leq C_{0}\label{eq:Eepsbes}
\end{equation}
for all $\epsilon\in\left(0,\epsilon_{0}\right)$. 
\end{assumption}

As a first result, we give an energy estimate for (\ref{eq:StokesPart})\textendash (\ref{eq:StokesBdry}).
We consider for $\epsilon>0$ the free energy 
\begin{equation}
E^{\epsilon}\left(c^{\epsilon}\right)(t)=\frac{\epsilon}{2}\int_{\Omega}\left|\nabla c^{\epsilon}\left(x,t\right)\right|^{2}\d x+\frac{1}{\epsilon}\int_{\Omega}f\left(c^{\epsilon}\left(x,t\right)\right)\d x\text{ for }t\in\left[0,T_{0}\right].\label{eq:gb}
\end{equation}
Then one derives
\begin{equation}\label{eq:EnergyEstim}
  \sup_{0\leq t\leq T} E^\epsilon(c^\epsilon(t))+\int_0^T\int_\Omega \left(|\nabla \mathbf{v}^\epsilon|^2+|\nabla \mu^\epsilon|^2\right)\,dx \, dt +\alpha_0\int_0^t \int_{\partial\Omega}|v|^2\,d\sigma\, dt \leq C_0.
\end{equation}
in a standard manner from testing \eqref{eq:StokesPart} with $\mathbf{v}^\epsilon$, \eqref{eq:CH-Part1} with $\mu^\epsilon$ and \eqref{eq:CH-Part2} with $\partial_t c^\epsilon$ and integration by parts.
As a corollary we obtain
\begin{lem}\label{lem:energy}
\label{energy}Let $\left(c^{\epsilon},\mu^{\epsilon},\mathbf{v}^{\epsilon},p^{\epsilon}\right)$
be a classical solution to \eqref{eq:StokesPart}-\eqref{eq:StokesBdry}
and let $\epsilon_{0}>0$ and $C_{0}>0$ be given such that \eqref{eq:Main-est}
and \eqref{eq:Eepsbes} hold true. Then there is some $\epsilon_{1}\in\left(0,\epsilon_{0}\right)$
and some constant $C>0$, depending only on $T_{0},C_{0}$ and $\epsilon_{0}$,
such that 
\[
\epsilon^{7}\left\Vert \Delta c^{\epsilon}\right\Vert _{L^{2}\left(\Omega_{t}\right)}^{2}+\epsilon\sup_{\tau\in\left[0,t\right]}\left\Vert \nabla c^{\epsilon}\left(.,\tau\right)\right\Vert _{L^{2}(\Omega)}^{2}+\left\Vert \left(\nabla\mu^{\epsilon},\nabla\mathbf{v}^{\epsilon}\right)\right\Vert _{L^{2}\left(\Omega_{t}\right)}^{2}+\alpha_{0}\left\Vert \mathbf{v}^{\epsilon}\right\Vert _{L^{2}\left(\partial_{t}\Omega\right)}^{2}\leq C
\]
for all $t\in\left[0,T_\eps\right]$ and $\epsilon\in\left(0,\epsilon_{1}\right)$.
\end{lem}

\begin{proof}
All estimates apart from the one for $\eps^7\left\Vert \Delta c^{\epsilon}\right\Vert _{L^{2}\left(\Omega_{t}\right)}^{2}$
follow directly from \eqref{eq:EnergyEstim}.
Because of the Dirichlet boundary condition of $\mu^{\epsilon}$
we get
\begin{align*}
\left\Vert \Delta c^{\epsilon}\right\Vert _{L^{2}\left(\Omega_{t}\right)} & \leq\frac{1}{\epsilon}\left\Vert \mu^{\epsilon}\right\Vert _{L^{2}\left(\Omega_{t}\right)}+\frac{1}{\epsilon^{2}}\left\Vert f'\left(c^{\epsilon}\right)\right\Vert _{L^{2}\left(\Omega_{t}\right)} \leq\frac{1}{\epsilon}C\left\Vert \nabla\mu^{\epsilon}\right\Vert _{L^{2}\left(\Omega_{t}\right)}+\frac{1}{\epsilon^{2}}\left\Vert f'\left(c^{\epsilon}\right)\right\Vert _{L^{2}\left(\Omega_{t}\right)}\\
 & \leq\frac{C}{\epsilon^{2}}\left(1+\left\Vert c^{\epsilon}\right\Vert _{L^{6}\left(\Omega_{t}\right)}^{3}\right) \leq\frac{C}{\epsilon^{2}}\left(1+\left\Vert \nabla c^{\epsilon}\right\Vert _{L^{\infty}\left(0,t;L^{2}(\Omega)\right)}^{3}\right) \leq\frac{C}{\epsilon^{2}}\left(1+\epsilon^{-\frac{3}{2}}\right)
\end{align*}
for $\epsilon$ small enough, where we used Poincar\'e's inequality
in the second inequality, and the fact that $f$ is a polynomial of
fourth order in the third inequality.
\end{proof}
The contribution is organized as follows: Section \ref{chap:Fundamentals}
summarizes the needed mathematical tools, in particular
existence results for stationary Stokes equations with relevant boundary
conditions and we discuss a modified spectral estimate, which is key
for the proof of Theorem \ref{Main}.
Section \ref{chap:Main Proof} is then devoted to showing Theorem
\ref{Main}. First we will state a result on existence of approximate solutions, cf. Theorem
\ref{thm:Main-Apprx-Structure} below. This result
and all subsequently discussed properties of the approximate solutions
which are needed in this work, are shown in \cite{NSCH2}, see also \cite{ichPhD}. A key result
in Subsection \ref{subsec:The-Approximate-Solutions} is Lemma \ref{Wichtig},
which provides an estimate for the leading term of the error in the
velocity $\mathbf{v}_{A}^{\epsilon}-\mathbf{v}^{\epsilon}$. In order
to show this, a spectral decomposition of $c^{\epsilon}-c_{A}^{\epsilon}$
is needed. In Subsection \ref{sec:Auxiliary-Results}, we collect
many important statements which are essential to the proof of Theorem
\ref{Main}, many of which are concerned with dealing with the aforementioned
error in the velocity. These results enable us to effectively deal
with the problems arising due to the presence of the convective term
in the Cahn-Hilliard equation.

\section{Preliminaries\label{chap:Fundamentals}}

\subsection{Stationary Stokes Equation in One Phase\label{sec:Instationary-Stokes-Equation}}

We consider the one-phase stationary Stokes equation
\begin{align}
-\Delta\mathbf{v}+\nabla p & =\mathbf{f} &  & \text{in }\Omega,\label{eq:instokes1}\\
\operatorname{div}\mathbf{v} & =g &  & \text{in }\Omega,\label{eq:instokes2}\\
\left(-2D_{s}\mathbf{v}+p\mathbf{I}\right)\mathbf{n}_{\partial\Omega} & =\alpha_{0}\mathbf{v} &  & \text{on }\partial\Omega\label{eq:instokes3}
\end{align}
for given $\mathbf{f}\in V_{g}'(\Omega)$ and $g\in L^{2}(\Omega)$.
We denote $C_{\sigma}^{\infty}(\overline{\Omega}):=\left\{ \left.\mathbf{u}\in C^{\infty}\left(\overline{\Omega}\right)^{2}\right|\operatorname{div}\mathbf{u}=0\right\} $,
$H_{\sigma}^{1}(\Omega):=\overline{C_{\sigma}^{\infty}(\overline{\Omega})}^{H_{1}(\Omega)}$
and set
\begin{equation}
V_{g}(\Omega):=\begin{cases}
H_{\sigma}^{1}(\Omega) & \text{if }g=0,\\
H^{1}(\Omega)^{2} & \text{else,}
\end{cases}\label{eq:Vg}\qquad 
H_{g}(\Omega):=\begin{cases}
L_{\sigma}^{2}(\Omega) & \text{if }g\equiv 0,\\
L^{2}(\Omega)^{2} & \text{else}
\end{cases}
\end{equation}
and let $V_{g}'(\Omega)$ denote the dual space of $V_{g}(\Omega)$. 

We call $\mathbf{v}\in V_{g}(\Omega)$ a weak solution
of (\ref{eq:instokes1})\textendash (\ref{eq:instokes3}) if 
\begin{equation}
2\int_{\Omega}D_{s}\mathbf{v}:D_{s}\psi\d x+\alpha_{0}\int_{\partial\Omega}\mathbf{v}\cdot\psi\d\mathcal{H}^{1}(s)=\left\langle \mathbf{f},\psi\right\rangle _{V_{g}',V_{g}}\label{eq:instokesweak}
\end{equation}
holds for all $\psi\in C_{\sigma}^{\infty}\left(\overline{\Omega}\right)$
and
\begin{equation}
\operatorname{div}\mathbf{v}=g\text{ in }L^{2}(\Omega).\label{eq:instokesdiv}
\end{equation}
Note that in the case $g=0$ the condition (\ref{eq:instokesdiv})
is already included in the definition of the space $V_{0}$ and can
thus be omitted. Moreover, a classical solution to (\ref{eq:instokes1})\textendash (\ref{eq:instokes3})
is a weak solution.

\begin{thm}
\label{exinstokes} For each $g\in L^{2}(\Omega)$ and
$\mathbf{f}\in V_{g}'(\Omega)$ there is a unique weak
solution $\mathbf{v}\in V_{g}(\Omega)$ of \eqref{eq:instokes1}\textendash \eqref{eq:instokes3}.
Moreover there exists a constant $C\left(\Omega,\alpha_{0}\right)>0$,
which is independent of $\mathbf{f}$, such that 
\begin{equation}
\Vert \mathbf{v}\Vert _{H^{1}(\Omega)}\leq C(\Omega,\alpha_{0})\big(\Vert \mathbf{f}\Vert _{V_{g}'(\Omega)}+\Vert g\Vert _{L^{2}(\Omega)}\big).\label{eq:instokesab}
\end{equation}
\end{thm}

\begin{proof}
In the case $g=0$ the result is a direct
consequence of the Lax-Milgram Lemma. The case $g\neq 0$ can be easily reduced to the latter case by considering $\tilde{\mathbf{v}}= \mathbf{v}-\nabla q$, where $q\in H^2(\Omega)\cap H^1_0(\Omega)$ is such that $\Delta q= g$.
\end{proof}
The following corollary yields existence of a pressure term.
\begin{cor}
\label{cor:weakvp}Let $g\in L^{2}(\Omega)$ and $\mathbf{f}\in L^{2}(\Omega)^{2}$.
Then there is a unique $\left(\mathbf{v},p\right)\in V_{g}\times L^{2}(\Omega)$
of \eqref{eq:instokes1}\textendash\eqref{eq:instokes3} such
that 
\[
2\int_{\Omega}D_{s}\mathbf{v}:D_{s}\mathbf{\psi}-p\mathrm{div}\psi\d x+\alpha_{0}\int_{\partial\Omega}\mathbf{v}\cdot\psi\d\mathcal{H}^{1}(s)=\int_{\Omega}\mathbf{f}\cdot\psi\d x\quad \text{for all }\psi\in H^1(\Omega)
\]
and \eqref{eq:instokesdiv}
holds. Moreover, there is a constant $C>0$, independent of $\mathbf{v}$
and $p$, such that 
\[
\left\Vert \left(\mathbf{v},p\right)\right\Vert _{H^{1}(\Omega)\times L^{2}(\Omega)}\leq C\big(\Vert \mathbf{f}\Vert _{L^{2}(\Omega)}+\Vert g\Vert _{L^{2}(\Omega)}\big).
\]
\end{cor}
\begin{proof}
Let $\mathbf{v}$ be the weak solution to (\ref{eq:instokesweak})\textendash (\ref{eq:instokesdiv})
as given by Theorem \ref{exinstokes}. Elliptic theory implies that
$$
\Delta_{D}\colon \mathcal{D}(\Delta_D):= H^{2}(\Omega)\cap H_{0}^{1}(\Omega)\rightarrow L^{2}(\Omega)\colon u\mapsto \Delta u
$$
is bijective. Thus, the adjoint operator $\left(\Delta_{D}\right)'\colon L^{2}(\Omega)'\rightarrow(H^{2}(\Omega)\cap H_{0}^{1}(\Omega))'$
is also bijective. Using the continuity of the trace operator and H\"older's inequality
we find that the functional $F\colon \mathcal{D}(\Delta_D)\to \R$
\[
F(\varphi):=\int_{\Omega}\left(2D_{s}\mathbf{v}:D_{s}\left(\nabla\varphi\right)-\mathbf{f}\cdot\nabla\varphi\right)\d x+\alpha_{0}\int_{\partial\Omega}\mathbf{v}\cdot\nabla\varphi\d\mathcal{H}^{1}(s)\quad\forall\varphi\in \mathcal{D}(\Delta_D)
\]
is bounded and linear. Thus the Riesz representation theorem yields
the existence of $p\in L^{2}(\Omega)$ such that 
\begin{equation}
\left(p,\Delta\varphi\right)_{L^{2}}=\left\langle \Delta_{D}'\left(\left(p,.\right)_{L^{2}}\right),\varphi\right\rangle _{\mathcal{D}(\Delta_D)',\mathcal{D}(\Delta_D)}=F(\varphi)\label{eq:pdef}
\end{equation}
for all $\varphi\in \mathcal{D}(\Delta_D)$.
Since the operator $\left(\left(\Delta_{D}\right)'\right)^{-1}$ is bounded, we find
\begin{align*}
\left\Vert p\right\Vert _{L^{2}(\Omega)} & \leq C\left\Vert F\right\Vert _{(H^{2}(\Omega)\cap H_{0}^{1}(\Omega))'} \leq C\big(\Vert \mathbf{v}\Vert _{H^{1}(\Omega)}+\Vert \mathbf{f}\Vert _{L^{2}(\Omega)}\big)
 \leq C\big(\Vert \mathbf{f}\Vert _{L^{2}(\Omega)}+\Vert g\Vert _{L^{2}(\Omega)}\big),
\end{align*}
where we used (\ref{eq:instokesab}) in the last line.

Now let $\psi\in H^{1}(\Omega)^{2}$ be arbitrary and let $q\in \mathcal{D}(\Delta_D)$ be such that $\Delta q= \operatorname{div}\psi$.
Moreover set $\psi_{0}:=\psi-\nabla q$. Then $\operatorname{div}\psi_{0}=0$ and
\begin{align*}
  &\int_{\Omega}2D_{s}\mathbf{v}:D_{s}\mathbf{\psi}-p\mathrm{div}\psi\d x+\alpha_{0}\int_{\partial\Omega}\mathbf{v}\cdot\psi\d\mathcal{H}^{1}(s) \\
  &=\int_{\Omega}\mathbf{f}\cdot\psi_{0}\d x+\int_{\Omega}(2D_{s}\mathbf{v}:D_{s}\left(\nabla q\right)-p\Delta q)\d x +\alpha_{0}\int_{\partial\Omega}\mathbf{v}\cdot\nabla q\d\mathcal{H}^{1}(s) =\int_{\Omega}\mathbf{f}\cdot\psi\d x, {}
\end{align*}
where we used (\ref{eq:instokesweak}) and (\ref{eq:pdef}). As $\psi\in H^{1}(\Omega)^{2}$ was arbitrary,
this yields the claim.
\end{proof}
\begin{thm}[Existence of Strong Solutions]
\label{thm:strongexstokes}~\\
Let $g\equiv0$ and $\mathbf{f}\in L^{2}(\Omega)^{2}$.
Then there exists a unique solution $\left(\mathbf{v},p\right)\in H^{2}(\Omega)^{2}\times H^{1}(\Omega)$
to \eqref{eq:instokes1}\textendash \eqref{eq:instokes3}, which satisfies
the estimate
\[
\left\Vert \mathbf{v}\right\Vert _{H^{2}(\Omega)}+\left\Vert p\right\Vert _{H^{1}(\Omega)}\leq C\left\Vert \mathbf{f}\right\Vert _{L^{2}(\Omega)}.
\]
 Moreover, if $\mathbf{f}$ is smooth, then $\mathbf{v}$ and $p$
are smooth as well.
\end{thm}
\begin{proof}
For $q\in\left(1,\infty\right)$, Theorem 3.1 in \cite{ShimizuStokes}
 implies
that there is $\lambda>0$ such that for every $\mathbf{g}\in L^{q}(\Omega)^{2}$
and $\mathbf{a}\in W_{q}^{1}(\Omega)^{2}$ the problem
\begin{align}
\lambda\mathbf{u}-\Delta\mathbf{u}+\nabla q & =\mathbf{g} &  & \text{in }\Omega,\nonumber \\
\operatorname{div}\mathbf{u} & =0 &  & \text{in }\Omega,\nonumber \\
\left(-2D_{s}\mathbf{u}+q\mathbf{I}\right)\mathbf{n}_{\partial\Omega} & =\mathbf{a}|_{\partial\Omega} &  & \text{on }\partial\Omega\label{eq:inhomogenneumann}
\end{align}
admits for a unique solution $\left(\mathbf{u},q\right)\in W_{q}^{2}(\Omega)^{2}\times W_{q}^{1}(\Omega)$.
Additionally, the estimate 
\begin{equation}
\left\Vert \mathbf{u}\right\Vert _{W_{q}^{2}(\Omega)}+\left\Vert q\right\Vert _{W_{q}^{1}(\Omega)}\leq C\big(\Vert \mathbf{g}\Vert _{L^{q}(\Omega)}+\Vert \mathbf{a}\Vert _{W_{q}^{1}(\Omega)}\big)\label{eq:absch=0000E4tzungstrongstokes}
\end{equation}
holds. Considering a weak solution $\left(\mathbf{v},p\right)\in V_{0}\times L^{2}(\Omega)$
of (\ref{eq:instokes1})\textendash (\ref{eq:instokes3}) as given
in Corollary \ref{cor:weakvp} and defining $\mathbf{g}:=\mathbf{f}+\lambda\mathbf{v}\in L^{2}(\Omega)^{2}$
and $\mathbf{a}:=\alpha_{0}\mathbf{v}\in H^{1}(\Omega)^{2}$,
we now introduce $\left(\mathbf{u},q\right)\in H^{2}(\Omega)\times H^{1}(\Omega)$
as the strong solution to (\ref{eq:inhomogenneumann}) regarding these
data. Writing $\mathbf{w}:=\mathbf{u}-\mathbf{v}$ and $r:=q-p$ we
easily find that $\left(\mathbf{w},r\right)\in H^1(\Omega)^2\times L^2(\Omega)$ is a weak solution to
\begin{align*}
\lambda\mathbf{w}-\Delta\mathbf{w}+\nabla r & =0 &  & \text{in }\Omega,\\
\operatorname{div}\mathbf{w} & =0 &  & \text{in }\Omega,\\
\left(-2D_{s}\mathbf{w}+r\mathbf{I}\right)\mathbf{n}_{\partial\Omega} & =0 &  & \text{on }\partial\Omega.
\end{align*}
Testing with $\psi=\mathbf{w}$
we immediately find that $\mathbf{w}\equiv0$ a.e. and thus $\mathbf{u}=\mathbf{v}$,
in particular $\mathbf{v}\in H^{2}(\Omega)^{2}$. Furthermore,
$\mathbf{w}=0$ implies $\nabla r=0$ in $\Omega$ and $r=0$ on $\partial\Omega$,
so that we can conclude $r\equiv0$ a.e. in $\Omega$ leading to $p=q$
and $p\in H^{1}(\Omega)$. The estimate follows from (\ref{eq:absch=0000E4tzungstrongstokes})
and (\ref{eq:instokesab}). 
For higher regularity one may use results on existence of solutions with higher regularity, e.g.\ due to Grubb and Solonnikov~\cite{solonnikovstokesstat} in a similar manner to obtain smoothness of the solution for smooth boundaries and smooth data.
\end{proof}
\begin{lem}
\label{LpLqstokes}Let $g\equiv0$ and $\mathbf{f}\in V_{0}'$, and
let $\mathbf{v}\in H_{\sigma}^{1}(\Omega)$ be the weak
solution to \eqref{eq:instokes1}\textendash \eqref{eq:instokes3}.
Then for all $q'\in\left(1,2\right)$
\begin{align*}
\left\Vert \mathbf{v}\right\Vert _{L^{q'}(\Omega)} & \leq C_{q}\sup_{\psi\in W_{q}^{2}(\Omega)^{2},\psi\neq0}\frac{\left|\mathbf{f}\left(\psi\right)\right|}{\left\Vert \psi\right\Vert _{W_{q}^{2}(\Omega)}},
\end{align*}
where $\frac{1}{q}+\frac{1}{q'}=1$ and $C_{q}>0$ is independent
of $\mathbf{v}$ and $f$.
\end{lem}

\begin{proof}
For this we introduce $T(\mathbf{u},p):=-2D_{s}\mathbf{u}+p\mathbf{I}$
for $\mathbf{u}\in W_{q}^{1}(\Omega)$, $p\in L^{2}(\Omega)$
and set
\[
D(A_{S})=\left\{ \left.\mathbf{u}\in W_{q}^{2}(\Omega)\right|\operatorname{div}\mathbf{u}=0,\exists p\in W_{q}^{1}(\Omega):T\left(\mathbf{u},p\right)\mathbf{n}|_{\partial\Omega}=\alpha_{0}\mathbf{u}|_{\partial\Omega}\right\} .
\]
We define the operator 
\[
A_{S}\colon D(A_{S})\subset L_{\sigma}^{q}(\Omega)\rightarrow L_{\sigma}^{q}(\Omega),\,\mathbf{u}\mapsto P_{\sigma}\left(-\Delta\mathbf{u}+\nabla p\right),
\]
for $p$ as in the definition of $D(A_{S})$ and where
$P_{\sigma}$ denotes the Helmholtz projection given by 
\[
P_{\sigma}:L^{q}(\Omega)^{2}\rightarrow L_{\sigma}^{q}(\Omega),\,\psi\mapsto P_{\sigma}\left(\psi\right)=\psi-\nabla r,
\]
where $r\in W_{q,0}^{1}(\Omega)$ is the unique weak solution
to
\begin{align*}
\Delta r & =\operatorname{div}\psi &  & \text{in }\Omega,\\
r & =0 &  & \text{on }\partial\Omega.
\end{align*}
One can verify in a straight-forward manner that $A_S$ is well defined.
Moreover, 
\begin{align}
\int_{\Omega}\left(A_{S}\mathbf{u}\right)\cdot\mathbf{u}\d x 
 & =\int_{\Omega}2\left|D_{s}\mathbf{u}\right|^{2}\d x+\alpha_{0}\int_{\partial\Omega}\mathbf{u}^{2}\d\mathcal{H}^{1}(s) \geq C\left\Vert \mathbf{u}\right\Vert _{L^{2}(\Omega)}^{2}\label{eq:lowerspectralbound}
\end{align}
for some $C>0$ and $\mathbf{u}\in\mathcal{D}\left(A_{S}\right)$,
where we used \cite[Corollary 5.8]{korn} in the last line. This immediately
shows the injectivity of $A_{S}$. Concerning surjectivity, let $\mathbf{\tilde{f}}\in L_{\sigma}^{q}(\Omega)$.
As $q>2$, Theorem \ref{thm:strongexstokes} implies that there is
a unique strong solution $\left(\tilde{\mathbf{v}},p\right)\in H^{2}(\Omega)\times H^{1}(\Omega)$
to (\ref{eq:instokes1})\textendash (\ref{eq:instokes3}) (with $\mathbf{f}$
replaced by $\tilde{\mathbf{f}}$ and $g\equiv0$). Choosing $\lambda>0$
as in the proof of Theorem \ref{thm:strongexstokes}, we find that
\textbf{$\mathbf{g}:=\tilde{\mathbf{f}}+\lambda\mathbf{\tilde{v}}$}
and $\mathbf{a}:=\alpha_{0}\mathbf{\tilde{v}}$ satisfy $\mathbf{g}\in L^{q}(\Omega)$
and $\mathbf{a}\in W_{q}^{1}(\Omega)$ as a consequence
of the Sobolev embedding theorem. Thus, Theorem~3.1.\ in \cite{ShimizuStokes}
implies the existence of a unique solution $\left(\mathbf{u},r\right)\in W_{q}^{2}(\Omega)\times W_{q}^{1}(\Omega)$
to (\ref{eq:inhomogenneumann}) and an analogous argumentation as
in the proof of Theorem \ref{thm:strongexstokes} leads to  $\mathbf{\tilde{v}}=\mathbf{u}$ and $p=r$ along with the
estimate
\begin{equation}
\left\Vert \mathbf{\tilde{v}}\right\Vert _{W_{q}^{2}(\Omega)}+\left\Vert p\right\Vert _{W_{q}^{1}(\Omega)}\leq C\Vert \tilde{\mathbf{f}}\Vert _{L^{q}(\Omega)}.\label{eq:inversestetig}
\end{equation}
 In particular, $T\left(\tilde{\mathbf{v}},p\right)\mathbf{n}|_{\partial\Omega}=\alpha_{0}\mathbf{\tilde{v}}|_{\partial\Omega}$
is satisfied. So $\tilde{\mathbf{v}}\in\mathcal{D}\left(A_{s}\right)$
and, since $-\Delta\tilde{\mathbf{v}}+\nabla p=\tilde{\mathbf{f}}$
holds in $L^{q}(\Omega)$, we have $A_{s}\left(\tilde{\mathbf{v}}\right)=\tilde{\mathbf{f}}$.
In fact, this not only implies surjectivity, but also the existence
of a bounded inverse $A_{S}^{-1}$ as a result of (\ref{eq:inversestetig}).
Consequently, $\left(\mathcal{D}\left(A_{s}\right),\left\Vert .\right\Vert _{A_{s}}\right)$
is a Banach space, where $\left\Vert .\right\Vert _{A_{s}}$ denotes
the graph norm. All these considerations result in the fact that the
adjoint $A_{S}':\left(L_{\sigma}^{q}(\Omega)\right)'\rightarrow\left(\mathcal{D}\left(A_{S}\right)\right)'$
is an invertible and bounded operator.

Let now $\mathbf{v}\in H_{\sigma}^{1}(\Omega)$ be the
given weak solution to (\ref{eq:instokes1})\textendash (\ref{eq:instokes3})
and fix $q>2$. Then $\mathbf{v}\in L_{\sigma}^{q'}(\Omega)\cong\left(L_{\sigma}^{q}(\Omega)\right)'$
and we have for $\psi\in\mathcal{D}\left(A_{S}\right)$
\begin{align*}
\left\langle A_{S}'\mathbf{v},\psi\right\rangle _{\left(\mathcal{D}\left(A_{S}\right)\right)',\mathcal{D}\left(A_{S}\right)} 
 & =2\int_{\Omega}D_{s}\mathbf{v}:D_{s}\psi\d x+\alpha_{0}\int_{\partial\Omega}\psi\cdot\mathbf{v}\d x =\left\langle \mathbf{f},\psi\right\rangle _{\left(\mathcal{D}\left(A_{S}\right)\right)',\mathcal{D}\left(A_{S}\right)}.
\end{align*}
 As a result $A_{S}'\mathbf{v}=\mathbf{f}$ in $\left(\mathcal{D}\left(A_{S}\right)\right)'$
and thus $\mathbf{v}=\left(A_{S}'\right)^{-1}\mathbf{f}$ in $\left(L_{\sigma}^{q}(\Omega)\right)'$
which enables us to estimate
\[
\left\Vert \mathbf{v}\right\Vert _{L^{q'}(\Omega)}=\big\Vert (A_{S}')^{-1}\mathbf{f}\big\Vert _{\left(L_{\sigma}^{q}(\Omega)\right)'}\leq C\left\Vert \mathbf{f}\right\Vert _{\left(\mathcal{D}\left(A_{S}\right)\right)'}\leq C\left\Vert \mathbf{f}\right\Vert _{\left(W_{q}^{2}(\Omega)\right)'}.
\]
\end{proof}

\subsection{Differential-Geometric Background\label{sec:Diffgeo}}

We use a similar notation as in \cite{nsac}. We parameterize the curves $\left(\Gamma_{t}\right)_{t\in\left[0,T_{0}\right]}$
by choosing a family of smooth diffeomorphisms 
\begin{equation}
X_{0}\colon \mathbb{T}^{1}\times\left[0,T_{0}\right]\rightarrow\Omega\label{eq:X0def}
\end{equation}
such that $\partial_{s}X_{0}\left(s,t\right)\neq0$ for all $s\in\mathbb{T}^{1}$,
$t\in\left[0,T_{0}\right]$. In particular $\bigcup_{t\in\left[0,T_{0}\right]}X_{0}(\mathbb{T}^{1}\times\{ t\})\times\{ t\} =\Gamma$.
Moreover, we define the tangent and normal vectors on $\Gamma_{t}$
at $X_{0}(s,t)$ as 
\begin{equation}
\boldsymbol{\tau}(s,t):=\frac{\partial_{s}X_{0}(s,t)}{\left|\partial_{s}X_{0}(s,t)\right|}\text{ and }\mathbf{n}(s,t):=\left(\begin{array}{cc}
0 & -1\\
1 & 0
\end{array}\right)\boldsymbol{\tau}(s,t)\label{eq:ntau}
\end{equation}
for all $(s,t)\in\mathbb{T}^{1}\times\left[0,T_{0}\right]$.
We choose $X_{0}$ (and thereby the orientation of $\Gamma_{t}$)
such that $\mathbf{n}(.,t)$ is the exterior normal with
respect to $\Omega^{-}(t)$. Thus, for a point $p\in\Gamma_{t}$
with $p=X_{0}(s,t)$ it holds $\mathbf{n}_{\Gamma_{t}}(p)=\mathbf{n}(s,t)$

Furthermore, we define $V(s,t):=V_{\Gamma_{t}}(X_{0}(s,t))$
and $H(s,t):=H_{\Gamma_{t}}(X_{0}(s,t))$
and note that $V(s,t)=\partial_{t}X_{0}(s,t)\cdot\mathbf{n}(s,t)$
for all $\left(s,t\right)\in\mathbb{T}^{1}\times\left[0,T_{0}\right]$
by definition of the normal velocity. We also introduce the pull-back
and write for a function $\mathbf{v}\colon \Gamma\rightarrow\mathbb{R}^{d}$,
$d\in\mathbb{N}$ 
\begin{equation}
\left(X_{0}^{*}\mathbf{v}\right)(s,t):=\mathbf{v}(X_{0}(s,t),t)\quad \text{for all }(s,t)\in\mathbb{T}^{1}\times [0,T_{0}].\label{eq:X0star}
\end{equation}
On the other hand, we define for a function $h\colon\mathbb{T}^{1}\times [0,T_{0}]$
\begin{equation}
\big(X_{0}^{*,-1}h\big)(p):=h(X_{0}^{-1}(p))\label{eq:X0-1star}\quad \text{for all }p\in\Gamma_{t}, t\in [0,T_{0}].
\end{equation}

Choosing $\delta>0$ small enough, the orthogonal projection $\operatorname{Pr}_{\Gamma_{t}}\colon \Gamma_{t}(3\delta)\rightarrow\Gamma_{t}$
is well defined and smooth for all $t\in\left[0,T_{0}\right]$ and the mapping 
\[
\phi_{t}\colon \Gamma_{t}(3\delta)\rightarrow (-3\delta,3\delta)\times\Gamma_{t},\,x\mapsto\left(d_{\Gamma}(x,t),\operatorname{Pr}_{\Gamma_{t}}(x)\right)
\]
is a diffeomorphism. Its inverse is given by $\phi_{t}^{-1}(r,p)=p+r\mathbf{n}_{\Gamma_{t}}(p)$.
Although $\operatorname{Pr}_{\Gamma_{t}}$ and $\phi_{t}$ are well defined in $\Gamma_{t}(3\delta)$,
almost all computations later on are performed in $\Gamma_{t}(2\delta)$,
which is why, for the sake of readability, we work on $\Gamma_{t}(2\delta)$
in the following.

Combining $\phi_{t}^{-1}$ and $X_{0}$ we may define a diffeomorphism
\begin{align}
X\colon(-2\delta,2\delta)\times\mathbb{T}^{1}\times\left[0,T_{0}\right] & \rightarrow\Gamma(2\delta),\nonumber \\
\left(r,s,t\right) & \mapsto\left(\phi_{t}^{-1}\left(r,X_{0}(s,t)\right),t\right)=\left(X_{0}(s,t)+r\mathbf{n}(s,t),t\right)\label{eq:diffeo}
\end{align}
with inverse given by 
\begin{equation}
X^{-1}\colon\Gamma(2\delta)\rightarrow\left(-2\delta,2\delta\right)\times\mathbb{T}^{1}\times\left[0,T_{0}\right],\,(x,t)\mapsto\left(d_{\Gamma}(x,t),S(x,t),t\right),\label{eq:diffeoinv}
\end{equation}
where we define 
\begin{equation}
S(x,t):=\left(X_{0}^{-1}\left(\operatorname{Pr}_{\Gamma_{t}}(x)\right)\right)_{1}\label{eq:Sfett}
\end{equation}
for $(x,t)\in\Gamma(2\delta)$ and where $\left(.\right)_{1}$
signifies that we take the first component. In particular it holds
$S(x,t)=S(\operatorname{Pr}_{\Gamma_{t}}(x),t)$.
In the following we will write $\mathbf{n}(x,t):=\mathbf{n}\left(S(x,t),t\right)$
and $\boldsymbol{\tau}(x,t):=\boldsymbol{\tau}\left(S(x,t),t\right)$
for $(x,t)\in\Gamma(3\delta)$.
\begin{prop}
\label{Diff-Prop} For every $t\in\left[0,T_{0}\right]$, $x\in\Gamma_{t}(2\delta)$, $s\in \mathbb{T}^1$, $r\in (-2\delta,2\delta)$ it holds
\begin{alignat*}{2}
\left|\nabla d_{\Gamma}(x,t)\right|&=1,&\quad   
\Delta d_{\Gamma}\left(X_{0}(s,t),t\right)&=-H(s,t),\\
-\partial_{t}d_{\Gamma}\left(X(r,s,t)\right)&=V(s,t),\quad&
\nabla d_{\Gamma}\left(X(r,s,t)\right)&=\mathbf{n}(s,t),\\
\nabla S(x,t)\cdot\nabla d_{\Gamma}(x,t)&=0.
\end{alignat*}
\end{prop}

\begin{proof}
  We refer to \cite[Chapter 2.3]{pruess} and \cite[Chapter 4.1]{chenAC} for the proofs.
\end{proof}
For a function $\phi\colon\Gamma(2\delta)\rightarrow\mathbb{R}$
we define $\tilde{\phi}(r,s,t):=\phi\left(X(r,s,t)\right)$
and often write $\phi(r,s,t)$ instead of $\tilde{\phi}(r,s,t)$.
In the case that $\phi$ is twice continuously differentiable,
we introduce the notations 
\begin{align}
\partial_{t}^{\Gamma}\tilde{\phi}(r,s,t) & :=\left(\partial_{t}+\partial_{t}S\left(X(r,s,t)\right)\partial_{s}\right)\tilde{\phi}(r,s,t),\nonumber \\
\nabla^{\Gamma}\tilde{\phi}(r,s,t) & :=\nabla S\left(X(r,s,t)\right)\partial_{s}\tilde{\phi}(r,s,t),\nonumber \\
\Delta^{\Gamma}\tilde{\phi}(r,s,t) & :=\left(\Delta S\left(X(r,s,t)\right)\partial_{s}+\left(\nabla S\cdot\nabla S\right)\left(X(r,s,t)\right)\partial_{ss}\right)\tilde{\phi}(r,s,t).\label{eq:surf-diff}
\end{align}
Similarly, if $\mathbf{v}\colon\Gamma(2\delta)\rightarrow\mathbb{R}^{2}$
is continuously differentiable, we will also write $\tilde{\mathbf{v}}\left(r,st\right):=\mathbf{v}\left(X(r,s,t)\right)$
and introduce
\begin{equation}
\mathrm{div}^{\Gamma}\tilde{\mathbf{v}}(r,s,t)=\nabla S\left(X(r,s,t)\right)\cdot\partial_{s}\tilde{\mathbf{v}}(r,s,t).\label{eq:divop}
\end{equation}
For later use we introduce
\begin{align*}
\nabla^{\Gamma}\phi(x,t)&:=\nabla S(x,t)\partial_{s}\tilde{\phi}\left(d_{\Gamma}(x,t),S(x,t),t\right)\quad
\text{and}\\ 
\operatorname{div}^{\Gamma}\mathbf{v}(x,t)&:=\nabla S(x,t)\partial_{s}\tilde{\mathbf{v}}\left(d_{\Gamma}(x,t),S(x,t),t\right)  
\end{align*}
for $(x,t)\in\Gamma(2\delta)$.
With these notations we have the decompositions 
\begin{align}
\nabla\phi(x,t) & =\partial_{\mathbf{n}}\phi(x,t)\mathbf{n}+\nabla^{\Gamma}\phi(x,t),\label{eq:surfgraddecomp}\\
\operatorname{div}\mathbf{v}(x,t) & =\partial_{\mathbf{n}}\mathbf{v}(x,t)\cdot\mathbf{n}+\operatorname{div}^{\Gamma}\mathbf{v}(x,t)\label{eq:surfdivdecomp}
\end{align}
for all $(x,t)\in\Gamma(2\delta)$, as 
\[
\frac{d}{dr}\left(\phi\circ X\right)\big|_{(r,s,t)=\left(d_{\Gamma}(x,t),S(x,t),t\right)}=\partial_{\mathbf{n}}\phi(x,t).
\]
\begin{rem}
\label{hnotations} If $h\colon \mathbb{T}^{1}\times\left[0,T_{0}\right]\rightarrow\mathbb{R}$
is a function that is independent of $r\in\left(-2\delta,2\delta\right)$,
the functions $\partial_{t}^{\Gamma}h,\nabla^{\Gamma}h$ and $\Delta^{\Gamma}h$
will nevertheless depend on $r$ via the derivatives of $S$. To connect
the presented concepts with the classical surface operators we introduce
the following notations:
\begin{align*}
D_{t,\Gamma}h(s,t) & :=\partial_{t}^{\Gamma}h(0,s,t),\quad
\nabla_{\Gamma}h(s,t)  :=\nabla^{\Gamma}h(0,s,t),\quad
\Delta_{\Gamma}h(s,t)  :=\Delta^{\Gamma}h(0,s,t).
\end{align*}
Later in this work, we will often consider a concatenation $h\left(S(x,t),t\right)$
and thus will write for simplicity
\begin{align}
\partial_{t}^{\Gamma}h(x,t) & :=\left(\partial_{t}+\partial_{t}S(x,t)\partial_{s}\right)h(S(x,t),t),\nonumber \\
\nabla^{\Gamma}h(x,t) & :=\left(\nabla S(x,t)\partial_{s}\right)h(S(x,t),t),\nonumber \\
\Delta^{\Gamma}h(x,t) & :=\left(\Delta S(x,t)\partial_{s}+\nabla S(x,t)\cdot\nabla S(x,t)\partial_{ss}\right)h(S(x,t),t)\label{eq:hsurf}
\end{align}
for $(x,t)\in\Gamma(2\delta)$. As a consequence we obtain the identity
\begin{equation}
\partial_{t}^{\Gamma}h(x,t)=X_{0}^{*}\left(\partial_{t}^{\Gamma}h\right)(s,t)=\partial_{t}^{\Gamma}h(0,s,t)=D_{t,\Gamma}h(s,t)\label{eq:identity}
\end{equation}
for $(s,t)\in\mathbb{T}^{1}\times\left[0,T_{0}\right]$
and $\left(X_{0}(s,t),t\right)=(x,t)\in\Gamma$.
This might seem cumbersome but turns out to be convenient throughout
this work.
\end{rem}

In later parts of this article, we will introduce stretched coordinates
of the form 
\begin{equation}
\rho^{\epsilon}(x,t)=\frac{d_{\Gamma}(x,t)-\epsilon h\left(S(x,t),t\right)}{\epsilon}\label{eq:diffpartrho}
\end{equation}
for $(x,t)\in\Gamma(2\delta)$, $\epsilon\in\left(0,1\right)$
and for some smooth function $h\colon \mathbb{T}^{1}\times\left[0,T_{0}\right]\rightarrow\mathbb{R}$
(which will later on also depend on $\epsilon$). Writing $\rho=\rho^{\epsilon}$,
the relation between the regular and the stretched variables can be
expressed as
\begin{equation}
\hat{X}(\rho,s,t):=X(\epsilon(\rho+h(s,t)),s,t)=\left(X_{0}(s,t)+\epsilon(\rho+h(s,t))\mathbf{n}(s,t),t\right).\label{eq:Xhat}
\end{equation}
\begin{lem}
\label{lem:surfright} Let $\phi\colon\mathbb{R}\times\Gamma(2\delta)\rightarrow\mathbb{R}$
be twice continuously differentiable and let $\rho$ be given
as in \eqref{eq:diffpartrho}. Then the following formulas hold for
$(x,t)\in\Gamma(2\delta)$ and $\epsilon\in\left(0,1\right)$
\begin{align*}
\partial_{t}\left(\phi\left(\rho(x,t),x,t\right)\right) & =\left(-\tfrac1{\epsilon}V(s,t)-\partial_{t}^{\Gamma}h(x,t)\right)\partial_{\rho}\phi\left(\rho,x,t\right)+\partial_{t}\phi\left(\rho,x,t\right),\\
\nabla\left(\phi\left(\rho(x,t),x,t\right)\right) & =\left(\tfrac1{\epsilon}\mathbf{n}(s,t)-\nabla^{\Gamma}h(x,t)\right)\partial_{\rho}\phi\left(\rho,x,t\right)+\nabla_{x}\phi\left(\rho,x,t\right),\\
\Delta\left(\phi\left(\rho(x,t),x,t\right)\right) & =\big(\tfrac1{\epsilon^2}+|\nabla^{\Gamma}h(x,t)|^{2}\big)\partial_{\rho\rho}\phi\left(\rho,x,t\right)\\
 & \quad+\left(\epsilon^{-1}\Delta d_{\Gamma}(x,t)-\Delta^{\Gamma}h(x,t)\right)\partial_{\rho}\phi\left(\rho,x,t\right)\\
 & \quad+2\left(\epsilon^{-1}\mathbf{n}(s,t)-\nabla^{\Gamma}h(x,t)\right)\cdot\nabla_{x}\partial_{\rho}\phi\left(\rho,x,t\right)+\Delta_{x}\phi\left(\rho(x,t),x,t\right),
\end{align*}
where $s=S(x,t)$ and $\rho=\rho(x,t)$.
Here $\nabla_{x}$ and $\Delta_{x}$ operate solely on the $x$-variable
of $\phi$.
\end{lem}

\begin{proof}
This follows from the chain rule, Proposition \ref{Diff-Prop} and
the notations introduced in Remark \ref{hnotations}.
\end{proof}
By (\ref{eq:surfgraddecomp}) and (\ref{eq:surfdivdecomp}) we have
\begin{align}
\nabla^{\Gamma}u(x,t)=\left(\mathbf{I}-\mathbf{n}\left(S(x,t),t\right)\otimes\mathbf{n}\left(S(x,t),t\right)\right)\nabla u(x,t)\label{eq:surfdiffallg}
\quad
\text{and}\\ 
\operatorname{div}^{\Gamma}\mathbf{v}(x,t)=\left(\mathbf{I}-\mathbf{n}\left(S(x,t),t\right)\otimes\mathbf{n}\left(S(x,t),t\right)\right):\nabla\mathbf{v}(x,t)\label{eq:surfdivallg}
\end{align}
for suitable $u\colon \Gamma(2\delta)\rightarrow\mathbb{R}$,
$\mathbf{v}\colon \Gamma(2\delta)\rightarrow\mathbb{R}^{2}$.
A consequence is:
\begin{lem}
\label{lem:divlm}Let $t\in\left[0,T_{0}\right]$ and $\mathbf{v}\in H^{1}\left(\Gamma_{t}(\delta)\right)^{2}$,
$u\in H^{1}\left(\Gamma_{t}(\delta)\right)$. Then it holds
\[
\int_{\Gamma_{t}(\delta)}u\mathrm{div}^{\Gamma}\mathbf{v}\d x=-\int_{\Gamma_{t}(\delta)}\nabla^{\Gamma}u\cdot\mathbf{v}\d x-\int_{\Gamma_{t}(\delta)}u\mathbf{v}\cdot\mathbf{n}\kappa\d x+\int_{\partial\left(\Gamma_{t}(\delta)\right)}u\left(\left(\mathbf{I}-\mathbf{n}\otimes\mathbf{n}\right)\cdot\mathbf{v}\right)\cdot\nu\d\mathcal{H}^{1}(s),
\]
where $\kappa:=-\mathrm{div}\left(\mathbf{n}\left(S(x,t),t\right)\right)$
and $\nu(s)$ is the outer unit normal to $\Gamma_{t}(\delta)$
for $s\in\partial\left(\Gamma_{t}(\delta)\right)$.
\end{lem}

\begin{proof}
This is a consequence (\ref{eq:surfdiffallg}), (\ref{eq:surfdivallg}),
and the divergence theorem.
\end{proof}
For later use we define 
\begin{equation}
\left[\partial_{\mathbf{n}},\nabla^{\Gamma}\right]u:=\partial_{\mathbf{n}}\left(\left(\mathbf{I}-\mathbf{n}\otimes\mathbf{n}\right)\nabla u\right)-\left(\mathbf{I}-\mathbf{n}\otimes\mathbf{n}\right)\nabla\left(\partial_{\mathbf{n}}u\right)\label{eq:comuutatordef}
\end{equation}
and compute 
\begin{equation}
\left[\partial_{\mathbf{n}},\nabla^{\Gamma}\right]u=-\nabla S\left(\partial_{s}\mathbf{n}\cdot\nabla^{\Gamma}u\right).\label{eq:commutator}
\end{equation}

\subsection{Remainder Terms \label{sec:Remainder}}

We introduce the following function spaces.
For $t\in\left[0,T_{0}\right]$ and $1\leq p<\infty$ we define
\[
L^{p,\infty}\left(\Gamma_{t}(2\delta)\right):=\left\{ \left.f:\Gamma_{t}(2\delta)\rightarrow\mathbb{R}\;\text{measurable}\right|\left\Vert f\right\Vert _{L^{p,\infty}\left(\Gamma_{t}(2\delta)\right)}<\infty\right\} ,
\]
where 
\[
\left\Vert f\right\Vert _{L^{p,\infty}\left(\Gamma_{t}(2\delta)\right)}:=\left(\int_{\mathbb{T}^{1}}\text{esssup}_{\left|r\right|\leq2\delta}\left|f\left(\left(X(r,s,t)\right)_{1}\right)\right|^{p}\d s\right)^{\frac{1}{p}}.
\]
Here $X_{1}(r,s,t):=X_{0}(s,t)+r\mathbf{n}(s,t)$
denotes the first component of $X$. The following embedding was already remarked in \cite[Subsection~2.5]{nsac}.
\begin{lem}
\label{L4inf} We have
$
H^{1}\left(\Gamma_{t}(2\delta)\right)\hookrightarrow L^{4,\infty}\left(\Gamma_{t}(2\delta)\right)
$
with operator norm uniformly bounded with respect to $t\in [0,T_0]$.
\end{lem}
\begin{proof}
This is a consequence of the Gagliardo-Nirenberg interpolation and
the fact that $\Gamma_{t}$ is one-dimensional.
\end{proof}

For $T\in\left[0,T_{0}\right]$, $1\leq p,q<\infty$ and $\alpha\in\left(0,3\delta\right)$ we set 
\begin{align*}
L^{q}\left(0,T;L^{p}\left(\Gamma_{t}(\alpha)\right)\right)&:=\left\{ \left.f\colon \Gamma\left(\alpha,T\right)\rightarrow\mathbb{R}\;\text{measurable}\right|\left\Vert f\right\Vert _{L^{q}\left(0,T;L^{p}\left(\Gamma_{t}(\alpha)\right)\right)}<\infty\right\} ,\\
\left\Vert f\right\Vert _{L^{q}\left(0,T;L^{p}\left(\Gamma_{t}(\alpha)\right)\right)}&:=\left(\int_{0}^{T}\left(\int_{\Gamma_{t}(\alpha)}\left|f(x,t)\right|^{p}\d x\right)^{\frac{q}{p}}\d t\right)^{\frac{1}{q}}.
\end{align*}
In a similar way, we define $L^{q}\left(0,T;L^{p}\left(\Omega\backslash\Gamma_{t}(\alpha)\right)\right)$ and $L^{q}\left(0,T;L^{p}\left(\Omega^\pm(t))\right)\right)$
and the corresponding norms. Moreover, for $m\in\mathbb{N}_0$ we denote for $U(t)= \Omega^\pm (t)$ or $U(t)= \Gamma_t(\alpha)$
\begin{align*}
  L^p(0,T;H^m(U(t)))&:=\{f\in L^p(0,T;L^2(\Omega^\pm(t))): \partial_x^\alpha f \in L^p(0,T;L^2(U(t)))\forall |\alpha|\leq m  \},\\
  \|f\|_{L^p(0,T;H^m(U(t)))} &:= \sum_{|\alpha|\leq m} \|\partial_x^\alpha f\|_{L^p(0,T;L^2(U(t)))}.
\end{align*}

For future use, we introduce a concept of remainder terms, similar
to \cite[Definition~2.5]{nsac}.
\begin{defn}
\label{def:Ralphadef}Let $n\in\mathbb{N}$, $\epsilon_{0}>0$. For
$\alpha>0$ let $\ra$ denote the vector space of all families $\left(\hat{r}_{\epsilon}\right)_{\epsilon\in\left(0,\epsilon_{0}\right)}$
of continuous functions $\hat{r}_{\epsilon}:\mathbb{R}\times\Gamma(2\delta)\rightarrow\mathbb{R}^{n}$
which satisfy
\[
\left|\hat{r}_{\epsilon}\left(\rho,x,t\right)\right|\leq Ce^{-\alpha\left|\rho\right|}\text{ for all }\rho\in\mathbb{R},(x,t)\in\Gamma(2\delta),\epsilon\in\left(0,1\right).
\]
Moreover, let $\mathcal{R}_{\alpha}^{0}$ be the subspace of all $\left(\hat{r}_{\epsilon}\right)_{\epsilon\in\left(0,\epsilon_{0}\right)}\in\mathcal{R}_{\alpha}$
such that 
\[
\hat{r}_{\epsilon}\left(\rho,x,t\right)=0\text{ for all }\rho\in\mathbb{R},(x,t)\in\Gamma.
\]
\end{defn}

\subsection{Spectral Theory\label{chap:Spectral-Theory}}

The results in this chapter are adapted from \cite{chen}. For detailed
proofs concerning the changed stretched variable see \cite[Chapter~3]{ichPhD}. Moreover, we define 
\begin{equation}
J(r,s,t):=\det\left(D_{\left(r,s\right)}X(r,s,t)\right)\label{eq:jacob}
\end{equation}
The statements in this section are made under the following assumptions:
\begin{assumption}
\label{assu:Spektral}Let $\epsilon\in\left(0,\epsilon_{0}\right)$,
$T\in\left(0,T_{0}\right]$ and $\xi$ be a cut-off function satisfying
\eqref{eq:cut-off}. We assume that $c_{A}^{\epsilon}:\Omega_{T}\rightarrow\mathbb{R}$
is a smooth function, which has the structure 
\begin{align}
  c_{A}^{\epsilon}(x,t) & =\xi(d_{\Gamma}(x,t))\left(\theta_{0}(\rho(x,t))+\epsilon p^{\epsilon}\left(\operatorname{Pr}_{\Gamma_{t}}(x),t\right)\theta_{1}(\rho(x,t))\right)+\xi(d_{\Gamma}(x,t))\epsilon^{2}q^{\epsilon}(x,t)\nonumber \\
 & \quad+\left(1-\xi(d_{\Gamma}(x,t))\right)\left(c_{A}^{\epsilon,+}(x,t)\chi_{\Omega_{T_{0}}^{+}}(x,t)+c_{A}^{\epsilon,-}(x,t)\chi_{\Omega_{T_{0}}^{-}}(x,t)\right)\label{eq:fie}
\end{align}
for all $(x,t)\in\Omega_{T}$, where $\rho(x,t):=\frac{d_{\Gamma}(x,t)}{\epsilon}-h^{\epsilon}(S(x,t),t)$.
The occurring functions are supposed to be smooth and satisfy for
some $C^{*}>0$ the following properties:

$\theta_{1}\colon \mathbb{R}\rightarrow\mathbb{R}$ is a bounded function
satisfying 
\begin{equation}
\int_{\mathbb{R}}\theta_{1}(\rho)\theta_{0}'(\rho)^{2}f^{\left(3\right)}(\theta_{0}(\rho))\d\rho=0.\label{eq:teta0}
\end{equation}
 Furthermore, $p^{\epsilon}\colon \Gamma\rightarrow\mathbb{R}$, $q^{\epsilon}\colon \Gamma(2\delta)\rightarrow\mathbb{R}$
satisfy 
\begin{equation}
\sup_{\epsilon\in\left(0,\epsilon_{0}\right)}\sup_{(x,t)\in\Gamma\left(2\delta;T\right)}\left(\left|p^{\epsilon}(\operatorname{Pr}_{\Gamma_{t}}(x),t)\right|+\frac{\epsilon}{\epsilon+\left|d_{\Gamma}(x,t)-\epsilon h^{\epsilon}(S(x,t),t)\right|}\left|q^{\epsilon}(x,t)\right|\right)\leq C^{*},\label{eq:Dasistes}
\end{equation}
$h^{\epsilon}\colon \mathbb{T}^{1}\times\left[0,T\right]\rightarrow\mathbb{R}$
fulfills
\begin{equation}
\sup_{\epsilon\in\left(0,\epsilon_{0}\right)}\sup_{(s,t)\in\mathbb{T}^{1}\times\left[0,T\right]}\left(\left|h^{\epsilon}(s,t)\right|+\left|\partial_{s}h^{\epsilon}(s,t)\right|\right)\leq C^{*}\label{eq:hbes}
\end{equation}
and $c_{A}^{\epsilon,\pm}\colon \Omega_{T}^{\pm}\rightarrow\mathbb{R}$
satisfy 
\begin{equation}
\pm c_{A}^{\epsilon,\pm}>0\text{ in }\Omega_{T}^{\pm}.\label{eq:caepm}
\end{equation}
 Additionally, we suppose that there is some $C^\ast$ such that 
\begin{align}
\sup_{\epsilon\in\left(0,\epsilon_{0}\right)}\left(\sup_{(x,t)\in\Omega_{T}}\left|c_{A}^{\epsilon}(x,t)\right|+\sup_{x\in\Gamma(\delta)}\left|\nabla^{\Gamma}c_{A}^{\epsilon}(x,t)\right|\right)&\leq C^\ast,\label{eq:obfi}\\
\inf_{\epsilon\in\left(0,\epsilon_{0}\right)}\inf_{(x,t)\in\Omega_{T}\backslash\Gamma\left(\delta;T\right)}f''\left(c_{A}^{\epsilon}(x,t)\right)&\geq\frac{1}{C^\ast}.\label{eq:fnachunten}
\end{align}
\end{assumption}

\begin{cor}
\label{Chen-hilf} Let Assumptions \ref{assu:Spektral} hold true and let
$t\in\left[0,T\right]$, let $\psi\in H^{1}(\Gamma_{t}(\delta))$
and $\Lambda_{\epsilon}\in\mathbb{R}$ be such that 
\[
\int_{\Gamma_{t}(\delta)}\epsilon\left|\nabla\psi(x)\right|^{2}+\epsilon^{-1}f''\left(c_{A}^{\epsilon}(x,t)\right)\psi(x)^{2}\d x\leq\Lambda_{\epsilon}
\]
and denote $I_{\epsilon}^{s,t}:=\left(-\frac{\delta}{\epsilon}-h^{\epsilon}(s,t),\frac{\delta}{\epsilon}-h^{\epsilon}(s,t)\right)$.
Then, for $\epsilon>0$ small enough, there exist functions $Z\in H^{1}(\mathbb{T}^{1})$,
$\psi^{\mathbf{R}}\in H^{1}(\Gamma_{t}(\delta))$
and smooth $\Psi\colon I_{\epsilon}^{s,t}\times\mathbb{T}^{1}\to \R$ such that
\begin{equation}
\psi\left(r,s\right)=\epsilon^{-\frac{1}{2}}Z(s)\left(\beta(s)\theta_{0}'\left(\rho\left(r,s\right)\right)+\Psi\left(\rho\left(r,s\right),s\right)\right)+\psi^{\mathbf{R}}\left(r,s\right)\label{eq:Chenneu1}
\end{equation}
for almost all $\left(r,s\right)\in\left(-\delta,\delta\right)\times\mathbb{T}^{1}$,
where $\rho\left(r,s\right)=\frac{r}{\epsilon}-h^{\epsilon}(s,t)$
and $\beta(s)=\left(\int_{I_{\epsilon}^{s,t}}\left(\theta_{0}'(\rho)\right)^{2}\d\rho\right)^{-\frac{1}{2}}$.
Moreover,
\begin{equation}
\left\Vert \psi^{\mathbf{R}}\right\Vert _{L^{2}\left(\Gamma_{t}(\delta)\right)}^{2}\leq C\left(\epsilon\Lambda_{\epsilon}+\epsilon^{2}\left\Vert \psi\right\Vert _{L^{2}\left(\Gamma_{t}(\delta)\right)}^{2}\right),\label{eq:Chenneu3}
\end{equation}
\begin{equation}
\left\Vert Z\right\Vert _{H^{1}\left(\mathbb{T}^{1}\right)}^{2}+\left\Vert \nabla^{\Gamma}\psi\right\Vert _{L^{2}\left(\Gamma_{t}(\delta)\right)}^{2}+\left\Vert \psi^{\mathbf{R}}\right\Vert _{H^{1}\left(\Gamma_{t}(\delta)\right)}^{2}\le C\left(\left\Vert \psi\right\Vert _{L^{2}\left(\Gamma_{t}(\delta)\right)}^{2}+\frac{\Lambda_{\epsilon}}{\epsilon}\right),\label{eq:Chenneu2}
\end{equation}
and
\begin{equation}
\sup_{s\in\mathbb{T}^{1}}\left(\int_{I_{\epsilon}^{s,t}}\left(\Psi\left(\rho,s\right)^{2}+\Psi_{\rho}\left(\rho,s\right)^{2}\right)J\left(\epsilon\left(\rho+h^{\epsilon}(s,t)\right),s\right)\d\rho\right)\leq C\epsilon^{2}.\label{eq:Chenneu4}
\end{equation}
\end{cor}
\begin{proof}
We define 
$
\tilde{\psi}:=\frac{\psi}{\left\Vert \psi\right\Vert _{L^{2}\left(\Gamma_{t}(\delta)\right)}}.
$
Then we have 
\[
\int_{\Gamma_{t}(\delta)}\epsilon\left|\nabla\tilde{\psi}\right|^{2}+\epsilon^{-1}f''\left(c_{A}^{\epsilon}\right)\tilde{\psi}^{2}\d x\leq\frac{\Lambda_{\epsilon}}{\left\Vert \psi\right\Vert _{L^{2}\left(\Gamma_{t}(\delta)\right)}^{2}}
\]
and may use  Lemma~2.2 and Lemma~2.4 in \cite{chen} adapted to
the case of the stretched variable $\rho=\frac{r}{\epsilon}-h^{\epsilon}(s,t)$
instead of $z=\frac{r}{\epsilon}$, where $r\in I_{1}$, $s\in\mathbb{T}^{1}$, see \cite[Chapter~3]{ichPhD} for the details.
This yields existence of some functions $\tilde{Z}\in H^{1}(\mathbb{T}^{1})$,
$\tilde{\psi}^{\mathbf{R}}\in H^{1}(\Gamma_{t}(\delta))$
and $\Psi$ such that
\begin{equation}
\tilde{\psi}\left(r,s\right)=\epsilon^{-\frac{1}{2}}\tilde{Z}(s)\left(\beta(s)\theta_{0}'\left(\rho\left(r,s\right)\right)+\Psi\left(\rho\left(r,s\right),s\right)\right)+\tilde{\psi}^{\mathbf{R}}\left(r,s\right)\label{eq:Chenspec1}
\end{equation}
with
\begin{align}
\Vert \tilde{Z}\Vert _{H^{1}(\mathbb{T}^{1})}^{2}+\Vert \nabla^{\Gamma}\tilde{\psi}\Vert _{L^{2}(\Gamma_{t}(\delta))}^{2}+\Vert \tilde{\psi}^{\mathbf{R}}\Vert _{H^{1}(\Gamma_{t}(\delta))}^{2}&\le C\left(1+\frac{\Lambda_{\epsilon}}{\epsilon\Vert \psi\Vert _{L^{2}\left(\Gamma_{t}(\delta)\right)}^{2}}\right),\label{eq:Chenspec2}\\
\Vert \tilde{\psi}^{\mathbf{R}}\Vert _{L^{2}\left(\Gamma_{t}(\delta)\right)}^{2}&\leq C\left(\epsilon\frac{\Lambda_{\epsilon}}{\left\Vert \psi\right\Vert _{L^{2}\left(\Gamma_{t}(\delta)\right)}^{2}}+\epsilon^{2}\right),\label{eq:Chenspec3}
\end{align}
and such that $\Psi$ satisfies (\ref{eq:Chenneu4}). Furthermore,
if we define
\begin{align*}
  \psi_{1}\left(r,s\right)&:=\epsilon^{-\frac{1}{2}}\left(\beta(s)\theta_{0}'(\rho(r,s))+\Psi(\rho(r,s),s)\right), \\
  Z(s)&:=\left(\psi_{1},\psi\right)_{J}\quad\text{and}\quad \psi^{\mathbf{R}}(r,s):=\psi(r,s)-Z(s)\psi_{1}(r,s),
\end{align*}
we have the identities
\begin{align*}
Z(s) & =\left(\psi_{1},\psi\right)_{J}=\left(\psi_{1},\tilde{\psi}\left\Vert \psi\right\Vert _{L^{2}\left(\Gamma_{t}(\delta)\right)}\right)=\tilde{Z}(s)\left\Vert \psi\right\Vert _{L^{2}\left(\Gamma_{t}(\delta)\right)}\quad
and\\
\psi^{\mathbf{R}}\left(r,s\right) & =\tilde{\psi}\left(r,s\right)\left\Vert \psi\right\Vert _{L^{2}\left(\Gamma_{t}(\delta)\right)}+\tilde{Z}(s)\left\Vert \psi\right\Vert _{L^{2}\left(\Gamma_{t}(\delta)\right)}, \psi_{1}(r,s)=\tilde{\psi}^{\mathbf{R}}(r,s)\left\Vert \psi\right\Vert _{L^{2}\left(\Gamma_{t}(\delta)\right)}
\end{align*}
for almost all $\left(r,s\right)\in\left(-\delta,\delta\right)\times\mathbb{T}^{1}$.
Thus, (\ref{eq:Chenneu1}), (\ref{eq:Chenneu3}), (\ref{eq:Chenneu2})
follow immediately.
\end{proof}
In the following we consider $H_{0}^{1}(\Omega)$ equipped
with the scalar product $\left(u,v\right)_{1}=\int_{\Omega}\nabla u\cdot\nabla v\d x$.
The induced norm $\left|.\right|_{1}$ is equivalent to the usual
$H^{1}$-norm by Poincar\'e's inequality. 
\begin{thm}[Spectral Estimate]
 \label{specHillmod}~\\ Let Assumption \ref{assu:Spektral} hold true and
let $t\in\left[0,T\right]$. There exist constants $C_{1}>0$, $C_{2}\geq0$
and $\epsilon_{1}>0$, independent of $t$, such that for all $\psi\in H_{0}^{1}(\Omega)$
it holds
\begin{align}
&\int_{\Omega}\epsilon\left|\nabla\psi\right|^{2}+\epsilon^{-1}f''\left(c_{A}^{\epsilon}\right)\psi^{2}\d x  \geq C_{1}\left(\epsilon\left\Vert \psi\right\Vert _{L^{2}(\Omega)}^{2}+\epsilon^{-1}\left\Vert \psi\right\Vert _{L^{2}\left(\Omega\backslash\Gamma_{t}(\delta)\right)}+\epsilon\left\Vert \nabla^{\Gamma}\psi\right\Vert _{L^{2}\left(\Gamma_{t}(\delta)\right)}^{2}\right)\nonumber \\
 & \qquad+C_{1}\left(\epsilon^{3}\left\Vert \nabla\psi\right\Vert _{L^{2}(\Omega)}^{2}+\epsilon\left\Vert \nabla\psi\right\Vert _{L^{2}\left(\Omega\backslash\Gamma_{t}(\delta)\right)}^{2}\right)-C_{2}\left\Vert \psi\right\Vert _{H^{-1}(\Omega)}^{2}.\label{eq:spect}
\end{align}
 for all $\epsilon\in\left(0,\epsilon_{1}\right)$.
\end{thm}

\begin{proof}
Due to (\ref{eq:fnachunten}) we may estimate 
\begin{align}
  &\int_{\Omega}\epsilon\left|\nabla\psi\right|^{2}+\epsilon^{-1}f''\left(c_{A}^{\epsilon}\right)\psi^{2}\d x \nonumber\\
  & \geq\int_{\Omega\backslash\Gamma_{t}(\delta)}\epsilon\left|\nabla\psi\right|^{2}+C\epsilon^{-1}\left|\psi\right|^{2}\d x+\int_{\mathbb{T}^{1}}\int_{-\delta}^{\delta}\left(\epsilon\left|\psi_{r}\right|^{2}+\epsilon\left|\nabla_{\boldsymbol{\tau}}\psi\right|^{2}+\epsilon^{-1}f''\left(c_{A}^{\epsilon}\right)\psi^{2}\right)J\d r\d s\nonumber \\
  & \geq\int_{\Omega\backslash\Gamma_{t}(\delta)}\epsilon\left|\nabla\psi\right|^{2}+C_{1}\epsilon^{-1}\left|\psi\right|^{2}\d x+\epsilon\int_{\Gamma_{t}(\delta)}\left|\nabla_{\boldsymbol{\tau}}\psi\right|^{2}\d x -C_{2}\epsilon\int_{\Omega}\psi^{2}\d x,\label{eq:Abspekt}
\end{align}
where the last inequality is a consequence of \cite[Lemma
2.2]{chen},  adapted to the case of the stretched variable $\rho=\frac{r}{\epsilon}-h^{\epsilon}(s,t)$
instead of $z=\frac{r}{\epsilon}$, where $r\in\left(-\delta,\delta\right)$,
$s\in\Gamma$, cf.~\cite[Proof of Theorem~3.12]{ichPhD} for more details. We observe that we may now use (\ref{eq:Abspekt})
to derive 
\begin{align}
\int_{\Omega}\epsilon\left|\nabla\psi\right|^{2}+\epsilon^{-1}f''\left(c_{A}^{\epsilon}\right)\psi^{2}\d x & \geq C_{1}\left(\epsilon\left\Vert \nabla^{\Gamma}\psi\right\Vert _{L^{2}\left(\Gamma_{t}(\delta)\right)}^{2}+\epsilon^{-1}\left\Vert \psi\right\Vert _{L^{2}\left(\Omega\backslash\Gamma_{t}(\delta)\right)}^{2}+\epsilon\left\Vert \nabla\psi\right\Vert _{L^{2}\left(\Omega\backslash\Gamma_{t}(\delta)\right)}^{2}\right)\nonumber \\
 & \quad+C_{1}\epsilon^{3}\left\Vert \nabla\psi\right\Vert _{L^{2}(\Omega)}^{2}-C_{2}\epsilon\left\Vert \psi\right\Vert _{L^{2}(\Omega)}^{2}\label{eq:N2}
\end{align}
for $C_{1},C_{2}>0$ and all $\epsilon\in\left(0,\epsilon_{1}\right)$,
after choosing $\epsilon_{1}$ so small that $\epsilon_{1}\leq\frac{1}{2}$
is fulfilled. Now, in order to prove (\ref{eq:spect}) we fix a constant
$c>C_{2}$ and $\epsilon\in\left(0,\epsilon_{0}\right)$ and consider
two different cases:
First, we assume 
\[
\int_{\Omega}\epsilon\left|\nabla\psi\right|^{2}+\epsilon^{-1}f''\left(c_{A}^{\epsilon}\right)\psi^{2}\d x>c\epsilon\left\Vert \psi\right\Vert _{L^{2}(\Omega)}^{2}
\]
which leads to the claim immediately, with $C_{2}=0.$
In the case
\[
\int_{\Omega}\epsilon\left|\nabla\psi\right|^{2}+\epsilon^{-1}f''\left(c_{A}^{\epsilon}\right)\psi^{2}\d x\leq c\epsilon\left\Vert \psi\right\Vert _{L^{2}(\Omega)}^{2}
\]
let $w\in H^{2}(\Omega)\cap H_{0}^{1}(\Omega)$
be the unique solution to $-\Delta w = \psi$.
Then \cite[Theorem~3.1]{chen} implies 
\begin{equation}
\tilde{C}\epsilon\left\Vert \psi\right\Vert _{L^{2}(\Omega)}^{2}\leq\left\Vert \nabla w\right\Vert _{L^{2}(\Omega)}^{2}.\label{eq:L2est}
\end{equation}
Moreover,$\left\Vert \psi\right\Vert _{H^{-1}(\Omega)}^{2}=\left\Vert \nabla w\right\Vert _{L^{2}(\Omega)}^{2}$
and thus we get
\begin{align*}
\int_{\Omega}\epsilon\left|\nabla\psi\right|^{2}+\epsilon^{-1}f'\left(c_{A}^{\epsilon}\right)\psi^{2}\d x & \geq C\left(\epsilon\left\Vert \psi\right\Vert _{L^{2}(\Omega)}^{2}+\epsilon^{-1}\left\Vert \psi\right\Vert _{L^{2}\left(\Omega\backslash\Gamma_{t}(\delta)\right)}+\epsilon\left\Vert \nabla^{\Gamma}\psi\right\Vert _{L^{2}\left(\Gamma_{t}(\delta)\right)}^{2}\right)\\
 & \quad+C\left(\epsilon\left\Vert \nabla\psi\right\Vert _{L^{2}\left(\Omega\backslash\Gamma_{t}(\delta)\right)}^{2}+\epsilon^{3}\left\Vert \nabla\psi\right\Vert _{L^{2}(\Omega)}^{2}\right)-\tilde{C}\left\Vert \psi\right\Vert _{H^{-1}(\Omega)}^{2}.
\end{align*}
 This proves the assertion.
\end{proof}

\section{Proof of Theorem \ref{Main} \label{chap:Main Proof}}


\subsection{\label{subsec:The-Approximate-Solutions}The Approximate Solutions}

A major ingredient of this work is the construction of an approximate
solution, which satisfies (\ref{eq:StokesPart})\textendash (\ref{eq:StokesBdry})
up to a sufficiently high order. In the following we present a collection
of properties of the approximations, which are necessary to prove
Theorem \ref{Main} and are constructed in \cite{NSCH2}, alternatively see \cite{ichPhD}.
\begin{thm}
\label{thm:Main-Apprx-Structure}For every $\epsilon\in\left(0,1\right)$
there are
$\mathbf{v}_{A}^{\epsilon},\mathbf{w}_{1}^{\epsilon}\colon\Omega_{T_{0}}\rightarrow\mathbb{R}^{2}$ $c_{A}^{\epsilon},\mu_{A}^{\epsilon},p_{A}^{\epsilon}\colon\Omega_{T_{0}}\rightarrow\mathbb{R}$
and $\rs\colon\Omega_{T_{0}}\rightarrow\mathbb{R}^{2}$, $\rdiv,\rc,\rh\colon \Omega_{T_{0}}\rightarrow\mathbb{R}$
such that
\begin{align}
-\Delta\mathbf{v}_{A}^{\epsilon}+\nabla p_{A}^{\epsilon} & =\mu_{A}^{\epsilon}\nabla c_{A}^{\epsilon}+\rs,\label{eq:Stokesapp}\\
\mathrm{div}\mathbf{v}_{A}^{\epsilon} & =\rdiv,\label{eq:Divapp}\\
\partial_{t}c_{A}^{\epsilon}+\left(\mathbf{v}_{A}^{\epsilon}+\epsilon^{M-\frac{1}{2}}\left.\mathbf{w}_{1}^{\epsilon}\right|_{\Gamma}\xi\left(d_{\Gamma}\right)\right)\cdot\nabla c_{A}^{\epsilon} & =\Delta\mu_{A}^{\epsilon}+\rc,\label{eq:Cahnapp}\\
\mu_{A}^{\epsilon} & =-\epsilon\Delta c_{A}^{\epsilon}+\epsilon^{-1}f'\left(c_{A}^{\epsilon}\right)+\rh,\label{eq:Hilliardapp}
\end{align}
in $\Omega_{T_{0}}$ and
\begin{equation}
c_{A}^{\epsilon}=-1,\quad \mu_{A}^{\epsilon}=0,\quad \left(-2D_{s}\mathbf{v}_{A}^{\epsilon}+p_{A}^{\epsilon}\mathbf{I}\right)\mathbf{n}_{\partial\Omega}=\alpha_{0}\mathbf{v}_{A}^{\epsilon},\quad \rdiv=0\label{eq:boundaryconditions}
\end{equation}
are satisfied on $\partial_{T_{0}}\Omega$. If additionally
Assumption \ref{assu:Main-est} holds for $\epsilon_{0}\in\left(0,1\right)$,
$K\geq1$ and a family $\left(T_{\epsilon}\right)_{\epsilon\in\left(0,\epsilon_{0}\right)}\subset\left(0,T_{0}\right]$,
then there are some $\epsilon_{1}\in\left(0,\epsilon_{0}\right]$,
$C(K)>0$ depending on $K$ and $C_{K}:\left(0,T_{0}\right]\times\left(0,1\right]\rightarrow\left(0,\infty\right)$
(also dependent on $K$), which satisfies $C_{K}\left(T,\epsilon\right)\rightarrow0$
as $\left(T,\epsilon\right)\rightarrow0$, such that
\begin{align}
\int_{0}^{T_{\epsilon}}\left|\int_{\Omega}\rc(x,t)\varphi(x,t)\d x\right|\d t & \leq C_{K}\left(T_{\epsilon},\epsilon\right)\epsilon^{M}\left\Vert \varphi\right\Vert _{L^{\infty}\left(0,T_{\epsilon};H^{1}(\Omega)\right)},\label{eq:remcahn}\\
\int_{0}^{T_{\epsilon}}\left|\int_{\Omega}\rh(x,t)\left(c^{\epsilon}(x,t)-c_{A}^{\epsilon}(x,t)\right)\d x\right|\d t & \leq C_{K}\left(T_{\epsilon},\epsilon\right)\epsilon^{2M},\label{eq:remhill}\\
\left\Vert \rs\right\Vert _{L^{2}\left(0,T_{\epsilon};\left(H^{1}(\Omega)\right)'\right)}+\left\Vert \rdiv\right\Vert _{L^{2}\left(\Omega_{T_{\epsilon}}\right)} & \leq C(K)\epsilon^{M},\label{eq:remstokes}\\
\left\Vert \rh\nabla c_{A}^{\epsilon}\right\Vert _{L^{2}\left(0,T_{\epsilon};\left(H^{1}(\Omega)^{2}\right)'\right)} & \leq C(K)C\left(T_{\epsilon},\epsilon\right)\epsilon^{M}\label{eq:rch2-nablacae}\\
\left\Vert \rc\right\Vert _{L^{2}\left(\partial_{T_{\epsilon}}\Omega\left(\frac{\delta}{2}\right)\right)} & \leq C(K)\epsilon^{M}\label{eq:rch1-rch2-Linfbdry}
\end{align}
for all $\epsilon\in\left(0,\epsilon_{1}\right)$ and $\varphi\in L^{\infty}\left(0,T_{\epsilon};H^{1}(\Omega)\right)$.
\end{thm}

In the following we will need a more intricate knowledge of the approximate
solutions. Let $\xi$ be a cut-off function satisfying (\ref{eq:cut-off}),
and we denote $\mathbf{v}^{\pm}:=\mathbf{v}|_{\Omega_{T_{0}}^{\pm}}$,
$\mu^{\pm}:=\mu|_{\Omega_{T_{0}}^{\pm}}$ for solutions $\mu,\mathbf{v}$
of (\ref{eq:S-SAC1})\textendash (\ref{eq:S-SAC8}). We assume that
$\mathbf{v}^{\pm}$, $\mu^{\pm}$ are smoothly extended to $\Omega_{T_{0}}^{\pm}\cup\Gamma\left(2\delta;T_{0}\right)$,
where \textbf{$\mathbf{v}^{\pm}$ }is moreover divergence free in
that region. We refer to \cite[Remark~3.1]{NSCH2},  for more details on this
extension and \cite[Remark~4.3]{NSCH2} for more information on the
structural details discussed below. We have
\begin{align*}
c_{A}^{\epsilon}(x,t) & =\xi\left(d_{\Gamma}(x,t)\right)c_{I}(x,t)+\left(1-\xi\left(d_{\Gamma}(x,t)\right)\right)c_{O,\mathbf{B}}(x,t),\\
\mu_{A}^{\epsilon}(x,t) & =\xi\left(d_{\Gamma}(x,t)\right)\mu_{I}(x,t)+\left(1-\xi\left(d_{\Gamma}(x,t)\right)\right)\mu_{O,\mathbf{B}}(x,t)+\epsilon^{M-\frac{1}{2}}\mu_{A,M-\frac{1}{2}}^{\epsilon}(x,t),\\
\mathbf{v}_{A}^{\epsilon}(x,t) & =\xi\left(d_{\Gamma}(x,t)\right)\mathbf{v}_{I}(x,t)+\left(1-\xi\left(d_{\Gamma}(x,t)\right)\right)\mathbf{v}_{O,\mathbf{B}}(x,t)+\epsilon^{M-\frac{1}{2}}\mathbf{v}_{A,M-\frac{1}{2}}^{\epsilon}(x,t)
\end{align*}
for $(x,t)\in\Omega_{T_{0}}$. Here $c_{O,\mathbf{B}}=\pm1+\mathcal{O}\left(\epsilon\right)$
in $C^{1}\left(\Omega_{T_{0}}^{\pm}\right)$ as $\epsilon\rightarrow0$,
with $\left\Vert c_{O,\mathbf{B}}\right\Vert _{C^{2}\left(\Omega_{T_{0}}^{\pm}\right)}\leq C$
and $\mu_{O,\mathbf{B}}=\mu^{\pm}+\mathcal{O}\left(\epsilon\right)$
and $\mathbf{v}_{O,\mathbf{B}}=\mathbf{v}^{\pm}+\mathcal{O}\left(\epsilon\right)$
in $L^{\infty}\left(\Omega_{T_{0}}^{\pm}\right)$ as $\epsilon\rightarrow0$.
Moreover, 
\[
c_{I}(x,t)=\sum_{k=0}^{M+1}\epsilon^{k}c_{k}\left(\rho(x,t),x,t\right)\;\forall(x,t)\in\Gamma\left(2\delta;T_{0}\right),
\]
where $c_{k}:\mathbb{R}\times\Gamma\left(2\delta;T_{0}\right)\rightarrow\mathbb{R}$
, $k\in\left\{ 0,\ldots,M+1\right\} $, are smooth and bounded functions,
which do not depend on $\epsilon$ and have bounded derivatives. Here
\[
\rho(x,t)=\frac{d_{\Gamma}(x,t)}{\epsilon}-h_{A}^{\epsilon}\left(S(x,t),t\right)\quad\forall(x,t)\in\Gamma\left(2\delta;T_{0}\right),
\]
where $h_{A}^{\epsilon}(s,t)=\sum_{k=0}^{M}\epsilon^{k}h_{k+1}(s,t)+\epsilon^{M-\frac{3}{2}}h_{M-\frac{1}{2}}^{\epsilon}(s,t)$
for $(s,t)\in\mathbb{T}^{1}\times\left[0,T_{0}\right]$
and $h_{k}$ are smooth and bounded functions independent of $\epsilon$
with bounded derivatives, for $k\in\left\{ 1,\ldots,M+1\right\} $.
Moreover, $\mu_{I}$ and $\mathbf{v}_{I}$ have the same kind of expansion with coefficients $\mu_{k}$ and $\mathbf{v}_{k}$, $k\in\left\{ 0,\ldots,M+1\right\} $.
In particular, we have 
\begin{align}
c_{0}(\rho,x,t) & =\theta_{0}(\rho),\;\mathbf{v}_{0}(\rho,x,t)=\mathbf{v}^{+}(x,t)\eta(\rho)-\mathbf{v}^{-}(x,t)\left(1-\eta(\rho)\right)\nonumber \\
\mu_{0}(\rho,x,t) & =\mu^{+}(x,t)\eta(\rho)-\mu^{-}(x,t)\left(1-\eta(\rho)\right)\label{eq:0teordnung}
\end{align}
for $(\rho,x,t)\in\mathbb{R}\times\Gamma\left(2\delta;T_{0}\right)$,
where $\eta:\mathbb{R}\rightarrow\left[0,1\right]$ is smooth and
satisfies $\eta=0$ in $\left(-\infty,-1\right]$, $\eta=1$ in $\left[1,\infty\right)$
and $\eta'\geq0$ in $\mathbb{R}$. The so-called inner terms satisfy
moreover $\nabla_{x}^{l}\partial_{t}^{m}\partial_{\rho}^{i}u\in\mathcal{R}_{\alpha}$
for some $\alpha>0$, where $i\geq1,m,l\geq0$ and $u=c_{k},\mu_{k},\mathbf{v}_{k}$
for $k\in\left\{ 0,\ldots,M+1\right\} $. Additionally, we note for
later use $h_{A}^{\epsilon}\left(s,0\right)=0$ for all $s\in\mathbb{T}^{1}$.

Regarding the structure of the fractional order terms, we have
\begin{align*}
\mu_{A,M-\frac{1}{2}}^{\epsilon} & =\xi\left(d_{\Gamma}\right)\mu_{M-\frac{1}{2}}^{\epsilon}+\left(1-\xi\left(d_{\Gamma}\right)\right)\left(\mu_{M-\frac{1}{2}}^{+,\epsilon}\overline{\chi_{\Omega_{T_{0}}^{+}}}+\mu_{M-\frac{1}{2}}^{-,\epsilon}\chi_{\Omega_{T_{0}}}^{-}\right),\\
\mathbf{v}_{A,M-\frac{1}{2}}^{\epsilon}(x,t) & =\xi\left(d_{\Gamma}\right)\mathbf{v}_{M-\frac{1}{2}}^{\epsilon}+\left(1-\xi\left(d_{\Gamma}\right)\right)\left(\mathbf{v}_{M-\frac{1}{2}}^{+,\epsilon}\overline{\chi_{\Omega_{T_{0}}^{+}}}+\mathbf{v}_{M-\frac{1}{2}}^{-,\epsilon}\chi_{\Omega_{T_{0}}}^{-}\right)
\end{align*}
in $\Omega_{T_{0}}$, where $\mu_{M-\frac{1}{2}}^{\pm,\epsilon},\mathbf{v}_{M-\frac{1}{2}}^{\pm,\epsilon}$
are functions defined on $\Omega_{T_{0}}^{\pm}\cup\Gamma\left(2\delta;T_{0}\right)$
and $\mu_{M-\frac{1}{2}}^{\epsilon}:=\mu_{M-\frac{1}{2}}^{+,\epsilon}\eta-\mu_{M-\frac{1}{2}}^{-,\epsilon}\left(1-\eta\right)$
and \textbf{$\mathbf{v}_{M-\frac{1}{2}}^{\epsilon}:=\mathbf{v}_{M-\frac{1}{2}}^{+,\epsilon}\eta-\mathbf{v}_{M-\frac{1}{2}}^{-,\epsilon}\left(1-\eta\right)$
}in $\Gamma\left(2\delta;T_{0}\right)$. As technical details, we
remark that 
\begin{equation}
\rh=\epsilon^{M-\frac{1}{2}}\mu_{M-\frac{1}{2}}^{-}+\mathcal{O}\left(\epsilon^{M+1}\right)\text{ in }L^{\infty}\left(\Omega_{T_{0}}\backslash\Gamma(2\delta)\right)\text{ as }\epsilon\rightarrow0,\label{eq:bdry-remainder}
\end{equation}
which is a direct consequence of \cite[Remark~4.4]{NSCH2} and that
$\mu_{M-\frac{1}{2}}^{-}=0$ on $\partial_{T_{0}}\Omega$, which is
discussed in \cite[Remark~4.3]{NSCH2}.

A key element in the proof of Theorem \ref{Main} is an  understanding
of the term $\mathbf{w}_{1}^{\epsilon}$ mentioned in Theorem \ref{thm:Main-Apprx-Structure}
and also of the appearing fractional order terms, which are in the
end a consequence of the appearance of $\mathbf{w}_{1}^{\epsilon}$.
This motivates the following analysis: For $T\in\left(0,T_{0}\right]$
we consider weak solutions $\tilde{\mathbf{w}}_{1}^{\epsilon}:\Omega_{T}\rightarrow\mathbb{R}^{2}$
and $q_{1}^{\epsilon}:\Omega_{T}\rightarrow\mathbb{R}$ of 
\begin{align}
-\Delta\tilde{\mathbf{w}}_{1}^{\epsilon}+\nabla q_{1}^{\epsilon} & =-\epsilon\operatorname{div}\left(\left(\nabla c_{A}^{\epsilon}-\mathbf{h}\right)\otimes_{s}\nabla R\right) &  & \text{in }\Omega_{T},\label{eq:w1}\\
\operatorname{div}\tilde{\mathbf{w}}_{1}^{\epsilon} & =0 &  & \text{in }\Omega_{T},\label{eq:w12}\\
\left(-2D_{s}\tilde{\mathbf{w}}_{1}^{\epsilon}+q_{1}^{\epsilon}\mathbf{I}\right)\cdot\mathbf{n}_{\partial\Omega} & =\alpha_{0}\tilde{\mathbf{w}}_{1}^{\epsilon} &  & \text{on }\partial_{T}\Omega\label{eq:w13}
\end{align}
 in the sense of (\ref{eq:instokesweak}). Here we denote $R:=c^{\epsilon}-c_{A}^{\epsilon}$
and we define $\mathbf{h}$ by 
\begin{equation}
\mathbf{h}(x,t):=-\xi\left(d_{\Gamma}(x,t)\right)\sum_{k=0}^{M+1}\epsilon^{k}\partial_{\rho}c_{k}\left(\rho(x,t),x,t\right)\epsilon^{M-\frac{3}{2}}\nabla^{\Gamma}h_{M-\frac{1}{2}}^{\epsilon}(x,t)\label{eq:hh}
\end{equation}
and $\otimes_{s}$ as $\mathbf{a}\otimes_{s}\mathbf{b}:=\mathbf{a}\otimes\mathbf{b}+\mathbf{b}\otimes\mathbf{a}$
for $\mathbf{a},\mathbf{b}\in\mathbb{R}^{n}$. We calculate 
\begin{align}
 & \left(\nabla c_{A}^{\epsilon}-\mathbf{h}\right)(x,t)\nonumber \\
 & =\xi'\left(d_{\Gamma}(x,t)\right)\nabla d_{\Gamma}(x,t)c_{I}(x,t)+\xi\left(d_{\Gamma}(x,t)\right)\left(\sum_{k=0}^{M+1}\epsilon^{k}\nabla_{x}c_{k}\left(\rho(x,t),x,t\right)\right)\nonumber \\
 & \quad+\xi\left(d_{\Gamma}(x,t)\right)\left(\sum_{k=0}^{M+1}\epsilon^{k}\partial_{\rho}c_{k}\left(\rho(x,t),x,t\right)\left(\rho(x,t),x,t\right)\left(\frac{1}{\epsilon}\nabla d_{\Gamma}(x,t)-\sum_{i=0}^{M}\epsilon^{i}\nabla^{\Gamma}h_{i+1}(x,t)\right)\right)\nonumber \\
 & \quad+\nabla\left(\left(1-\xi\left(d_{\Gamma}(x,t)\right)\right)c_{O,\mathbf{B}}(x,t)\right)\label{eq:ca-h}
\end{align}
for $(x,t)\in\Omega_{T_{0}}$. We understand the right-hand side of (\ref{eq:w1}) as a functional in $\left(V_{0}\right)'$
given by
\begin{equation}
\mathbf{f}^{\epsilon}\left(\psi\right):=\int_{\Omega}\epsilon\left(\left(\nabla c_{A}^{\epsilon}-\mathbf{h}\right)\otimes\nabla R+\nabla R\otimes\left(\nabla c_{A}^{\epsilon}-\mathbf{h}\right)\right):\nabla\psi\d x\label{eq:fepsh}
\end{equation}
for $\psi\in V_{0}$ and fixed $t\in\left[0,T\right]$. $\mathbf{w}_{1}^{\epsilon}$
as introduced in Theorem \ref{thm:Main-Apprx-Structure} is just a
rescaling of $\tilde{\mathbf{w}}_{1}^{\epsilon}$, i.e.,
\begin{equation}
\mathbf{w}_{1}^{\epsilon}=\frac{\tilde{\mathbf{w}}_{1}^{\epsilon}}{\epsilon^{M-\frac{1}{2}}}\label{eq:w1e}
\end{equation}
holds. Furthermore, we introduce
\begin{eqnarray}
X_{T} & = & L^{2}\left(0,T;H^{\frac{7}{2}}\left(\mathbb{T}^{1}\right)\right)\cap H^{1}\left(0,T;H^{\frac{1}{2}}\left(\mathbb{T}^{1}\right)\right)\label{eq:XT-1}
\end{eqnarray}
for $T\in\mathbb{R}_{+}\cup\left\{ \infty\right\} $, where we equip
$X_{T}$ with the norm 
\[
\left\Vert h\right\Vert _{X_{T}}=\left\Vert h\right\Vert _{L^{2}\left(0,T;H^{\frac{7}{2}}\left(\mathbb{T}^{1}\right)\right)}+\left\Vert h\right\Vert _{H^{1}\left(0,T;H^{\frac{1}{2}}\left(\mathbb{T}^{1}\right)\right)}+\left\Vert h|_{t=0}\right\Vert _{H^{2}\left(\mathbb{T}^{1}\right)}.
\]
Note that $X_{T}\hookrightarrow C^{0}\left(\left[0,T\right];H^{2}\left(\mathbb{T}^{1}\right)\right)$,
where the operator norm of the embedding can be bounded independently
of $T$.

The following lemma is shown in \cite[Lemma~3.13]{NSCH2} and enables
us to access the results obtained in Subsection \ref{chap:Spectral-Theory}.
\begin{lem}
\label{lem:spekholds}Let $\epsilon_{0}>0$, $T\in\left(0,T_{0}\right]$
and $\left(T_{\epsilon}\right)_{\epsilon\in\left(0,\epsilon_{0}\right)}\subset\left(0,T\right]$
be given. We assume that there is some $\bar{C}>0$ such that
\[
\sup_{\epsilon\in\left(0,\epsilon_{0}\right)}\left\Vert h_{M-\frac{1}{2}}^{\epsilon}\right\Vert _{X_{T_{\epsilon}}}\leq\bar{C}
\]
holds. Then there is $\epsilon_{1}\in\left(0,\epsilon_{0}\right]$
such that $c_{A}^{\epsilon}\left(.,t\right)$ satisfies Assumption
\ref{assu:Spektral} for all $t\in\left[0,T_{\epsilon}\right]$ and
$\epsilon\in\left(0,\epsilon_{1}\right)$, where the appearing constant
$C^{*}$ does not depend on $\epsilon$, $T_{\epsilon}$, $h_{M-\frac{1}{2}}^{\epsilon}$
or $\bar{C}$.
\end{lem}

The following technical proposition is an essential ingredient for many estimates.
Essentially it states that an error $R$ can be split into a multiple of $\theta_{0}'$ plus perturbation terms that is of higher order
in $\epsilon$. 
\begin{prop}
\label{Rbar-Zerl} Let $\epsilon_{0}>0$, $T\in\left(0,T_{0}\right]$
and a family $\left(T_{\epsilon}\right)_{\epsilon\in\left(0,\epsilon_{0}\right)}\subset\left(0,T\right]$
be given. Let Assumption \ref{assu:Main-est} hold true for $c_{A}=c_{A}^{\epsilon}$
and we assume that there is some $\bar{C}\geq1$ such that
\[
\sup_{\epsilon\in\left(0,\epsilon_{0}\right)}\left\Vert h_{M-\frac{1}{2}}^{\epsilon}\right\Vert _{X_{T_{\epsilon}}}\leq\bar{C}.
\]
We denote
\[
I_{\epsilon}^{s,t}:=\left(-\frac{\delta}{\epsilon}-h_{A}^{\epsilon}(s,t),\frac{\delta}{\epsilon}-h_{A}^{\epsilon}(s,t)\right)\quad \text{and} \quad \beta(s,t):=\left\Vert \theta_{0}'\right\Vert _{L^{2}\left(I_{\epsilon}^{s,t}\right)}^{-1}
\]
for $\epsilon\in\left(0,\epsilon_{0}\right)$, $s\in\mathbb{T}^{1}$
and $t\in\left[0,T_{\epsilon}\right]$. Then there is some $\epsilon_{1}\in\left(0,\epsilon_{0}\right]$
and there exist $Z\in L^{2}\left(0,T_{\epsilon};H^{1}\left(\mathbb{T}^{1}\right)\right)$,
$F_{2}^{\mathbf{R}}\in L^{2}\left(0,T_{\epsilon};H^{1}\left(\Gamma_{t}(\delta)\right)\right)$
and smooth $F_{1}^{\mathbf{R}}\colon \Gamma\left(\delta;T_{\epsilon}\right)\rightarrow\mathbb{R}$
such that 
\begin{equation}
R(x,t)=\epsilon^{-\frac{1}{2}}Z\left(S(x,t),t\right)\left(\beta\left(S(x,t),t\right)\theta_{0}'\left(\rho(x,t)\right)+F_{1}^{\mathbf{R}}(x,t)\right)+F_{2}^{\mathbf{R}}(x,t)\label{eq:Zerl}
\end{equation}
for almost all $(x,t)\in\Gamma\left(\delta;T_{\epsilon}\right)$
and all $\epsilon\in\left(0,\epsilon_{1}\right)$. 
Furthermore, there exist $C(K)$, $C>0$ independent of
$\epsilon$, $T_{\epsilon}$, $h_{M-\frac{1}{2}}^{\epsilon}$ and
$\bar{C}$ such that $\Vert \beta\Vert _{L^{\infty}(\mathbb{T}^{1}\times (0,T_{\epsilon}))}\leq C$ and
\begin{align}
\left\Vert F_{2}^{\mathbf{R}}\right\Vert _{L^{2}\left(\Gamma\left(\delta;T_{\epsilon}\right)\right)}^{2}&\leq C(K)\epsilon^{2M+1},\label{eq:Zerl-F2R}\\
\left\Vert Z\right\Vert _{L^{2}\left(0,T_{\epsilon};H^{1}\left(\mathbb{T}^{1}\right)\right)}^{2}+\left\Vert F_{2}^{\mathbf{R}}\right\Vert _{L^{2}\left(0,T_{\epsilon};H^{1}\left(\Gamma_{t}(\delta)\right)\right)}^{2}&\leq C(K)\epsilon^{2M-1}\label{eq:Zerl-main}
\end{align}
for all $\epsilon\in\left(0,\epsilon_{1}\right)$ as well as
\begin{equation}
\sup_{t\in\left[0,T_{\epsilon}\right]}\sup_{s\in\mathbb{T}^{1}}\int_{I_{\epsilon}^{s,t}}\left(\left|F_{1}^{\mathbf{R}}\left(\rho,s,t\right)\right|^{2}+\left|\partial_{\rho}F_{1}^{\mathbf{R}}\left(\rho,s,t\right)\right|^{2}\right)J^{\epsilon}\left(\rho,s,t\right)\d\rho\leq C(K)\epsilon^{2}\label{eq:Zerl-F1R}
\end{equation}
for all $\epsilon\in\left(0,\epsilon_{1}\right)$, where
\[
F_{1}^{\mathbf{R}}\left(\rho,s,t\right):=F_{1}^{\mathbf{R}}\left(X\left(\epsilon\left(\rho+h_{A}^{\epsilon}(s,t)\right),s,t\right)\right)
\]
 for $X$ as in (\ref{eq:diffeo}) and 
$
J^{\epsilon}(\rho,s,t):=J(\epsilon(\rho+h_{A}^{\epsilon}(s,t)),s,t)
$
 with $J(r,s,t):=\det\left(D_{\left(r,s\right)}X\right)(r,s,t)$. 
\end{prop}

\begin{proof}
Let $\epsilon_{1}$ be chosen as in Lemma \ref{lem:spekholds}. Then
$c_{A}^{\epsilon}$ satisfies Assumption \ref{assu:Spektral} for
all $\epsilon\in\left(0,\epsilon_{1}\right)$. Let
\[
\Lambda_{\epsilon}(t):=\int_{\Gamma_{t}(\delta)}\epsilon\left|\nabla R\right|^{2}+\epsilon^{-1}f''\left(c_{A}^{\epsilon}\right)\left(R\right){}^{2}\d x.
\]
Then (\ref{eq:Main-est-c}) and (\ref{eq:fnachunten}) imply
\begin{equation}
\int_{0}^{T_{\epsilon}}\Lambda_{\epsilon}(t)\leq CK^{2}\epsilon^{2M}\qquad \text{for all }\eps\in (0,\eps_1).\label{eq:lambdae}
\end{equation}
Hence for each $t\in\left[0,T_{\epsilon}\right]$, Lemma \ref{Chen-hilf}
implies the existence of functions $Z\left(.,t\right)\in H^{1}\left(\mathbb{T}^{1}\right)$,
$F_{1}^{\mathbf{R}}\left(.,t\right):\Gamma_{t}(\delta)\rightarrow\mathbb{R}$
and $F_{2}^{\mathbf{R}}\left(.,t\right)\in H^{1}\left(\Gamma_{t}(\delta)\right)$
such that (\ref{eq:Zerl}) holds for almost all $x\in\Gamma_{t}(\delta)$
and all $\epsilon\in\left(0,\epsilon_{1}\right)$. Moreover,
\[
\left\Vert F_{2}^{\mathbf{R}}\left(.,t\right)\right\Vert _{L^{2}\left(\Gamma_{t}(\delta)\right)}^{2}\leq C\left(\epsilon\Lambda_{\epsilon}(t)+\epsilon^{2}\left\Vert R\left(.,t\right)\right\Vert _{L^{2}\left(\Gamma_{t}(\delta)\right)}^{2}\right)
\]
\[
\left\Vert Z\left(.,t\right)\right\Vert _{H^{1}\left(\mathbb{T}^{1}\right)}^{2}+\left\Vert F_{2}^{\mathbf{R}}\left(.,t\right)\right\Vert _{H^{1}\left(\Gamma_{t}(\delta)\right)}^{2}\le C\left(\left\Vert R\left(.,t\right)\right\Vert _{L^{2}\left(\Gamma_{t}(\delta)\right)}^{2}+\frac{\Lambda_{\epsilon}(t)}{\epsilon}\right)
\]
for all $\epsilon\in\left(0,\epsilon_{2}\right)$. Note that $C>0$ is independent of $\epsilon$, $T_{\epsilon}$ and $\bar{C}$
since $C^\ast$ in Lemma~\ref{lem:spekholds} is independent of these quantities as well. Since
$\left\Vert R\right\Vert _{L^{2}\left(\Omega_{T_{\epsilon}}\right)}^{2}\leq CK^{2}\epsilon^{2M-1}$
and (\ref{eq:lambdae}) hold true due to (\ref{eq:Main-est}), integration
over $\left(0,T_{\epsilon}\right)$ yields (\ref{eq:Zerl-F2R}) and
(\ref{eq:Zerl-main}).
Finally, (\ref{eq:Zerl-F1R}) is a direct consequence of (\ref{eq:Chenneu4}).
\end{proof}
Now we show the main estimate for $\tilde{\mathbf{w}}_{1}^{\epsilon}$
:
\begin{lem}
\label{Wichtig} Let $\epsilon_{0}>0$, $T'\in\left(0,T_{0}\right]$
and a family $\left(T_{\epsilon}\right)_{\epsilon\in\left(0,\epsilon_{0}\right)}\subset\left(0,T'\right]$
be given. Let Assumption \ref{assu:Main-est} hold true for $c_{A}=c_{A}^{\epsilon}$
and we assume that there is $\bar{C}\geq1$ such that
\begin{equation}
\sup_{\epsilon\in\left(0,\epsilon_{0}\right)}\left\Vert h_{M-\frac{1}{2}}^{\epsilon}\right\Vert _{X_{T_{\epsilon}}}\leq\bar{C}.\label{eq:bes}
\end{equation}
Then there exists a constant $C(K)>0$, which is
independent of $\epsilon$, $T_{\epsilon}$, $h_{M-\frac{1}{2}}^{\epsilon}$
and $\bar{C}$, and some $\epsilon_{1}\in\left(0,\epsilon_{0}\right)$
such that 
\begin{align}
  \left\Vert \tilde{\mathbf{w}}_{1}^{\epsilon}\right\Vert _{L^{2}\left(0,T;H^{1}(\Omega)\right)} & \leq C(K)\epsilon^{M-\frac{1}{2}}\quad \text{for all }\epsilon\in\left(0,\epsilon_{1}\right), T\in\left(0,T_{\epsilon}\right]. \label{eq:w1epsab}
\end{align}
\end{lem}

\begin{proof}
First of all, we note that there exists $\epsilon_{1}\in\left(0,\epsilon_{0}\right]$,
which depends on $\bar{C}$, such that
\begin{equation}
\left|\frac{d_{\Gamma}(x,t)}{\epsilon}-h_{A}^{\epsilon}\left(S(x,t),t\right)\right|\geq\frac{\delta}{2\epsilon}\label{eq:epsab1}
\end{equation}
for all $(x,t)\in\Gamma\left(2\delta;T_{\epsilon}\right)\backslash\Gamma\left(\delta;T_{\epsilon}\right)$
and $\epsilon\in\left(0,\epsilon_{1}\right)$ because of $X_{T}\hookrightarrow C^{0}\left(\left[0,T\right];C^{1}\left(\mathbb{T}^{1}\right)\right)$
and (\ref{eq:bes}). After possibly choosing $\epsilon_{1}>0$ smaller,
we may ensure that
\begin{equation}
\left|\theta_{0}\left(\rho(x,t)\right)-\chi_{\Omega^{+}}(x,t)+\chi_{\Omega^{-}}(x,t)\right|+\left|\theta_{0}'\left(\rho(x,t)\right)\right|\leq C_{1}e^{-C_{2}\frac{\delta}{2\epsilon}}\label{eq:expab}
\end{equation}
holds true for all $(x,t)\in\Gamma\left(2\delta;T_{\epsilon}\right)\backslash\Gamma\left(\delta;T_{\epsilon}\right)$
and $\epsilon\in\left(0,\epsilon_{1}\right)$, as a consequence of
(\ref{eq:optimopti}), where $C_{1},C_{2}>0$ can be chosen independently
of $\epsilon_{1}$. As a last condition on $\epsilon_{1}$ we impose
that $\epsilon_{1}^{M-\frac{3}{2}}\leq\frac{1}{\bar{C}}$, which yields
\begin{equation}
\epsilon^{M-\frac{3}{2}}\left\Vert h_{M-\frac{1}{2}}^{\epsilon}\right\Vert _{X_{T_{\epsilon}}}\leq1\label{eq:epsab2}
\end{equation}
 for all $\epsilon\in\left(0,\epsilon_{1}\right)$. 

Since $\twe$ is a weak solution to (\ref{eq:w1})\textendash (\ref{eq:w13})
in $\Omega_{T_{\epsilon}}$, we have due to Theorem \ref{exinstokes}
\[
\left\Vert \twe\right\Vert _{L^{2}\left(0,T;H^{1}(\Omega)\right)}\leq C\left\Vert \mathbf{f}^{\epsilon}\right\Vert _{L^{2}\left(0,T;V_{0}'(\Omega)\right)}
\]
for all $T\in\left(0,T_{\epsilon}\right)$, where $\mathbf{f}^{\epsilon}$
is given as in (\ref{eq:fepsh}). Let in the following $T\in\left(0,T_{\epsilon}\right]$
and $\psi\in L^{2}\left(0,T_{\epsilon};V_{0}(\Omega)\right)$,
$\psi\neq0$.
As a starting point, we decompose
\begin{align}\nonumber
  &\int_{\Omega_{T}}\epsilon\left(\nabla c_{A}^{\epsilon}-\mathbf{h}\right)\otimes\nabla R:\nabla\psi\d(x,t)=\\
  & \int_{\Gamma\left(\delta,T\right)}\epsilon\left(\nabla c_{A}^{\epsilon}-\mathbf{h}\right)\otimes\nabla R:\nabla\psi\d(x,t) +\int_{\Omega_{T}\backslash\Gamma\left(\delta;T\right)}\epsilon\left(\nabla c_{A}^{\epsilon}-\mathbf{h}\right)\otimes\nabla R:\nabla\psi\d(x,t)\label{eq:splitup}
\end{align}
and estimate the two integrals on the right hand side separately.
The second summand in $\mathbf{f}^{\epsilon}$ may be treated analogously.

To estimate the second integral in (\ref{eq:splitup}), note that
$c_{I},\nabla_{x}c_{k},\partial_{\rho}c_{k},\nabla^{\Gamma}h_{i}\in L^{\infty}\left(\Gamma(2\delta)\right),$
$i\in\left\{ 1,\ldots,M+1\right\} $, $k\in\left\{ 0,\ldots,M+1\right\} $,
$c_{O,\mathbf{B}},\nabla c_{O,\mathbf{B}}\in L^{\infty}\left(\Omega_{T_{0}}^{\pm}\right)$
and that we may employ (\ref{eq:expab}). Thus, $\left|\nabla c_{A}^{\epsilon}(x,t)-\mathbf{h}(x,t)\right|\leq C_{1}\left(1+\frac{1}{\epsilon}e^{-C_{2}\frac{\delta}{2\epsilon}}\right)$
for all $(x,t)\in\Omega_{T_{\epsilon}}\backslash\Gamma\left(\delta;T_{\epsilon}\right)$
and $\epsilon\in\left(0,\epsilon_{1}\right)$ and we may estimate
\begin{align*}
\int_{0}^{T}\int_{\Omega\backslash\Gamma_{t}(\delta)}\left|\epsilon\left(\nabla c_{A}^{\epsilon}-\mathbf{h}\right)\otimes\nabla R:\nabla\psi\right|\d x\d t & \leq C\epsilon\left\Vert \nabla R\right\Vert _{L^{2}\left(0,T;L^{2}\left(\Omega\backslash\Gamma_{t}(\delta)\right)\right)}\left\Vert \psi\right\Vert _{L^{2}\left(0,T;H^{1}(\Omega)\right)}\\
 & \leq C(K)\epsilon^{M+\frac{1}{2}}\left\Vert \psi\right\Vert _{L^{2}\left(0,T;H^{1}(\Omega)\right)}
\end{align*}
for $T\in\left(0,T_{\epsilon}\right)$, where we used (\ref{eq:Main-est-a})
in the last inequality. Dealing with the first integral on the right
hand side of (\ref{eq:splitup}) is more complicated. We compute
\begin{align}
&\int_{\Gamma(\delta;T)}  \epsilon\left(\nabla c_{A}^{\epsilon}-\mathbf{h}\right)\otimes\nabla R:\nabla\psi\d(x,t)\nonumber \\
 & =\int_{\Gamma(\delta;T)}\theta_{0}'(\rho)\left(\mathbf{n}-\epsilon\left(\sum_{i=0}^{M}\epsilon^{i}\nabla^{\Gamma}h_{i+1}\right)\right)\otimes\nabla R:\nabla\psi\d(x,t)\nonumber \\
 & \quad+\int_{\Gamma(\delta;T)}\epsilon\left(\nabla\left(c_{A}^{\epsilon}-\theta_{0}(\rho)\right)-\left(\mathbf{h}+\theta_{0}^{'}(\rho)\epsilon^{M-\frac{3}{2}}\nabla^{\Gamma}h_{M-\frac{1}{2}}^{\epsilon}\right)\right)\otimes\nabla R:\nabla\psi\d(x,t),\label{eq:1term}
\end{align}
where we employ the shortened notations $\rho=\rho(x,t)$
and $\mathbf{n}=\mathbf{n}\left(S(x,t),t\right)$. 
As
$$
\left(c_{A}^{\epsilon}-\theta_{0}\circ\rho\right)(x,t)=\sum_{i=1}^{M+1}\epsilon^{i}c_{i}(\rho(x,t),x,t)
$$
for all $(x,t)\in\Gamma\left(\delta;T_{\epsilon}\right)$
we find that there exists some $C>0$ independent of $K$ and $\epsilon$
such that 
\[
\left|\nabla\left(c_{A}^{\epsilon}-\theta_{0}(\rho)\right)-\left(\mathbf{h}+\theta_{0}^{'}(\rho)\epsilon^{M-\frac{3}{2}}\nabla^{\Gamma}h_{M-\frac{1}{2}}^{\epsilon}\right)\right|\leq C
\]
for all $(x,t)\in\Gamma\left(\delta;T_{\epsilon}\right)$.
Thus 
\begin{align*}
\int_{0}^{T}\int_{\gt} & \left|\epsilon\left(\nabla\left(c_{A}^{\epsilon}-\theta_{0}(\rho)\right)-\left(\mathbf{h}+\theta_{0}^{'}(\rho)\epsilon^{M-\frac{3}{2}}\nabla^{\Gamma}h_{M-\frac{1}{2}}^{\epsilon}\right)\right)\otimes\nabla R:\nabla\psi\right|\d x\d t\\
 & \leq C\epsilon\left\Vert \nabla R\right\Vert _{L^{2}\left(\Gamma_(\delta,T)\right)}\left\Vert \psi\right\Vert _{L^{2}\left(0,T;H^{1}(\Omega)\right)} \leq C(K)\epsilon^{M-\frac{1}{2}}\left\Vert \psi\right\Vert _{L^{2}\left(0,T;H^{1}(\Omega)\right)}
\end{align*}
for $T\in\left(0,T_{\epsilon}\right]$ and $\epsilon\in\left(0,\epsilon_{1}\right)$,
by (\ref{eq:Main-est}). 

Using the boundedness of $\theta_{0}'$ in $L^{\infty}\left(\mathbb{R}\right)$
and that of $\nabla^{\Gamma}h_{i}$ in $L^{\infty}\left(\Gamma(2\delta)\right)$,
$i\in\left\{ 1,\ldots,M+1\right\} $, we also find 
\begin{align*}
  &\int_{0}^{T}\int_{\gt}\left|\theta_{0}'(\rho)\mathbf{n}\otimes\nabla^{\Gamma}R:\nabla\psi\right|\d x\d t\\
  & \qquad \leq C\left\Vert \nabla^{\Gamma}R\right\Vert _{L^{2}\left(\Gamma(\delta,T)\right)}\left\Vert \psi\right\Vert _{L^{2}\left(0,T;H^{1}(\Omega)\right)} \leq C(K)\epsilon^{M-\frac{1}{2}}\left\Vert \psi\right\Vert _{L^{2}\left(0,T;H^{1}(\Omega)\right)},
\\
&  \int_{0}^{T}\!\int_{\gt}\!\left|\epsilon\theta_{0}'(\rho)\left(\sum_{i=0}^{M}\epsilon^{i}\nabla^{\Gamma}h_{i+1}\right)\!\otimes\nabla R:\nabla\psi\right|\!\d x\d t\\
  & \qquad \leq C\epsilon\left\Vert \nabla R\right\Vert _{L^{2}\left(\Gamma(\delta,T)\right)}\left\Vert \psi\right\Vert _{L^{2}\left(0,T;H^{1}(\Omega)\right)} \leq C(K)\epsilon^{M-\frac{1}{2}}\left\Vert \psi\right\Vert _{L^{2}\left(0,T;H^{1}(\Omega)\right)}
\end{align*}
by (\ref{eq:Main-est}). Hence, plugging these results into (\ref{eq:1term}),
we obtain
\begin{align*}
\left|\,\int_{\Gamma\left(\delta;T\right)}\epsilon\left(\nabla c_{A}^{\epsilon}-\mathbf{h}\right)\otimes\nabla R:\nabla\psi\d(x,t)\right| &\leq  \mathcal{I}+C(K)\epsilon^{M-\frac{1}{2}}\left\Vert \psi\right\Vert _{L^{2}\left(0,T;H^{1}(\Omega)\right)}
\end{align*}
for $T\in\left(0,T_{\epsilon}\right)$ and $\epsilon\in\left(0,\epsilon_{1}\right)$, where
\[
\text{\ensuremath{\mathcal{I}}}:=\left|\,\int_{\Gamma\left(\delta;T\right)}\theta_{0}'(\rho)\mathbf{n}\otimes\mathbf{n}\partial_{\mathbf{n}}R:\nabla\psi\d(x,t)\right|.
\]
 Since $\psi\in V_{0}$,
we have $\operatorname{div}\psi=0$, which implies by (\ref{eq:surfdivdecomp})
that $\operatorname{div}^{\Gamma}\psi=-\mathbf{n}\otimes\mathbf{n}:\nabla\psi$
holds. As the assumptions of Proposition \ref{Rbar-Zerl} are satisfied,
we may estimate $\mathcal{I}$ using (\ref{eq:Zerl}) and obtain
\begin{align*}
\mathcal{I} & =\left|\,\int_{\Gamma\left(\delta;T\right)}\theta_{0}'(\rho)\partial_{\mathbf{n}}\left(\epsilon^{-\frac{1}{2}}Z(S(x,t),t)\left(\beta(S(x,t),t)\theta_{0}'(\rho)+F_{1}^{\mathbf{R}}\right)+F_{2}^{\mathbf{R}}\right)\operatorname{div}^{\Gamma}\psi\d(x,t)\right|\\
 & \leq\left|\,\int_{0}^{T}\int_{\Gamma_{t}(\delta)}\frac{1}{2}\partial_{\mathbf{n}}\left(\theta_{0}'(\rho)^{2}\right)\epsilon^{-\frac{1}{2}}Z(S(x,t),t)\beta(S(x,t),t)\mbox{div}^{\Gamma}\psi\d x\d t\right|\\
 & \quad+C_{1}\left|\,\int_{0}^{T}\int_{\mathbb{T}^{1}}\int_{-\frac{\delta}{\epsilon}-h_{A}^{\epsilon}(s,t)}^{\frac{\delta}{\epsilon}-h_{A}^{\epsilon}(s,t)}\theta_{0}'(\rho)\epsilon^{-\frac{1}{2}}Z(s,t)\partial_{\rho}F_{1}^{\mathbf{R}}(\rho,s,t)\mbox{div}^{\Gamma}\psi J^{\epsilon}(\rho,s,t)\d\rho\d s\d t\right|\\
 & \quad+C_{2}\left\Vert F_{2}^{\mathbf{R}}\right\Vert _{L^{2}\left(0,T;H^{1}\left(\Gamma_{t}(\delta)\right)\right)}\left\Vert \psi\right\Vert _{L^{2}\left(0,T;H^{1}\left(\Gamma_{t}(\delta)\right)\right)}\\
 & =:\mathcal{J}_{1}+\mathcal{J}_{2}+\mathcal{J}_{3}.
\end{align*}
Here we used the same notations as in Proposition~\ref{Rbar-Zerl}
and in the first lines the short notation $\rho=\rho(x,t)$.
Now (\ref{eq:Zerl-main}) implies 
\[
\mathcal{J}_{3}\leq C(K)\epsilon^{M-\frac{1}{2}}\left\Vert \psi\right\Vert _{L^{2}\left(0,T;H^{1}\left(\Gamma_{t}(\delta)\right)\right)}
\]
 and we may estimate $\mathcal{J}_{2}$ by
\begin{align*}
\mathcal{J}_{2} & \leq C\epsilon^{-1}\left\Vert \psi\right\Vert _{L^{2}\left(0,T;H^{1}\left(\Gamma_{t}(\delta)\right)\right)}\left(\int_{0}^{T}\int_{\mathbb{T}^{1}}Z(s,t)^{2}\int_{-\frac{\delta}{\epsilon}-h_{A}^{\epsilon}(s,t)}^{\frac{\delta}{\epsilon}-h_{A}^{\epsilon}(s,t)}\left(\partial_{\rho}F_{1}^{\mathbf{R}}\left(\rho,s,t\right)\right)^{2}J^{\epsilon}\d\rho\d s\d t\right)^{\frac{1}{2}}\\
 & \leq C(K)\epsilon^{M-\frac{1}{2}}\left\Vert \psi\right\Vert _{L^{2}\left(0,T;H^{1}\left(\Gamma_{t}(\delta)\right)\right)},
\end{align*}
where we used (\ref{eq:Zerl-F1R}) in the last line. To treat the remaining integral, we may use Lemma~\ref{lem:divlm}
to get 
\begin{align*}
\mathcal{J}_{1} & \leq\left|\int_{0}^{T}\int_{\Gamma_{t}(\delta)}\frac{1}{2}\nabla^{\Gamma}\left(\partial_{\mathbf{n}}\left(\theta_{0}'(\rho)^{2}\right)\epsilon^{-\frac{1}{2}}Z\left(S(x,t),t\right)\beta\left(S(x,t),t\right)\right)\cdot\psi\d x\d t\right|\\
 & \quad+\left|\int_{0}^{T}\int_{\Gamma_{t}(\delta)}\partial_{\mathbf{n}}\left(\theta_{0}'(\rho)^{2}\right)\epsilon^{-\frac{1}{2}}Z\left(S(x,t),t\right)\beta\left(S(x,t),t\right)\psi\cdot\mathbf{n}\kappa(x,t)\d x\d t\right|\\
 & \quad+C\sum_{\pm}\int_{0}^{T}\int_{\mathbb{T}^{1}}\left|\partial_{\rho}\left(\theta_{0}'\left(\frac{\pm\delta}{\epsilon}-h_{A}^{\epsilon}(s,t)\right)^{2}\right)\epsilon^{-\frac{3}{2}}Z(s,t)\beta(s,t)\psi\left(\pm \delta,s,t\right)\right|\d s\d t\\
 & :=\mathcal{J}_{1}^{1}+\mathcal{J}_{1}^{2}+\mathcal{J}_{1}^{3,+}+\mathcal{J}_{1}^{3,-}.
\end{align*}
Now 
\begin{align*}
\mathcal{J}_{1}^{3,\pm} & \leq C_{1}\epsilon^{-\frac{3}{2}}e^{-C_{2}\frac{\delta}{2\epsilon}}\int_{0}^{T}\int_{\mathbb{T}^{1}}\left|Z(s,t)\right|\sup_{r\in\left[-\delta,\delta\right]}\left|\psi(r,s,t)\right|\d s\d t\\
 & \leq C(K)\epsilon^{M-\frac{1}{2}}\left\Vert \psi\right\Vert _{L^{2}\left(0,T;H^{1}\left(\Gamma_{t}(\delta)\right)\right)},
\end{align*}
where we used (\ref{eq:expab}) and the uniform bound on $\beta$
in the first step and $H^{1}\left(\Gamma_{t}(\delta)\right)\hookrightarrow L^{2,\infty}\left(\Gamma_{t}(\delta)\right)$
(cf.\ Lemma \ref{L4inf}) in the second step.  For $\mathcal{J}_{1}^{2}$, we use integration
by parts and get 
\begin{align*}
\mathcal{J}_{1}^{2} & \leq\left|\int_{0}^{T}\int_{\Gamma_{t}(\delta)}\left(\theta_{0}'(\rho)\right)^{2}\epsilon^{-\frac{1}{2}}Z\left(S(x,t),t\right)\beta\left(S(x,t),t\right)\partial_{\mathbf{n}}\psi\cdot\mathbf{n}\left(S(x,t),t\right)\kappa(x,t)\d x\d t\right|\\
 & \quad+C\int_{0}^{T}\int_{\Gamma_{t}(\delta)}\left|\left(\theta_{0}'(\rho)\right)^{2}\epsilon^{-\frac{1}{2}}Z\left(S(x,t),t\right)\beta\left(S(x,t),t\right)\psi\right|\d x\d t\\
 & \quad+C(K)e^{-C_{2}\frac{\delta}{2\epsilon}}\left\Vert \psi\right\Vert _{L^{2}\left(0,T;H^{1}\left(\Gamma_{t}(\delta)\right)\right)}\\
 & \leq C\epsilon^{-\frac{1}{2}}\left\Vert Z\right\Vert _{L^{2}\left(0,T;H^{1}\left(\mathbb{T}^{1}\right)\right)}\left\Vert \psi\right\Vert _{L^{2}\left(0,T;H^{1}\left(\Gamma_{t}(\delta)\right)\right)}\epsilon^{\frac{1}{2}}\left\Vert \left(\theta_{0}'\right)^{2}\right\Vert _{L^{2}\left(\mathbb{R}\right)}\\
 & \quad+C(K)e^{-C_{2}\frac{\delta}{2\epsilon}}\left\Vert \psi\right\Vert _{L^{2}\left(0,T;H^{1}\left(\Gamma_{t}(\delta)\right)\right)}\\
 & \leq C(K)\epsilon^{M-\frac{1}{2}}\left\Vert \psi\right\Vert _{L^{2}\left(0,T;H^{1}\left(\Gamma_{t}(\delta)\right)\right)},
\end{align*}
where the exponential decaying term in the first inequality is a consequence
of the appearing boundary integral, which may be estimated as in the
case of $\mathcal{J}_{1}^{3,\pm}$. Moreover, we used a change of
variables $r\mapsto\frac{r}{\epsilon}-h_{A}^{\epsilon}$ in the second
step and (\ref{eq:Zerl-main}) in the last step.

Now we discuss  $\mathcal{J}_{1}^{1}$ \textendash{} the last term
we need to estimate. Note that by the definition of $\beta$ in Proposition
\ref{Rbar-Zerl}, we have
\begin{align*}
\nabla^{\Gamma}\beta(s,t) & =-\frac{1}{\left\Vert \theta_{0}'\right\Vert _{L^{2}(I_{\epsilon}^{s,t})}^{2}}\int_{-\frac{\delta}{\epsilon}-h_{A}^{\epsilon}(s,t)}^{\frac{\delta}{\epsilon}-h_{A}^{\epsilon}(s,t)}\frac{1}{2}\frac{d}{d\rho}\left(\theta_{0}'(\rho)^{2}\right)\d\rho\left(-\nabla^{\Gamma}h_{A}^{\epsilon}\right) \leq C_{1}e^{-C_{2}\frac{\delta}{2\epsilon}}
\end{align*}
for all $\epsilon\in\left(0,\epsilon_{1}\right)$, due to (\ref{eq:expab})
and $\epsilon^{M-\frac{3}{2}}\left\Vert h_{M-\frac{1}{2}}^{\epsilon}\right\Vert _{X_{T_{\epsilon}}}\leq1$,
cf.\ (\ref{eq:epsab2}). Thus, we compute
\begin{align*}
\mathcal{J}_{1}^{1} & \leq\left|\int_{0}^{T}\int_{\Gamma_{t}(\delta)}\frac{1}{2}\partial_{\mathbf{n}}\nabla^{\Gamma}\left(\theta_{0}'\left(\rho(x,t)\right)^{2}\right)\epsilon^{-\frac{1}{2}}Z\left(S(x,t),t\right)\beta\left(S(x,t),t\right)\cdot\psi\d x\d t\right|\\
 & \quad+\left|\int_{0}^{T}\int_{\Gamma_{t}(\delta)}\frac{1}{2}\left[\partial_{\mathbf{n}},\nabla^{\Gamma}\right]\left(\theta_{0}'\left(\rho(x,t)\right)^{2}\right)\epsilon^{-\frac{1}{2}}Z\left(S(x,t),t\right)\beta\left(S(x,t),t\right)\cdot\psi\d x\d t\right|\\
 & \quad+\left|\int_{0}^{T}\int_{\Gamma_{t}(\delta)}\frac{1}{2}\partial_{\mathbf{n}}\left(\theta_{0}'\left(\rho(x,t)\right)^{2}\right)\epsilon^{-\frac{1}{2}}\nabla^{\Gamma}\left(Z\left(S(x,t),t\right)\beta\left(S(x,t),t\right)\right)\cdot\psi\d x\d t\right|\\
 & \leq C_{1}\int_{0}^{T}\int_{\Gamma_{t}(\delta)}\left|\partial_{\rho}\left(\theta_{0}'\left(\rho(x,t)\right)^{2}\right)\nabla^{\Gamma}h_{A}^{\epsilon}\epsilon^{-\frac{1}{2}}Z\left(S(x,t),t\right)\beta\left(S(x,t),t\right)\cdot\partial_{\mathbf{n}}\psi\right|\d x\d t\\
 & \quad+C_{2}\int_{0}^{T}\int_{\Gamma_{t}(\delta)}\left|\partial_{\rho}\left(\theta_{0}'\left(\rho(x,t)\right)^{2}\right)\nabla^{\Gamma}h_{A}^{\epsilon}\epsilon^{-\frac{1}{2}}Z\left(S(x,t),t\right)\beta\left(S(x,t),t\right)\cdot\psi\right|\d x\d t\\
 & \quad+C_{3}\int_{0}^{T}\int_{\Gamma_{t}(\delta)}\left|\left(\theta_{0}'\left(\rho(x,t)\right)^{2}\right)\epsilon^{-\frac{1}{2}}\nabla^{\Gamma}Z\left(S(x,t),t\right)\beta\left(S(x,t),t\right)\cdot\partial_{\mathbf{n}}\psi\right|\d x\d t\\
 & \quad+C_{4}\int_{0}^{T}\int_{\Gamma_{t}(\delta)}\left|\left(\theta_{0}'\left(\rho(x,t)\right)^{2}\right)\epsilon^{-\frac{1}{2}}\partial_{s}Z\left(S(x,t),t\right)\beta\left(S(x,t),t\right)\psi\right|\d x\d t\\
 & \quad+C_{5}e^{-C_{6}\frac{\delta}{2\epsilon}}\left\Vert \psi\right\Vert _{L^{2}\left(0,T;H^{1}\left(\Gamma_{t}(\delta)\right)\right)}\left\Vert Z\right\Vert _{L^{2}\left(0,T;H^{1}\left(\mathbb{T}^{1}\right)\right)}\\
 & \leq C(K)\epsilon^{M-\frac{1}{2}}\left\Vert \psi\right\Vert _{L^{2}\left(0,T;H^{1}\left(\Gamma_{t}(\delta)\right)\right)}.
\end{align*}
Here we used the definition of $\left[\partial_{\mathbf{n}},\nabla^{\Gamma}\right]$
in the first estimate (cf.\ (\ref{eq:comuutatordef})), integration
by parts, (\ref{eq:commutator}) and the exponential decay of $\nabla^{\Gamma}\beta$
and the boundary terms in the second step. In the third step we again
used $\epsilon^{M-\frac{3}{2}}\left\Vert h_{M-\frac{1}{2}}^{\epsilon}\right\Vert _{X_{T_{\epsilon}}}\leq1$.
This concludes the proof.
\end{proof}
Regarding the fractional order terms, we have the following bounds,
which are a result of \cite{NSCH2}, Theorem 3.15. This enables us
to use (\ref{eq:w1epsab}), whenever Assumption \ref{assu:Main-est}
is satisfied.
\begin{lem}
\label{lem:hM-1}Let $\epsilon_{0}\in\left(0,1\right)$. If Assumption
\ref{assu:Main-est} holds for $c_{A}=c_{A}^{\epsilon}$, then there exist
$\epsilon_{1}\in\left(0,\epsilon_{0}\right]$ and a constant $C(K)>0$
independent of $\epsilon$ such that 
\begin{equation}
\big\Vert h_{M-\frac{1}{2}}^{\epsilon}\big\Vert _{X_{T_{\epsilon}}}+
\big\Vert \mu_{M-\frac{1}{2}}^{\pm,\epsilon}\big\Vert _{Z_{T_{\epsilon}}}+\big\Vert \mathbf{v}_{M-\frac{1}{2}}^{\pm,\epsilon}\big\Vert _{L^{6}\left(0,T_{\epsilon};H^{2}\left(\Omega^{\pm}(t)\right)\right)}\leq C(K)\label{eq:muv0,5est}
\end{equation}
for all $\epsilon\in\left(0,\epsilon_{1}\right)$, where $Z_{T_{\epsilon}}:=L^{2}\left(0,T_{\epsilon};H^{2}(\Omega^{\pm}(t))\right)\cap L^{6}\left(0,T_{\epsilon};H^{1}(\Omega^{\pm}(t))\right)$.
\end{lem}

As a direct consequence of (\ref{eq:muv0,5est}) and $X_{T}\hookrightarrow C^{0}(\left[0,T\right];C^{1}(\mathbb{T}^{1}))$,
we remark 
\begin{equation}
\left\Vert h_{A}^{\epsilon}\right\Vert _{C^{0}\left(0,T_{\epsilon};C^{1}(\mathbb{T}^1)\right)}\le C(K)\label{eq:haepglm}
\end{equation}
for all $\epsilon\in\left(0,\epsilon_{1}\right)$. Finally, concerning
the relation between $c_{I}$ and $c_{O,\mathbf{B}}$, we have in
the case that $\epsilon_{0}\in\left(0,1\right)$ and Assumption \ref{assu:Main-est}
holds for $c_{A}=c_{A}^{\epsilon}$ that
\begin{equation}
\left\Vert D_{x}^{l}\left(c_{I}-c_{O,\mathbf{B}}\right)\right\Vert _{L^{\infty}\left(\Gamma\left(2\delta;T_{\epsilon}\right)\backslash\Gamma\left(\delta;T_{\epsilon}\right)\right)}\leq C(K)e^{-\frac{C}{\epsilon}}\label{eq:matching}
\end{equation}
for $l\in\left\{ 0,1\right\} $ and constants $C(K),C>0$.
This is discussed in \cite[Corollary~4.9]{NSCH2}.

\subsection{Auxiliary Results\label{sec:Auxiliary-Results}}

Without repeating it, we will consider the following assumptions throughout
this section.
\begin{assumption}
\label{assu:Auxiliary}We assume that Assumption~\ref{assu:Main-est} holds true
holds for $c_{A}=c_{A}^{\epsilon}$, $\epsilon_{0}\in\left(0,1\right)$,
$K\geq1$ and a family $\left(T_{\epsilon}\right)_{\epsilon\in\left(0,\epsilon_{0}\right)}\subset\left(0,T_{0}\right]$.
Moreover, we assume that $\epsilon_{1}\in\left(0,\epsilon_{0}\right]$
is chosen small enough, such that \eqref{eq:remcahn}\textendash \eqref{eq:rch1-rch2-Linfbdry},
the statement of Lemma~\ref{lem:hM-1}, \eqref{eq:w1epsab} and \eqref{eq:matching}
hold true.

Finally, we denote $R:=c^{\epsilon}-c_{A}^{\epsilon}$.
\end{assumption}

The following proposition guarantees that Lemma~\ref{lem:energy} may be used.
\begin{prop}
\label{prop:Energyholds}Let $\epsilon_{0}\in\left(0,1\right)$ and
$\psi_{0}^{\epsilon}:\Omega\rightarrow\mathbb{R}$ be a smooth function
satisfying the inequality $\left\Vert \psi_{0}^{\epsilon}\right\Vert _{C^{1}(\Omega)}\leq C_{\psi_{0}}\epsilon^{M}$
for $\epsilon\in\left(0,\epsilon_{0}\right)$. Moreover let $c_{0}^{\epsilon}(x):=c_{A}^{\epsilon}\left(x,0\right)+\psi_{0}^{\epsilon}(x)$
for all $x\in\Omega$. Then there is some $\tilde{\epsilon}\in\left(0,\epsilon_{0}\right]$
and a constant $C_{0}>0$ which only depends on $\tilde{\epsilon}$,
$C_{\psi_{0}}$ and $\sup_{\epsilon\in\left(0,\epsilon_{0}\right)}\left\Vert c_{A}^{\epsilon}\left(x,0\right)\right\Vert _{L^{\infty}(\Omega)}$,
such that
\[
E^{\epsilon}\left(c_{0}^{\epsilon}\right)\leq C_{0},\quad\left\Vert c_{0}^{\epsilon}\right\Vert _{L^{\infty}(\Omega)}\leq C_{0}\qquad\text{for all }\epsilon\in\left(0,\tilde{\epsilon}\right),
\]
 where $E^{\epsilon}$
is given as in (\ref{eq:gb}).
\end{prop}
\begin{proof}
For simplicity we consider $c_{0}^{\epsilon}(x)=c_{A}^{\epsilon}\left(x,0\right)$
and highlight the situations where $\psi_{0}^{\epsilon}$ would play
a role. The estimate for $\left\Vert c_{0}^{\epsilon}\right\Vert _{L^{\infty}(\Omega)}$
follows immediately by the construction of $c_{A}^{\epsilon}$. Considering
$\frac{\epsilon}{2}\int_{\Omega}\left|\nabla c_{A}^{\epsilon}\left(x,0\right)\right|^{2}\d x$
we note that $\left\Vert \nabla c_{A}^{\epsilon}\right\Vert _{L^{\infty}\left(\Omega_{T_{0}}\backslash\Gamma(2\delta)\right)}\leq C\epsilon$
and estimate 
\begin{align}
&\frac{\epsilon}{2}\int_{\Gamma_{0}(2\delta)}\left|\nabla c_{A}^{\epsilon}(x,0)\right|^{2}\d x  \leq\frac{\epsilon}{2}\int_{\Gamma_{0}(2\delta)}\left|\xi\left(d_{\Gamma}\right)\nabla c_{I}\left(x,0\right)\right|^{2}\d x\nonumber \\
 & \quad+\frac{\epsilon}{2}\int_{\Gamma_{0}(2\delta)}\left|\left(1-\xi\left(d_{\Gamma}\right)\right)\nabla c_{O,\mathbf{B}}(x,0)+\nabla\left(\xi\left(d_{\Gamma}\right)\right)\left(c_{I}-c_{O,\mathbf{B}}\right)(x,0)\right|^{2}\d x.\label{eq:nablacaepsenergy}
\end{align}
Now we have $\nabla c_{O,\mathbf{B}}\left(.,0\right)\in\mathcal{O}\left(\epsilon\right)$
in $L^{\infty}\left(\Omega^{\pm}\left(0\right)\right)$ and $c_{I},c_{O,\mathbf{B}}\in\mathcal{O}\left(1\right)$
in $L^{\infty}\left(\Gamma_{0}(2\delta)\backslash\Gamma_{0}(\delta)\right)$.
Moreover, $\rho\left(x,0\right)=\frac{d_{\Gamma}\left(x,0\right)}{\epsilon}$,
as $h_{A}^{\epsilon}\left(x,0\right)=0$, and thus
$$
\nabla\left(c_{0}(\rho(x,0),x,0)\right)=\frac{1}{\epsilon}\theta_{0}'(\rho(x,0))\cdot\mathbf{n}(x,0).
$$
In particular 
\begin{align*}
\frac{\epsilon}{2}\int_{\Gamma_{0}(2\delta)}\left|\xi\left(d_{\Gamma}\right)\nabla\left(c_{0}\left(\rho\left(x,0\right),x,0\right)\right)\right|^{2}\d x & \leq C\int_{\mathbb{T}^{1}}\int_{-\frac{2\delta}{\epsilon}}^{\frac{2\delta}{\epsilon}}\theta_{0}'(\rho)^{2}\d\rho\d s\leq C.
\end{align*}
As $\epsilon^{k}\nabla\left(c_{k}\left(\rho\left(.,0\right),.,0\right)\right)\in\mathcal{O}\left(1\right)$
in $L^{\infty}\left(\Gamma_{0}(2\delta)\right)$ for $k\geq1$,
we may use (\ref{eq:nablacaepsenergy}) to find $\frac{\epsilon}{2}\int_{\Gamma_{0}(2\delta)}\left|\nabla c_{A}^{\epsilon}\left(x,0\right)\right|^{2}\d x\leq C_{1}$.
Note that $\psi_{0}^{\epsilon}$ can be estimated uniformly  in $C^{1}(\Omega)$
and is multiplied by $\epsilon^{M}$, so would cause no troubles in
these estimates. For the second term in $E^{\epsilon}\left(c_{0}^{\epsilon}\right)$,
we compute
\begin{align*}
\frac{1}{\epsilon}\int_{\Omega^{+}\left(0\right)}f\left(c_{0}^{\epsilon}\right)\d x & =\frac{1}{\epsilon}\int_{\Omega^{+}\left(0\right)}f'\left(\beta(x)\right)\left(c_{A}^{\epsilon}\left(x,0\right)-1\right)\d x\leq C
\end{align*}
for some suitable $\beta(x)\in\left(1,c_{A}^{\epsilon}\left(x,0\right)\right)$,
where we used a Taylor expansion and the
explicit structure of $c_{A}^{\epsilon}$. In particular,
in $\Gamma_{0}^{+}(\delta):=\Omega^{+}\left(0\right)\cap\Gamma_{0}(\delta)$
a change of variables yields
\[
\frac{1}{\epsilon}\int_{\Omega^{+}\left(0\right)}f'\left(\beta(x)\right)\left(c_{A}^{\epsilon}\left(x,0\right)-1\right)\d x\leq C_{1}+C_{2}\int_{\mathbb{T}^{1}}\int_{0}^{\frac{\delta}{\epsilon}}\left|\theta_{0}(\rho)-1\right|\d\rho\d s\leq C.
\]
The appearance of $\psi_{0}^{\epsilon}$ would have changed nothing
in this argumentation. This proves the claim.
\end{proof}
\begin{lem}
\label{lem:R-L2-Linf}Let $\alpha,\kappa\in (0,1)$. There are some $C(K),C(K,\alpha)$ such that for all $\epsilon\in\left(0,\epsilon_{1}\right)$
\begin{align}\nonumber
\left\Vert R\right\Vert _{L^{2}\left(0,T_{\epsilon};L^{\infty}(\Omega)\right)}&\leq C(K,\alpha)\epsilon^{M-\frac{3}{2}}\epsilon^{-\left(M+2\right)\alpha},\\
\left\Vert \nabla R\right\Vert _{L^{\infty}\left(0,T_{\epsilon};L^{2}(\Omega)\right)}&\leq C(K)\epsilon^{-\frac{1}{2}},\label{eq:gradR}\\\nonumber
\left\Vert \gamma R\right\Vert _{L^{\infty}\left(0,T_{\epsilon};L^{2+\kappa}(\Omega)\right)}&\leq C(K)\epsilon^{M-\frac{1}{2}-\frac{\kappa}{2+\kappa}M},\\\nonumber
\left\Vert R\right\Vert _{L^{\infty}\left(0,T_{\epsilon};L^{2}(\Omega)\right)}&\leq C(K)\epsilon^{\frac{1}{2}\left(M-\frac{1}{2}\right)}.
\end{align}
\end{lem}

\begin{proof}
 For $\alpha\in\left(0,1\right)$ it holds 
\begin{equation}
\left\Vert R\right\Vert _{L^{\infty}(\Omega)}\leq C\left(\alpha\right)\left\Vert R\right\Vert _{H^{1+\alpha}(\Omega)}\leq C\left(\alpha\right)\left\Vert R\right\Vert _{H^{1}(\Omega)}^{1-\alpha}\left\Vert R\right\Vert _{H^{2}(\Omega)}^{\alpha}.\label{eq:interpol}
\end{equation}
Due to the construction
and since $h_{A}^{\epsilon}$ is uniformly bounded in $X_{T_{\epsilon}}$
(cf.\ (\ref{eq:muv0,5est})). It can be easily verified by direct calculations
and the properties of $c_{A}^{\epsilon}$ given in Subsection \ref{subsec:The-Approximate-Solutions}
that $\left\Vert \Delta c_{A}^{\epsilon}\right\Vert _{L^{2}\left(\Omega_{T_{\epsilon}}\right)}\leq C(K)\frac{1}{\epsilon^{2}}$.
 Because of Lemma \ref{energy} and $R|_{\partial\Omega}=0$, we get
\begin{equation}
\|R\|_{H^2(\Omega)}\leq C'\left\Vert \Delta R\right\Vert _{L^{2}(\Omega_{T_{\epsilon}})}\leq C(K)\epsilon^{-\frac{7}{2}},\label{eq:laplaceR}
\end{equation}
where $C(K)$ depends 
only on $K$, $T_{0}$, and $C_{0}$ (where $C_{0}$ is the constant from
(\ref{eq:Eepsbes})). Using this and (\ref{eq:Main-est}) in (\ref{eq:interpol}),
we find 
\begin{align*}
\left\Vert R\right\Vert _{L^{2}\left(0,T_{\epsilon};L^{\infty}(\Omega)\right)} & \leq C(K)\left(\epsilon^{M-\frac{3}{2}}\right)^{1-\alpha}\left(\epsilon^{-\frac{7}{2}}\right)^{\alpha}=C(K)\epsilon^{M-\frac32}\epsilon^{-(M+2)\alpha}.
\end{align*}

In order to prove the second inequality,  we employ Lemma \ref{energy}, which yields
\[
\epsilon^{\frac{1}{2}}\left\Vert \nabla R\right\Vert _{L^{\infty}\left(0,T_{\epsilon};L^{2}(\Omega)\right)}\leq\epsilon^{\frac{1}{2}}\left(\left\Vert \nabla c^{\epsilon}\right\Vert _{L^{\infty}\left(0,T_{\epsilon};L^{2}(\Omega)\right)}+\left\Vert \nabla c_{A}^{\epsilon}\right\Vert _{L^{\infty}\left(0,T_{\epsilon};L^{2}(\Omega)\right)}\right)\leq C(K).
\]
Here we used $\epsilon^{\frac{1}{2}}\left\Vert \nabla c_{A}^{\epsilon}\right\Vert _{L^{\infty}\left(0,T_{\epsilon};L^{2}(\Omega)\right)}\leq C(K)$,
which is a consequence of the uniform bound on $c_{k},c_{O,\mathbf{B}}$
and their derivatives for $k\in\left\{ 0,\ldots,M+1\right\} $ and
the boundedness of $h_{M-\frac{1}{2}}^{\epsilon}$ in $X_{T_{\epsilon}}$,
combined with a change of variables.

For the proof of the third inequality we note that for $\kappa>0$ we have for any $u\in H^1_0(\Omega)$
\begin{equation}
  \left\Vert u\right\Vert _{L^{2+\kappa}(\Omega)}\leq C_{1}\left\Vert u\right\Vert _{L^{2}(\Omega)}^{1-\frac{\kappa}{2+\kappa}}\left\Vert \nabla u\right\Vert _{L^{2}(\Omega)}^{\frac{\kappa}{2+\kappa}}
  \label{eq:res1}
\end{equation}
for some $C_{1}>0$ due to the Gagliardo-Nirenberg interpolation
inequality. Moreover, (\ref{eq:res1}) together with by (\ref{eq:gradR})
and (\ref{eq:Main-est-d}) we obtain
\begin{align}\nonumber
  &\left\Vert \gamma R\right\Vert _{L^{\infty}(0,T_{\epsilon};L^{2+\kappa}(\Omega))} \le C_{1}\left\Vert \gamma R\right\Vert _{L^{\infty}(0,T_{\epsilon};L^{2}(\Omega))}^{1-\frac{\kappa}{2+\kappa}}\left\Vert \nabla R\right\Vert _{L^{\infty}(0,T_{\epsilon};L^{2}(\partial\Omega(\frac{\delta}{2})))}^{\frac{\kappa}{2+\kappa}}\\
  &\leq C(K)\epsilon^{M-\frac{1}{2}-\frac{\kappa}{2+\kappa}M}.\label{eq:res2}
\end{align}
because of Poincar\'e's inequality,  $\left\Vert R\right\Vert _{L^{2}(\Omega)}^{2}\leq\left\Vert R\right\Vert _{H^{-1}(\Omega)}\left\Vert \nabla R\right\Vert _{L^{2}(\Omega)}$,
(\ref{eq:Main-est-b}) and (\ref{eq:gradR}).
\end{proof}
The following lemma is an adapted version of \cite[Lemma~5.4]{nsac}.
\begin{lem}
\label{dim2einbett} Let $u\in H^{1}(\Omega).$ Then there
is some constant $C>0$ such that 
\begin{align*}
\left\Vert u\right\Vert _{L^{3}\left(\Gamma_{t}(\delta)\right)}^{3} & \leq C\left(\left\Vert u\right\Vert _{L^{2}\left(\gt\right)}+\left\Vert \nabla^{\Gamma}u\right\Vert _{L^{2}\left(\gt\right)}\right)^{\frac{1}{2}}\left(\left\Vert u\right\Vert _{L^{2}\left(\gt\right)}+\left\Vert \partial_{\mathbf{n}}u\right\Vert _{L^{2}\left(\gt\right)}\right)^{\frac{1}{2}}\\
 & \quad\cdot\left(\left\Vert u\right\Vert _{L^{2}\left(\gt\right)}\right)^{2}
\end{align*}
holds for all $t\in\left[0,T_{0}\right]$.
\end{lem}

\begin{proof}
Note
\[
\left\Vert u\right\Vert _{L^{3}\left(\gt\right)}^{3}\leq C\int_{-\delta}^{\delta}\int_{\Gamma_{t}}\left|u\left(p,r\right)\right|^{3}\d\mathcal{H}^{1}(p)\d r=C\left\Vert \left\Vert u\right\Vert _{L^{3}\left(\Gamma_{t}\right)}\right\Vert _{L^{3}\left(-\delta,\delta\right)}^{3}
\]
and $\left\Vert u\right\Vert _{L^{3}\left(\Gamma_{t}\right)}\leq C\left\Vert u\right\Vert _{H^{1}\left(\Gamma_{t}\right)}^{\frac{1}{6}}\left\Vert u\right\Vert _{L^{2}\left(\Gamma_{t}\right)}^{\frac{5}{6}}$
as $\Gamma_{t}$ is one-dimensional. Now H\"older's inequality leads
to 
\begin{align*}
\left\Vert u\right\Vert _{L^{3}\left(\gt\right)}^{3} & \leq C\left\Vert \left\Vert u\right\Vert _{H^{1}\left(\Gamma_{t}\right)}^{\frac{1}{6}}\left\Vert u\right\Vert _{L^{2}\left(\Gamma_{t}\right)}^{\frac{5}{6}}\right\Vert _{L^{3}\left(-\delta,\delta\right)}^{3}\\
 & \leq C\left\Vert \left\Vert u\right\Vert _{H^{1}\left(\Gamma_{t}\right)}\right\Vert _{L^{2}\left(-\delta,\delta\right)}^{\frac{1}{2}}\left\Vert \left\Vert u\right\Vert _{L^{\frac{10}{3}}\left(-\delta,\delta\right)}\right\Vert _{L^{2}\left(\Gamma_{t}\right)}^{\frac{5}{2}}\\
 & \leq C\left\Vert \left\Vert u\right\Vert _{H^{1}\left(\Gamma_{t}\right)}\right\Vert _{L^{2}\left(-\delta,\delta\right)}^{\frac{1}{2}}\left\Vert \left\Vert u\right\Vert _{H^{1}\left(-\delta,\delta\right)}\right\Vert _{L^{2}\left(\Gamma_{t}\right)}^{\frac{1}{2}}\left\Vert \left\Vert u\right\Vert _{L^{2}\left(-\delta,\delta\right)}\right\Vert _{L^{2}\left(\Gamma_{t}\right)}^{2},
\end{align*}
where we used $\left\Vert u\right\Vert _{L^{\frac{10}{3}}\left(-\delta,\delta\right)}\leq C\left\Vert u\right\Vert _{H^{1}\left(-\delta,\delta\right)}^{\frac{1}{5}}\left\Vert u\right\Vert _{L^{2}\left(-\delta,\delta\right)}^{\frac{4}{5}}$.
\end{proof}

\subsubsection{The Error in the Velocity\label{sec:Estimates-Regarding-the}}

For $\epsilon\in\left(0,\epsilon_{0}\right)$ we consider strong solutions
$\overline{\mathbf{v}^{\epsilon}}:\Omega_{T_{0}}\rightarrow\mathbb{R}^{2}$
and $\overline{p^{\epsilon}}:\Omega_{T_{0}}\rightarrow\mathbb{R}$
of the system
\begin{align}
-\Delta\overline{\mathbf{v}^{\epsilon}}+\nabla\overline{p^{\epsilon}} & =\mu_{A}^{\epsilon}\nabla c_{A}^{\epsilon} &  & \text{in }\Omega_{T_{0}},\label{eq:vbar}\\
\operatorname{div}\overline{\mathbf{v}^{\epsilon}} & =0 &  & \text{in }\Omega_{T_{0}},\label{eq:vbar2}\\
\left(-2D_{s}\overline{\mathbf{v}^{\epsilon}}+\overline{p^{\epsilon}}\mathbf{I}\right)\mathbf{n}_{\partial\Omega} & =\alpha_{0}\overline{\mathbf{v}^{\epsilon}} &  & \text{on }\partial_{T_{0}}\Omega\label{eq:vbar3}
\end{align}
(cf.\ Theorem \ref{thm:strongexstokes}) and weak solutions $\twz:\Omega_{T_{0}}\rightarrow\mathbb{R}^{2}$
and $q_{2}^{\epsilon}:\Omega_{T_{0}}\rightarrow\mathbb{R}$ of 
\begin{align}
-\Delta\twz+\nabla q_{2}^{\epsilon} & =-\epsilon\left(\operatorname{div}\left(\mathbf{h}\otimes_{s}\nabla R\right)+\operatorname{div}\left(\nabla R\otimes\nabla R\right)\right) &  & \text{in }\Omega_{T_{0}},\label{eq:w2}\\
\operatorname{div}\twz & =0 &  & \text{in }\Omega_{T_{0}},\label{eq:w22}\\
\left(-2D_{s}\twz+q_{2}^{\epsilon}\mathbf{I}\right)\mathbf{n}_{\partial\Omega} & =\alpha_{0}\twz &  & \text{in }\partial_{T_{0}}\Omega,\label{eq:w23}
\end{align}
where $\mathbf{h}$ is defined as in (\ref{eq:hh}). We consider the
right hand side of (\ref{eq:w2}) as a functional in $V_{0}^{'}$
given by 
\begin{equation}
\mathbf{g}^{\epsilon}\left(\psi\right):=\epsilon\int_{\Omega}\left(\left(\mathbf{h}\otimes_{s}\nabla R\right)+\left(\nabla R\otimes\nabla R\right)\right):\nabla\psi\d x\quad \text{for all }\psi \in V_0.\label{eq:geps}
\end{equation}
Introducing 
\begin{equation}
\mathbf{v}_{err}^{\epsilon}:=\mathbf{v}^{\epsilon}-\left(\overline{\mathbf{v}^{\epsilon}}+\twe+\twz\right)\label{eq:vepserrderf}
\end{equation}
we have $\mathbf{v}^{\epsilon}-\overline{\mathbf{v}^{\epsilon}}=\mathbf{v}_{err}^{\epsilon}+\twe+\twz$. Hence, if we control $\mathbf{v}_{err}^{\epsilon}$, $\twe$, and
$\twz$, we will control the error $\mathbf{v}^{\epsilon}-\overline{\mathbf{v}^{\epsilon}}$. 
\begin{lem}
\label{lem:w2eps-first}Let $\mathbf{\tilde{\mathbf{w}}}_{2}^{\epsilon}$
be the unique weak solution to \eqref{eq:w2}\textendash \eqref{eq:w23}
in $\Omega_{T_{0}}$ for $\epsilon\in\left(0,\epsilon_{1}\right)$.
Then it holds for all $r\in\left[1,2\right]$ and $q\in\left(1,2\right)$
\begin{align}
\left\Vert \tilde{\mathbf{w}}_{2}^{\epsilon}\right\Vert _{L^{r}\left(0,T_{\epsilon};L^{q}(\Omega)\right)} & \leq C(K,r,q)\epsilon^{\frac{2\left(M-1\right)}{r}}\label{eq:w2epsab}
\end{align}
for all $\epsilon\in\left(0,\epsilon_{1}\right)$.
\end{lem}

\begin{proof}
Since $\Omega\subset\mathbb{R}^{2}$,
we have $W_{q'}^{1}(\Omega)\hookrightarrow C^{0}(\Omega)$,
where $\frac{1}{q'}+\frac{1}{q}=1$. Thus Lemma \ref{LpLqstokes}
implies
\[
\left\Vert \tilde{\mathbf{w}}_{2}^{\epsilon}\right\Vert _{L^{r}\left(0,T_{\epsilon};L^{q}(\Omega)\right)}\leq C(q)\epsilon\left(\left\Vert \nabla R\otimes\mathbf{h}\right\Vert _{L^{r}\left(0,T_{\epsilon};L^{1}(\Omega)\right)}+\left\Vert \nabla R\otimes\nabla R\right\Vert _{L^{r}\left(0,T_{\epsilon};L^{1}(\Omega)\right)}\right).
\]
We use $X_{T_{\epsilon}}\hookrightarrow C^{0}(\left[0,T_{\epsilon}\right];C^{1}(\mathbb{T}^{1}))$
and $\partial_{\rho}c_{k}\in\mathcal{R}_{\alpha}$ for $k\in\left\{ 0,\ldots,M+1\right\} $
and get 
\begin{align*}
  &\epsilon\left\Vert \nabla R\otimes\mathbf{h}\right\Vert _{L^{r}\left(0,T_{\epsilon};L^{1}(\Omega)\right)} \\
  & \leq C\epsilon^{M-\frac{1}{2}}\epsilon^{\frac{1}{2}}\left\Vert \sum_{k=0}^{M+1}\epsilon^{k}\partial_{\rho}c_{k}\right\Vert _{L^{\infty}\left(\Gamma\left(2\delta;T_{0}\right);L^{2}\left(\mathbb{R}\right)\right)}\left\Vert h_{M-\frac{1}{2}}^{\epsilon}\right\Vert _{C^{0}\left(\left[0,T_{\epsilon}\right];C^{1}(\mathbb{T}^{1})\right)}\left\Vert \nabla R\right\Vert _{L^{2}\left(0,T_{\epsilon};L^{2}(\Omega)\right)}\\
 & \leq C(K)\epsilon^{2M-\frac{3}{2}}
\end{align*}
for all $\epsilon\in\left(0,\epsilon_{1}\right)$ due to (\ref{eq:muv0,5est})
and (\ref{eq:Main-est}). Moreover,
\begin{align*}
\epsilon\left\Vert \nabla R\otimes\nabla R\right\Vert _{L^{r}\left(0,T_{\epsilon};L^{1}(\Omega)\right)} & \leq\epsilon\left\Vert \nabla R\right\Vert _{L^{2}\left(0,T_{\epsilon};L^{2}(\Omega)\right)}^{\frac{2}{r}}\left\Vert \nabla R\right\Vert _{L^{\infty}\left(0,T_{\epsilon};L^{2}(\Omega)\right)}^{2-\frac{2}{r}}\leq C(K)\epsilon^{\frac{2M}{r}-\frac{2}{r}}
\end{align*}
for $\epsilon\in\left(0,\epsilon_{1}\right)$, by (\ref{eq:Main-est})
and (\ref{eq:gradR}). Combining the above estimates and using $r\geq 1$
the claim follows.
\end{proof}
\begin{lem}
\label{w2eps} Let $\varphi\in L^{\infty}\left(0,T_{\epsilon};H^{1}(\Omega)\right)$
and let the assumptions of Lemma \ref{lem:w2eps-first} hold. Then
there is some $r'>0$ such that
\begin{align*}
\int_{0}^{T_{\epsilon}}\left|\int_{\Omega}\left(\twz\cdot\nabla c_{A}^{\epsilon}\right)\varphi\d x\right|\d t & \leq C(K)T_{\epsilon}^{r'}\epsilon^{M}\left\Vert \varphi\right\Vert _{L^{\infty}\left(0,T_{\epsilon};H^{1}(\Omega)\right)} \quad \text{for all }\epsilon\in\left(0,\epsilon_{1}\right).
\end{align*}
\end{lem}

\begin{proof}
Let $r\in\left(1,2\right)$. As $\nabla c_{A}^{\epsilon}\in\mathcal{O}\left(\epsilon\right)$
in $L^{\infty}\left(\Omega_{T_{0}}\backslash\Gamma(2\delta)\right)$
it immediately follows
\begin{align}
\int_{0}^{T_{\epsilon}}\left|\,\int_{\Omega\backslash\Gamma_{t}(2\delta)}\left(\twz\cdot\nabla c_{A}^{\epsilon}\right)\varphi\d x\right|\d t & \leq C\epsilon\left\Vert \varphi\right\Vert _{L^{\infty}\left(0,T_{\epsilon};H^{1}(\Omega)\right)}\int_{0}^{T_{\epsilon}}\left\Vert \twz\right\Vert _{L^{q}(\Omega)}\d t\nonumber \\
 & \leq C(K)T_{\epsilon}^{\frac{1}{r'}}\left\Vert \varphi\right\Vert _{L^{\infty}\left(0,T_{\epsilon};H^{1}(\Omega)\right)}\epsilon^{\frac{2\left(M-1\right)}{r}+1}\label{eq:w2abauss}
\end{align}
by (\ref{eq:w2epsab}) for $q\in\left(1,2\right)$ and due to $H^{1}(\Omega)\hookrightarrow L^{s}(\Omega)$
for all $s\geq1$. The same estimate holds for $\left(\nabla\xi\left(d_{\Gamma}\right)\left(c_{I}-c_{O,\mathbf{B}}\right)+\left(1-\xi\left(d_{\Gamma}\right)\right)\nabla c_{O,\mathbf{B}}\right)$
in $\Gamma_{t}(2\delta)\backslash\Gamma_{t}(\delta)$
by (\ref{eq:matching}).

In $\Gamma\left(2\delta;T_{\epsilon}\right)$ we consider $\nabla\left(c_{0}\left(\rho(x,t),x,t\right)\right)=\nabla\left(\theta_{0}\left(\rho(x,t)\right)\right)$
and compute 
\begin{align*}
 & \int_{0}^{T_{\epsilon}}\left|\,\int_{\Gamma_{t}(2\delta)}\left(\twz\cdot\nabla\left(\theta_{0}\left(\rho(x,t)\right)\right)\right)\xi\left(d_{\Gamma}\right)\varphi\d x\right|\d t\\
 & \leq\int_{0}^{T_{\epsilon}}\int_{\Gamma_{t}(2\delta)}\left|\left(\twz\cdot\left(\mathbf{n}-\epsilon\nabla^{\Gamma}h_{A}^{\epsilon}(x,t)\right)\frac{1}{\epsilon}\theta_{0}'\left(\rho(x,t)\right)\right)\varphi\right|\d x\d t\\
 & \leq C(K)\left\Vert \varphi\right\Vert _{L^{\infty}\left(0,T_{\epsilon};H^{1}(\Omega)\right)}\epsilon^{-1}\int_{0}^{T_{\epsilon}}\left\Vert \tilde{\mathbf{w}}_{2}^{\epsilon}\right\Vert _{L^{q}(\Omega)}\d t\left\Vert \theta_{0}'\right\Vert _{L^{\infty}\left(\mathbb{R}\right)}\\
 & \leq C(K)T_{\epsilon}^{\frac{1}{r'}}\left\Vert \varphi\right\Vert _{L^{\infty}\left(0,T_{\epsilon};H^{1}(\Omega)\right)}\epsilon^{\frac{2\left(M-1\right)}{r}-1}.
\end{align*}
Since $\nabla\left(c_{I}-c_{0}\left(\rho\left(.\right),.\right)\right)\in\mathcal{O}\left(1\right)$
in $L^{\infty}\left(\Gamma\left(2\delta;T_{\epsilon}\right)\right)$,
we immediately get 
\[
\int_{0}^{T_{\epsilon}}\left|\,\int_{\Gamma_{t}(2\delta)}\left(\twz\cdot\nabla\left(c_{I}-c_{0}\left(\rho\left(.\right),.\right)\right)\right)\varphi\d x\right|\d t\leq C(K)T_{\epsilon}^{\frac{1}{r'}}\left\Vert \varphi\right\Vert _{L^{\infty}\left(0,T_{\epsilon};H^{1}(\Omega)\right)}\epsilon^{\frac{2\left(M-1\right)}{r}}
\]
by similar arguments as in (\ref{eq:w2abauss}).

As $M\geq4$ there always exists $r\in\left(1,2\right)$ (and with
it $r'\in\left(2,\infty\right)$) such that $\epsilon^{\frac{2\left(M-1\right)}{r}-1}<\epsilon^{M}$
which concludes the proof. 
\end{proof}
\begin{thm}[Error in the velocity]~\label{vehler}\\
  Let $\overline{\mathbf{v}^{\epsilon}}$ be a strong
solution to \eqref{eq:vbar}-\eqref{eq:vbar3}, let the
assumptions of Lemma \ref{lem:w2eps-first} hold true and let $\mathbf{v}_{err}^{\epsilon}:=\mathbf{v}^{\epsilon}-\left(\overline{\mathbf{v}^{\epsilon}}+\twe+\twz\right)$.
\begin{enumerate}
\item There is a constant $C(K)>0$ such that
\[
\left\Vert \mathbf{v}_{A}^{\epsilon}-\overline{\mathbf{v}^{\epsilon}}\right\Vert _{L^{2}\left(0,T_{\epsilon};H^{1}(\Omega)\right)}\leq C(K)\epsilon^{M}\quad \text{for all }\epsilon\in\left(0,\epsilon_{1}\right).
\]
\item For every $\beta\in\left(0,\frac{1}{2}\right)$ there are constants $C_{1}(\beta),C_{2}(\beta),C(K)>0$ such that 
\begin{align}
\left\Vert \mathbf{v}_{err}^{\epsilon}\right\Vert _{H^{1}(\Omega)} & \leq C_{1}\left(\left\Vert \rh\nabla c_{A}^{\epsilon}\right\Vert _{\left(H_{\sigma}^{1}(\Omega)\right)'}+\epsilon\left\Vert \nabla R\right\Vert _{L^{2}\left(\partial\Omega\left(\frac{\delta}{2}\right)\right)}^{1-2\beta}\left\Vert \gamma\nabla R\right\Vert _{H^{1}\left(\partial\Omega\left(\frac{\delta}{2}\right)\right)}^{1+2\beta}\right)\nonumber \\
 & \quad+C_{2}\epsilon^{2}\left\Vert \nabla R\right\Vert _{L^{2}\left(\partial\Omega\left(\frac{\delta}{2}\right)\right)}^{\frac{1}{2}-\beta}\left\Vert \gamma\nabla R\right\Vert _{H^{1}\left(\partial\Omega\left(\frac{\delta}{2}\right)\right)}^{\frac{1}{2}+\beta}\label{eq:vepserr-wotime}
\end{align}
for almost every $t\in (0,T_\epsilon)$ and
\begin{equation}
\left\Vert \mathbf{v}_{err}^{\epsilon}\right\Vert _{L^{1}\left(0,T_{\epsilon};H^{1}(\Omega)\right)}\leq C(K)C\left(T_{\epsilon},\epsilon\right)\epsilon^{M}\label{eq:vepserr-L1}
\end{equation}
for all $\epsilon\in\left(0,\epsilon_{1}\right)$, where $C\left(T,\epsilon\right)\rightarrow0$
as $\left(T,\epsilon\right)\rightarrow0$.
\end{enumerate}
\end{thm}

\begin{proof}
Ad 1) By definition, $\mathbf{v}_{A}^{\epsilon}-\overline{\mathbf{v}^{\epsilon}}$
satisfies 
\begin{align*}
-\Delta\left(\mathbf{v}_{A}^{\epsilon}-\overline{\mathbf{v}^{\epsilon}}\right)+\nabla\left(p_{A}^{\epsilon}-\overline{p^{\epsilon}}\right) & =\rs &  & \text{in }\Omega_{T_{\epsilon}},\\
\operatorname{div}\left(\mathbf{v}_{A}^{\epsilon}-\overline{\mathbf{v}^{\epsilon}}\right) & =\rdiv &  & \text{in }\Omega_{T_{\epsilon}},\\
\left(-2D_{s}\left(\mathbf{v}_{A}^{\epsilon}-\overline{\mathbf{v}^{\epsilon}}\right)+\left(p_{A}^{\epsilon}-\overline{p^{\epsilon}}\right)\mathbf{I}\right)\mathbf{n}_{\partial\Omega} & =\alpha_{0}\left(\mathbf{v}_{A}^{\epsilon}-\overline{\mathbf{v}^{\epsilon}}\right) &  & \text{on }\partial_{T_{\epsilon}}\Omega.
\end{align*}
Thus, we have by Theorem \ref{exinstokes} and since $\rdiv=0$ on
$\partial_{T_{0}}\Omega$ 
\[
\left\Vert \mathbf{v}_{A}^{\epsilon}-\overline{\mathbf{v}^{\epsilon}}\right\Vert _{L^{2}\left(0,T_{\epsilon};H^{1}(\Omega)\right)}\leq C\left(\left\Vert \rs\right\Vert _{L^{2}\left(0,T_{\epsilon};\left(H^{1}(\Omega)\right)'\right)}+\left\Vert \rdiv\right\Vert _{L^{2}\left(\Omega_{T_{\epsilon}}\right)}\right)
\]
and the claim follows from (\ref{eq:remstokes}).

Ad 2) First of all we have for $\psi\in H_{\sigma}^{1}(\Omega)$
\begin{equation}
\int_{\Omega}2D_{s}\left(\mathbf{v}^{\epsilon}-\overline{\mathbf{v}^{\epsilon}}\right):D_{s}\psi\d x+\alpha_{0}\int_{\partial\Omega}\left(\mathbf{v}^{\epsilon}-\overline{\mathbf{v}^{\epsilon}}\right)\cdot\psi\d\mathcal{H}^{1}(s)=\int_{\Omega}\left(\mu^{\epsilon}\nabla c^{\epsilon}-\mu_{A}^{\epsilon}\nabla c_{A}^{\epsilon}\right)\cdot\psi\d x.\label{eq:veps-vepsbar}
\end{equation}
Plugging in (\ref{eq:CH-Part2}), (\ref{eq:Hilliardapp}) and using
integration by parts we get 
\begin{align}
 & \int_{\Omega}\left(\mu^{\epsilon}\nabla c^{\epsilon}-\mu_{A}^{\epsilon}\nabla c_{A}^{\epsilon}\right)\cdot\psi\d x
 =\epsilon\int_{\Omega}\left(\nabla c^{\epsilon}\otimes\nabla c^{\epsilon}-\nabla c_{A}^{\epsilon}\otimes\nabla c_{A}^{\epsilon}\right):\nabla\psi\d x-\int_{\Omega}\rh\nabla c_{A}^{\epsilon}\cdot\psi\d x\nonumber \\
 & \quad+\epsilon\int_{\partial\Omega}\left(\left(\nabla c_{A}^{\epsilon}\otimes\nabla c_{A}^{\epsilon}-\nabla c^{\epsilon}\otimes\nabla c^{\epsilon}\right)\mathbf{n}_{\partial\Omega}+\frac{1}{2}\left(\left|\nabla c^{\epsilon}\right|^{2}-\left|\nabla c_{A}^{\epsilon}\right|^{2}\right)\mathbf{n}_{\partial\Omega}\right)\cdot\psi\d\mathcal{H}^{1}(s).\label{eq:rsveps-vepsbar}
\end{align}
Here we used $c^{\epsilon}=c_{A}^{\epsilon}=-1$ on $\partial_{T}\Omega$
together with $f\left(-1\right)=0$ and  $\operatorname{div}\left(\nabla c\otimes\nabla c\right)=\Delta c\nabla c+\frac{1}{2}\nabla\left(\left|\nabla c\right|^{2}\right)$
for sufficiently smooth $c\colon \Omega\rightarrow\mathbb{R}$. 

So, defining\textbf{ }$\mathbf{v}_{err}^{\epsilon}$ as in (\ref{eq:vepserrderf})
and taking into account (\ref{eq:veps-vepsbar}), (\ref{eq:rsveps-vepsbar}),
and the definitions of $\twe$ and $\twz$ (cf.\ (\ref{eq:w1}),
(\ref{eq:w2})) as weak solutions we find that $\mathbf{v}_{err}^{\epsilon}$
solves
\begin{align}
&\int_{\Omega}2D_{s}\mathbf{v}_{err}^{\epsilon}:D_{s}\psi\d x  +\alpha_{0}\int_{\partial\Omega}\mathbf{v}_{err}^{\epsilon}\cdot\psi\d\mathcal{H}^{1}(s)\nonumber \\
&=  \epsilon\!\int_{\partial\Omega}\!\!\left(\left(\nabla c_{A}^{\epsilon}\otimes\nabla c_{A}^{\epsilon}-\nabla c^{\epsilon}\otimes\nabla c^{\epsilon}\right)\mathbf{n}_{\partial\Omega}+\frac{1}{2}\left(\left|\nabla c^{\epsilon}\right|^{2}-\left|\nabla c_{A}^{\epsilon}\right|^{2}\right)\!\mathbf{n}_{\partial\Omega}\right)\!\cdot\psi\d\mathcal{H}^{1}(s)\nonumber \\
 & \qquad -\int_{\Omega}\rh\nabla c_{A}^{\epsilon}\cdot\psi\d x =:  \mathcal{F}^{\epsilon}\left(\psi\right)\label{eq:vepserr-weakform}
\end{align}
for all $\psi\in H_{\sigma}^{1}(\Omega)$. 
Due to (\ref{eq:rch2-nablacae}) we have
\begin{align}
\int_{0}^{T_{\epsilon}}\left|\int_{\Omega}\rh\nabla c_{A}^{\epsilon}\cdot\psi\d x\right|\d t & \leq\int_{0}^{T_{\epsilon}}\left\Vert \rh\nabla c_{A}^{\epsilon}\right\Vert _{\left(H_{\sigma}^{1}(\Omega)\right)'}\d t\left\Vert \psi\right\Vert _{H^{1}(\Omega)}\nonumber \\
 & \leq C(K)C\left(T_{\epsilon},\epsilon\right)\epsilon^{M},\label{eq:erstens}
\end{align}
where $C\left(T,\epsilon\right)\rightarrow0$ as $\left(T,\epsilon\right)\rightarrow0$.
Thus, we only need to estimate the appearing boundary terms in (\ref{eq:vepserr-weakform}). 
To this end, let $\beta\in\left(0,\frac{1}{2}\right)$ and we compute
\begin{align}
\epsilon\int_{0}^{T_{\epsilon}}\int_{\partial\Omega} & \left|\left(\left|\nabla c^{\epsilon}\right|^{2}-\left|\nabla c_{A}^{\epsilon}\right|^{2}\right)\psi\right|\d\mathcal{H}^{1}(s)\d t \leq\epsilon\int_{0}^{T_{\epsilon}}\int_{\partial\Omega}\left(\left|\nabla R\right|^{2}+2\left|\nabla R\right|\left|\nabla c_{A}^{\epsilon}\right|\right)\left|\psi\right|\d\mathcal{H}^{1}(s)\d t\nonumber \\
 & \leq C\int_{0}^{T_{\epsilon}}\left(\epsilon\left\Vert \gamma\nabla R\right\Vert _{H^{\frac{1}{2}+\beta}(\Omega)}^{2}+\epsilon^{2}\left\Vert \gamma\nabla R\right\Vert _{H^{\frac{1}{2}+\beta}(\Omega)}\right)\left\Vert \psi\right\Vert _{H^{1}(\Omega)}\d t\nonumber \\
 & \leq C_{1}\int_{0}^{T_{\epsilon}}\left(\epsilon\left\Vert \nabla R\right\Vert _{L^{2}\left(\partial\Omega\left(\frac{\delta}{2}\right)\right)}^{1-2\beta}\left\Vert \gamma\nabla R\right\Vert _{H^{1}(\Omega)}^{1+2\beta}\right)\left\Vert \psi\right\Vert _{H^{1}(\Omega)}\d t\nonumber \\
 & \quad+C_{2}\int_{0}^{T_{\epsilon}}\left(\epsilon^{2}\left\Vert \nabla R\right\Vert _{L^{2}\left(\partial\Omega\left(\frac{\delta}{2}\right)\right)}^{\frac{1}{2}-\beta}\left\Vert \gamma\nabla R\right\Vert _{H^{1}(\Omega)}^{\frac{1}{2}+\beta}\right)\left\Vert \psi\right\Vert _{H^{1}(\Omega)}\d t\label{eq:zweitens}\\
 & \leq C_{1}\left(\epsilon\left\Vert \nabla R\right\Vert _{L^{2}\left(\partial_{T_{\epsilon}}\Omega\left(\frac{\delta}{2}\right)\right)}^{1-2\beta}\left\Vert \gamma\nabla R\right\Vert _{L^{2}\left(0,T_{\epsilon};H^{1}(\Omega)\right)}^{1+2\beta}\right)\left\Vert \psi\right\Vert _{H^{1}(\Omega)}\nonumber \\
 & \quad+C_{2}T_{\epsilon}^{\frac{1}{2}}\left(\epsilon^{2}\left\Vert \nabla R\right\Vert _{L^{2}\left(\partial_{T_{\epsilon}}\Omega\left(\frac{\delta}{2}\right)\right)}^{\frac{1}{2}-\beta}\left\Vert \gamma\nabla R\right\Vert _{L^{2}\left(0,T_{\epsilon};H^{1}(\Omega)\right)}^{\frac{1}{2}+\beta}\right)\left\Vert \psi\right\Vert _{H^{1}(\Omega)},\label{eq:helfer-1}
\end{align}
where we used in the second inequality that $\nabla c_{A}^{\epsilon}=\mathcal{O}\left(\epsilon\right)$
in $L^{\infty}\left(\overline{\partial_{T_{0}}\Omega\left(\frac{\delta}{2}\right)}\right)$
and that $H^{\frac{1}{2}}\left(\partial\Omega\right)\hookrightarrow L^{s}\left(\partial\Omega\right)$
for all $s\in [1,\infty)$. 
and  $H^{\beta}\left(\partial\Omega\right)\hookrightarrow L^{2+\beta}\left(\partial\Omega\right)$,
since $\beta-\frac{1}{2}\geq-\frac{1}{2+\beta}$.
Now we may estimate
\begin{align}
\left\Vert \gamma\nabla R\right\Vert _{H^{1}(\Omega)} & \leq C\left\Vert \left(\gamma\Delta R,\left|\nabla R\right|,R\right)\right\Vert _{L^{2}\left(\partial\Omega\left(\frac{\delta}{2}\right)\right)}\label{eq:gamma-nablaR}
\end{align}
due to elliptic regularity theory and the definition of $\gamma$.
Using this in (\ref{eq:helfer-1}) together with (\ref{eq:Main-est-a})
and (\ref{eq:Main-est-d}), we find
\begin{align*}
\epsilon\int_{0}^{T_{\epsilon}}\int_{\partial\Omega}\left|\left(\left|\nabla c^{\epsilon}\right|^{2}-\left|\nabla c_{A}^{\epsilon}\right|^{2}\right)\psi\right|\d\mathcal{H}^{1}(s)\d t & \leq\left\Vert \psi\right\Vert _{H^{1}(\Omega)}C(K)\left(\epsilon^{2M-\frac{1}{2}-\beta}+T_{\epsilon}^{\frac{1}{2}}\epsilon^{M+\frac{5}{4}-\frac{1}{2}\beta}\right)\\
 & \leq\left\Vert \psi\right\Vert _{H^{1}(\Omega)}C(K)\left(\epsilon^{\frac{1}{2}}+T_{\epsilon}^{\frac{1}{2}}\right)\epsilon^{M}
\end{align*}
as $M\geq4$ and $\beta>0$ can be chosen sufficiently small.

For the remaining, not estimated term in (\ref{eq:vepserr-weakform}),
we note that
\begin{align*}
\epsilon\int_{0}^{T_{\epsilon}}\int_{\partial\Omega} & \left|\left(-\nabla c^{\epsilon}\otimes\nabla c^{\epsilon}+\nabla c_{A}^{\epsilon}\otimes\nabla c_{A}^{\epsilon}\right)\mathbf{n}_{\partial\Omega}\cdot\psi\right|\d\mathcal{H}^{1}(s)\d t\\
 & \leq\int_{0}^{T_{\epsilon}}\int_{\partial\Omega}\left(\left|\nabla R\right|^{2}+2\left|\nabla R\right|\left|\nabla c_{A}^{\epsilon}\right|\right)\left|\psi\right|\d\mathcal{H}^{1}(s)\d t
\end{align*}
and may then proceed as in (\ref{eq:helfer-1}). This proves (\ref{eq:vepserr-L1})
and also (\ref{eq:vepserr-wotime}) if we use (\ref{eq:erstens}) and (\ref{eq:zweitens}) without the integration in time.
\end{proof}
\begin{cor}
\label{cor:vehler}Let the assumptions of Theorem \ref{vehler} hold true
and let $\varphi\in L^{\infty}(0,T_{\epsilon};H^{1}(\Omega))$.
Then 
\begin{align}
\int_{0}^{T_{\epsilon}}\left|\int_{\Omega}\left(\mathbf{v}_{A}^{\epsilon}-\overline{\mathbf{v^{\epsilon}}}\right)\cdot\nabla c_{A}^{\epsilon}\varphi\d x\right|\d t & \leq C(K)T_{\epsilon}^{\frac{1}{2}}\epsilon^{M}\left\Vert \varphi\right\Vert _{L^{\infty}\left(0,T_{\epsilon};H^{1}(\Omega)\right)},\label{eq:vaeps-vepsbar}\\
\int_{0}^{T_{\epsilon}}\left|\int_{\Omega}\mathbf{v}_{err}^{\epsilon}\cdot\nabla c_{A}^{\epsilon}\varphi\d x\right|\d t & \leq C(K)C(\epsilon,T_{\epsilon})\epsilon^{M}\left\Vert \varphi\right\Vert _{L^{\infty}\left(0,T_{\epsilon};H^{1}(\Omega)\right)},\label{eq:vepserr-psi-nablacae}\\
  \int_{0}^{T_{\epsilon}}\left|\int_{\Omega}R\mathbf{v}_{err}^{\epsilon}\cdot\nabla R\gamma^{2}\d x\right|\d t & \leq C(K)C(\epsilon,T_{\epsilon})\epsilon^{2M-1},\label{eq:vepserr-R-nablaR}\\
\int_{0}^{T_{\epsilon}}\left|\int_{\Omega}\mathbf{v}_{err}^{\epsilon}\cdot\nabla R\varphi\d x\right|\d t & \leq C(K)C(\epsilon,T_{\epsilon})\epsilon^{M}\left\Vert \varphi\right\Vert _{L^{\infty}\left(0,T_{\epsilon};H^{1}(\Omega)\right)}\label{eq:vepserr-nablaR-psi}
\end{align}
for all $\epsilon\in\left(0,\epsilon_{1}\right)$ and $C(\epsilon,T)\rightarrow0$
if $\left(\epsilon,T\right)\rightarrow0$.
\end{cor}

\begin{proof}
Ad (\ref{eq:vaeps-vepsbar}): We have $\nabla c_{A}^{\epsilon}\in\mathcal{O}(\epsilon)$
in $L^{\infty}\left(\Omega_{T_{\epsilon}}\backslash\Gamma(2\delta;T_{\epsilon})\right)$
and thus get the estimate in $\Omega_{T_\eps}\setminus \Gamma(2\delta;T_\eps)$ by simply using H\"older's inequality and
Theorem \ref{vehler} 1). It remains to give an estimate inside $\Gamma(2\delta;T_{\epsilon})$:
We have $\nabla c_{O,\mathbf{B}}\in\mathcal{O}(\epsilon)$
in $L^{\infty}$ and the term involving $\left(c_{I}-c_{O,\mathbf{B}}\right)$
in $\nabla c_{A}^{\epsilon}$ can be handled by using (\ref{eq:matching}),
H\"older's inequality and Theorem \ref{vehler} 1) as before. Moreover,
we estimate 
\begin{align*}
  &\int_{\Gamma\left(2\delta;T_{\epsilon}\right)}\left|\left(\mathbf{v}_{A}^{\epsilon}-\overline{\mathbf{v}^{\epsilon}}\right)\xi\nabla\left(\theta_{0}\circ\rho\right)\varphi\right|\d (x,t)\\
  & \leq C\!\int_{0}^{T_{\epsilon}}\!\int_{\mathbb{T}^{1}}\!\left\Vert \left(\mathbf{v}_{A}^{\epsilon}-\overline{\mathbf{v}^{\epsilon}}\right)\varphi\right\Vert _{L^{\infty}(-2\delta,2\delta)}\!\int_{\mathbb{R}}\!\left|\theta_{0}'\left(\mathbf{n}+\nabla^{\Gamma}h_{A}^{\epsilon}\right)\right|\!\d\rho\d s\d t\\
  &\leq C(K)T_{\epsilon}^{\frac{1}{2}}\epsilon^{M}\left\Vert \varphi\right\Vert _{L^{\infty}\left(0,T_{\epsilon};H^{1}(\Omega)\right)},
\end{align*}
where we used $H^{1}(\Gamma_{t}(2\delta))\hookrightarrow L^{2,\infty}(\Gamma_{t}(2\delta))$
together with Theorem \ref{vehler} 1) in the last step. For $k\geq1$
we can use $\epsilon^{k}\nabla\left(c_{k}\left(\rho(.),.\right)\right)\in L^{\infty}\left(\Gamma(2\delta;T_{\epsilon})\right)$
uniformly in $\epsilon$. This proves (\ref{eq:vaeps-vepsbar}). 

Furthermore (\ref{eq:vepserr-psi-nablacae}) follows in the same way by using
(\ref{eq:vepserr-L1}) and noting that we may not generate a term
$T_{\epsilon}^{\frac{1}{2}}$ as we only control $\left\Vert \mathbf{v}_{err}^{\epsilon}\right\Vert _{L^{1}\left(0,T_{\epsilon};H^{1}(\Omega)\right)}$.

Ad (\ref{eq:vepserr-R-nablaR}): Since
$H^{1}(\Omega)\hookrightarrow L^{s}(\Omega)$
for all $s\in [1,\infty)$, we have
\begin{equation}
\int_{0}^{T_{\epsilon}}\left|\int_{\Omega}R\mathbf{v}_{err}^{\epsilon}\cdot\nabla R\gamma^{2}\d x\right|\d t\leq C(K)\int_{0}^{T_{\epsilon}}\left\Vert \mathbf{v}_{err}^{\epsilon}\right\Vert _{H^{1}(\Omega)}\left\Vert \gamma R\right\Vert _{L^{2+\kappa}(\Omega)}\left\Vert \gamma\nabla R\right\Vert _{L^{2}(\Omega)}\d t\label{eq:NULL}
\end{equation}
for $\kappa>0$. Regarding (\ref{eq:vepserr-wotime}), we need to
show three estimates:

Firstly, we have 
\begin{align}
 & \int_{0}^{T_{\epsilon}}\left\Vert \rh\nabla c_{A}^{\epsilon}\right\Vert _{\left(H^{1}(\Omega)\right)'}\left\Vert \gamma R\right\Vert _{L^{2+\kappa}(\Omega)}\left\Vert \gamma\nabla R\right\Vert _{L^{2}(\Omega)}\d t\nonumber \\
 & \le\left\Vert \rh\nabla c_{A}^{\epsilon}\right\Vert _{L^{2}\left(0,T_{\epsilon};\left(H_{\sigma}^{1}(\Omega)\right)'\right)}\left\Vert \gamma R\right\Vert _{L^{\infty}\left(0,T_{\epsilon};L^{2+\kappa}(\Omega)\right)}\left\Vert \gamma\nabla R\right\Vert _{L^{2}\left(\Omega_{T_{\epsilon}}\right)}\nonumber \\
 & \le C(K)C\left(T_{\epsilon},\epsilon\right)\epsilon^{2M}\left(\epsilon^{M-\frac{1}{2}-\frac{\kappa}{2+\kappa}M}\right) \leq C(K)C\left(T_{\epsilon},\epsilon\right)\epsilon^{2M-1},\label{eq:EINS}
\end{align}
where we used (\ref{eq:rch2-nablacae}), (\ref{eq:Main-est-d}) and
Lemma \ref{lem:R-L2-Linf} 3) and the fact
that $M\geq4$ and $\kappa>0$ can be chosen arbitrarily.

Secondly, we estimate for $\beta\in\left(0,\frac{1}{2}\right)$
\begin{align}
 & \int_{0}^{T_{\epsilon}}\epsilon\left\Vert \nabla R\right\Vert _{L^{2}\left(\partial\Omega\left(\frac{\delta}{2}\right)\right)}^{1-2\beta}\left\Vert \gamma\nabla R\right\Vert _{H^{1}(\Omega)}^{1+2\beta}\left\Vert \gamma R\right\Vert _{L^{2+\kappa}(\Omega)}\left\Vert \gamma\nabla R\right\Vert _{L^{2}(\Omega)}\d t\nonumber \\
 & \leq C\epsilon\left\Vert \nabla R\right\Vert _{L^{2}\left(\partial_{T_{\epsilon}}\Omega\left(\frac{\delta}{2}\right)\right)}^{1-2\beta}\left\Vert \gamma\nabla R\right\Vert _{L^{2}\left(0,T_{\epsilon};H^{1}(\Omega)\right)}^{1+2\beta}\left\Vert \gamma R\right\Vert _{L^{\infty}\left(0,T_{\epsilon};L^{2+\kappa}(\Omega)\right)}\left\Vert \nabla R\right\Vert _{L^{\infty}\left(0,T_{\epsilon};L^{2}(\Omega)\right)}\nonumber \\
 & \leq C(K)\left(\epsilon^{2M-\frac{1}{2}-\beta}\epsilon^{M-\frac{1}{2}-\frac{\kappa}{2+\kappa}M}\epsilon^{-\frac{1}{2}}\right) \leq C(K)\epsilon^{2M-\frac{1}{2}},\label{eq:ZWEI}
\end{align}
where we used (\ref{eq:gamma-nablaR}), (\ref{eq:Main-est-a}), (\ref{eq:Main-est-d}),
Lemma \ref{lem:R-L2-Linf} 3) and (\ref{eq:gradR}), $M\geq4$, and that $\beta>0$,$\kappa>0$ can be
chosen arbitrarily small. 

Similarly we obtain 
\begin{align}
 & \int_{0}^{T_{\epsilon}}\epsilon^{2}\left\Vert \nabla R\right\Vert _{L^{2}\left(\partial\Omega\left(\frac{\delta}{2}\right)\right)}^{\frac{1}{2}-\beta}\left\Vert \gamma\nabla R\right\Vert _{H^{1}(\Omega)}^{\frac{1}{2}+\beta}\left\Vert \gamma R\right\Vert _{L^{2+\kappa}(\Omega)}\left\Vert \gamma\nabla R\right\Vert _{L^{2}(\Omega)}\d t\nonumber \\
 & \leq C\epsilon^{2}\left\Vert \nabla R\right\Vert _{L^{2}\left(\partial_{T_{\epsilon}}\Omega\left(\frac{\delta}{2}\right)\right)}^{\frac{1}{2}-\beta}\left\Vert \gamma\nabla R\right\Vert _{L^{2}\left(0,T_{\epsilon};H^{1}(\Omega)\right)}^{\frac{1}{2}+\beta}\left\Vert \gamma R\right\Vert _{L^{\infty}\left(0,T_{\epsilon};L^{2+\kappa}(\Omega)\right)}\left\Vert \gamma\nabla R\right\Vert _{L^{2}\left(0,T_{\epsilon};L^{2}(\Omega)\right)}\nonumber \\
 & \leq C(K)\left(\epsilon^{M+\frac{5}{4}-\frac{\beta}{2}}\epsilon^{M-\frac{1}{2}-\frac{\kappa}{2+\kappa}M}\epsilon^{M-\frac{1}{2}}\right)\leq C(K)\epsilon^{2M-\frac{1}{2}}.\label{eq:DREI}
\end{align}
Now (\ref{eq:NULL})-(\ref{eq:DREI})
together with (\ref{eq:vepserr-wotime}) yield (\ref{eq:vepserr-R-nablaR}).

Concerning (\ref{eq:vepserr-nablaR-psi}) we note that
\begin{align*}
\left|\int_{\Omega}\mathbf{v}_{err}^{\epsilon}\cdot\nabla R\varphi\d x\right| & \leq\left\Vert \mathbf{v}_{err}^{\epsilon}\right\Vert _{H^{1}(\Omega)}\left\Vert \nabla R\right\Vert _{L^{2}(\Omega)}\left\Vert \varphi\right\Vert _{L^{4}(\Omega)}.
\end{align*}
 Regarding (\ref{eq:vepserr-wotime}), we again consider three different
terms: Firstly, 
\begin{align*}
\int_{0}^{T_{\epsilon}} & \left\Vert \rh\nabla c_{A}^{\epsilon}\right\Vert _{\left(H^{1}(\Omega)\right)'}\left\Vert \nabla R\right\Vert _{L^{2}(\Omega)}\left\Vert \varphi\right\Vert _{L^{4}(\Omega)}\d t\\
 & \leq\left\Vert \rh\nabla c_{A}^{\epsilon}\right\Vert _{L^{2}\left(0,T_{\epsilon};H^{1}(\Omega)'\right)}\left\Vert \nabla R\right\Vert _{L^{2}\left(\Omega_{T_{\epsilon}}\right)}\left\Vert \varphi\right\Vert _{L^{\infty}\left(0,T_{\epsilon};H^{1}(\Omega)\right)}
\end{align*}
where we may now use (\ref{eq:rch2-nablacae}) and (\ref{eq:Main-est}) and
$M\geq4$ to gain the estimate by the right-hand side of (\ref{eq:vepserr-nablaR-psi}). Secondly,
\begin{align*}
\int_{0}^{T_{\epsilon}} & \epsilon\left\Vert \nabla R\right\Vert _{L^{2}\left(\partial\Omega\left(\frac{\delta}{2}\right)\right)}^{1-2\beta}\left\Vert \gamma\nabla R\right\Vert _{H^{1}(\Omega)}^{1+2\beta}\left\Vert \nabla R\right\Vert _{L^{2}(\Omega)}\left\Vert \varphi\right\Vert _{L^{4}(\Omega)}\d t\\
 & \leq C\epsilon\left\Vert \nabla R\right\Vert _{L^{2}\left(\partial_{T_{\epsilon}}\Omega\left(\frac{\delta}{2}\right)\right)}^{1-2\beta}\left\Vert \gamma\nabla R\right\Vert _{L^{2}\left(0,T_{\epsilon};H^{1}(\Omega)\right)}^{1+2\beta}\left\Vert \nabla R\right\Vert _{L^{\infty}\left(0,T_{\epsilon};L^{2}(\Omega)\right)}\left\Vert \varphi\right\Vert _{L^{\infty}\left(0,T_{\epsilon};H^{1}(\Omega)\right)}
\end{align*}
for $\beta\in\left(0,\frac{1}{2}\right)$, where (\ref{eq:Main-est})
and (\ref{eq:gradR}) together with $M\geq4$ imply the desired
estimate. Thirdly,
\begin{align*}
\int_{0}^{T_{\epsilon}} & \epsilon^{2}\left\Vert \nabla R\right\Vert _{L^{2}\left(\partial\Omega\left(\frac{\delta}{2}\right)\right)}^{\frac{1}{2}-\beta}\left\Vert \gamma\nabla R\right\Vert _{H^{1}(\Omega)}^{\frac{1}{2}+\beta}\left\Vert \nabla R\right\Vert _{L^{2}(\Omega)}\left\Vert \varphi\right\Vert _{L^{4}(\Omega)}\d t\\
 & \le C\epsilon^{2}\left\Vert \nabla R\right\Vert _{L^{2}\left(\partial_{T_{\epsilon}}\Omega\left(\frac{\delta}{2}\right)\right)}^{\frac{1}{2}-\beta}\left\Vert \gamma\nabla R\right\Vert _{L^{2}\left(0,T_{\epsilon};H^{1}(\Omega)\right)}^{\frac{1}{2}+\beta}\left\Vert \nabla R\right\Vert _{L^{2}\left(\Omega_{T_{\epsilon}}\right)}\left\Vert \varphi\right\Vert _{L^{\infty}\left(0,T_{\epsilon};H^{1}(\Omega)\right)}
\end{align*}
for $\beta\in\left(0,\frac{1}{2}\right)$, where finally (\ref{eq:Main-est})
and $M\geq4$ imply the claim.
\end{proof}
\begin{lem}
\label{CHS-Konvektion} Let $\varphi\in L^{\infty}\left(0,T_{\epsilon};H^{1}(\Omega)\right)$ and $(\twe)^{\Gamma}=\twe-\left.\twe\right|_{\Gamma}$.
Then 
\begin{align}
\int_{0}^{T_{\epsilon}}\left|\,\int_{\Gamma_{t}(\delta)}\frac{1}{\epsilon}\left(\twe\right)^{\Gamma}\cdot\mathbf{n}\theta_{0}'(\rho)\varphi\d x\right|\d t & \le C(K)(T_{\epsilon})^{\frac{1}{2}}\epsilon^{M}\Vert \varphi\Vert _{L^{\infty}\left(0,T_{\epsilon};H^{1}(\Omega)\right)},\label{eq:Konvektion1}\\
\int_{0}^{T_{\epsilon}}\left|\,\int_{\Gamma_{t}(\delta)}\left(\twe\right)^{\Gamma}\cdot\nabla^{\Gamma}h_{A}^{\epsilon}\theta_{0}'(\rho)\varphi\d x\right|\d t & \le C(K)(T_{\epsilon})^{\frac{1}{2}}\epsilon^{M+1}\Vert \varphi\Vert _{L^{\infty}\left(0,T_{\epsilon};H^{1}(\Omega)\right)},\label{eq:Konvektion2}\\
\int_{0}^{T_{\epsilon}}\left|\,\int_{\Gamma_{t}(\delta)}\left(\twe\right)^{\Gamma}\cdot\left(\frac{\mathbf{n}}{\epsilon}-\nabla^{\Gamma}h_{A}^{\epsilon}\right)\epsilon\partial_{\rho}c_{1}\varphi\d x\right|\d t & \le C(K)(T_{\epsilon})^{\frac{1}{2}}\epsilon^{M+1}\Vert \varphi\Vert _{L^{\infty}\left(0,T_{\epsilon};H^{1}(\Omega)\right)}\label{eq:Konvektion3}
\end{align}
for all $\epsilon\in\left(0,\epsilon_{1}\right)$.
\end{lem}

\begin{proof}
  Proceeding as in \cite[proof of Lemma~5.1]{nsac} we find, using $\partial_{\mathbf{n}}\twe=-\operatorname{div}_{\boldsymbol{\tau}} \twe$,
\begin{align*}
 & \int_{\Gamma_{t}(\delta)}\frac{1}{\epsilon}\left(\twe-\left.\twe\right|_{\Gamma}\right)\cdot\mathbf{n}\theta_{0}'(\rho)\varphi\d x\\
 & =\int_{-\delta}^{\delta}\int_{0}^{r}\int_{\Gamma_{t}}\frac{1}{\epsilon}\twe\left(\sigma,p,t\right)\cdot\nabla_{\boldsymbol{\tau}}\left(\rho\left(r,p,t\right)\right)\theta_{0}''\left(\rho\left(r,p,t\right)\right)\varphi\left(r,p,t\right)J\left(r,p,t\right)\d\mathcal{H}^{1}(p)\d\sigma\d r\\
 & \quad+\int_{-\delta}^{\delta}\int_{0}^{r}\int_{\Gamma_{t}}\frac{1}{\epsilon}\twe\left(\sigma,p,t\right)\cdot\nabla_{\boldsymbol{\tau}}\left(\varphi\left(r,p,t\right)J\left(r,p,t\right)\right)\theta_{0}'\left(\rho\left(r,p,t\right)\right)\d\mathcal{H}^{1}(p)\d\sigma\d r\\
 & \quad+\int_{-\delta}^{\delta}\int_{0}^{r}\int_{\Gamma_{t}}\frac{1}{\epsilon}\twe\left(\sigma,p,t\right)\cdot\mathbf{n}_{\Gamma_{t}}(p)\kappa(p)\theta_{0}'\left(\rho\left(r,p,t\right)\right)\varphi\left(r,p,t\right)J\left(r,p,t\right)\d\mathcal{H}^{1}(p)\d\sigma\d r\\
 & =:I_{1}+I_{2}+I_{3},
\end{align*}
because of Lemma~\ref{lem:divlm}. 
To estimate the occurring integrals, we note that
\begin{equation}
\left|\int_{0}^{r}\twe\left(\sigma,p,t\right)\d\sigma\right|\leq r\left\Vert \twe\left(.,p,t\right)\right\Vert _{L^{\infty}\left(-\delta,\delta\right)}\leq Cr\left\Vert \twe\left(.,p,t\right)\right\Vert _{H^{1}\left(-\delta,\delta\right)}\label{eq:CHS-wabs}
\end{equation}
holds for all $p\in\Gamma_{t}$ and $r\in\left(-\delta,\delta\right)$. 
After a change of variables, we get
\begin{align*}
\left|I_{2}\right| & \leq C\epsilon\int_{\Gamma_{t}}\left\Vert \twe\left(.,p,t\right)\right\Vert _{H^{1}\left(-\delta,\delta\right)}\int_{-\frac{\delta}{\epsilon}-h_{A}^{\epsilon}}^{\frac{\delta}{\epsilon}-h_{A}^{\epsilon}}\left|\left(\nabla_{\boldsymbol{\tau}}\left(\varphi J\right)\left(\epsilon\left(\rho+h_{A}^{\epsilon}\right),p,t\right)\right)\left(\rho+h_{A}^{\epsilon}\right)\theta_{0}'(\rho)\right|\d\rho\d\mathcal{H}^{1}(p)\\
 & \leq C(K)\epsilon^{\frac{1}{2}}\left\Vert \twe\left(.,t\right)\right\Vert _{L^{2}\left(\Gamma_{t};H^{1}\left(-\delta,\delta\right)\right)}\left\Vert \left(\rho+1\right)\theta_{0}'\right\Vert _{L^{2}\left(\mathbb{R}\right)}\left(\left\Vert \varphi\right\Vert _{H^{1}(\Omega)}+\epsilon^{\frac{1}{2}}\left\Vert \varphi\right\Vert _{L^{2,\infty}\left(\Gamma_{t}(\delta)\right)}\right)
\end{align*}
where we used (\ref{eq:CHS-wabs}), $\left\Vert h_{A}^{\epsilon}\right\Vert _{C^{0}\left(\left[0,T\right];C^{1}\left(\mathbb{T}^{1}\right)\right)}\leq C(K)$
as in (\ref{eq:haepglm}). Employing Lemma \ref{L4inf} and the exponential
decay of $\theta_{0}'$,we find 
\[
\int_{0}^{T_{\epsilon}}\left|I_{2}\right|\d t\leq C(K)\left(T_{\epsilon}\right)^{\frac{1}{2}}\epsilon^{\frac{1}{2}}\left\Vert \twe\right\Vert _{L^{2}\left(0,T_{\epsilon};H^{1}(\Omega)\right)}\left\Vert \varphi\right\Vert _{L^{\infty}\left(0,T_{\epsilon};H^{1}(\Omega)\right)}.
\]
As $\nabla_{\boldsymbol{\tau}}\left(\rho\left(r,p,t\right)\right)=\nabla_{\boldsymbol{\tau}}\left(h_{A}^{\epsilon}\left(S\left(p,t\right),t\right)\right)$
we may estimate $I_{1}$ in a similar manner, and $\left|\kappa(p)\right|\le C$
for all $p\in\Gamma_{t}$, implies the equivalent estimate for $I_{3}$.
Lemma \ref{Wichtig} together with the estimates on $I_{1}$, $I_{2}$
and $I_{3}$ completes the proof for (\ref{eq:Konvektion1}).

To show (\ref{eq:Konvektion2}), we calculate 
\begin{align*}
 & \left|\,\int_{\Gamma_{t}(\delta)}\left(\left(\twe-\left.\twe\right|_{\Gamma_{t}}\right)\right)\cdot\nabla^{\Gamma}h_{A}^{\epsilon}\left(S(x,t),t\right)\theta_{0}'(\rho)\varphi\d x\right|\\
 & \leq C\int_{\mathbb{T}^{1}}\int_{-\delta}^{\delta}\int_{0}^{r}\left|\left(\partial_{\mathbf{n}}\twe\right)\left(X\left(\sigma,s,t\right)\right)\right|\d\sigma\left|\nabla^{\Gamma}h_{A}^{\epsilon}(s,t)\theta_{0}'\left(\rho\left(X(r,s,t)\right)\right)\varphi\right|\d r\d s.\\
 & \leq C(K)\int_{\mathbb{T}^{1}}\left\Vert \twe\left(X\left(.,s,t\right)\right)\right\Vert _{H^{1}\left(-\delta,\delta\right)}\left\Vert \varphi\right\Vert _{L^{\infty}\left(-\delta,\delta\right)}\int_{-\frac{\delta}{\epsilon}-h_{A}^{\epsilon}}^{\frac{\delta}{\epsilon}-h_{A}^{\epsilon}}\epsilon^{\frac{3}{2}}\left|\rho+1\right|\left|\theta_{0}^{'}(\rho)\right|\d\rho\d s
\end{align*}
since
\[
\left|\int_{0}^{r}\left(\partial_{\mathbf{n}}\twe\right)\left(X(r,s,t)\right)\d\sigma\right|\leq\left\Vert \twe\left(X\left(.,s,t\right)\right)\right\Vert _{H^{1}\left(-\delta,\delta\right)}\sqrt{\left|r\right|}\quad\forall r\in\left(-\delta,\delta\right)
\]
and $s\in\mathbb{T}^{1}$, $t\in\left[0,T_{\epsilon}\right]$. Integration
from $0$ to $T_{\epsilon}$ and Lemma \ref{Wichtig} yield the assertion.
The proof of (\ref{eq:Konvektion3}) follows analogously to the proof
of (\ref{eq:Konvektion2}) since $\partial_{\rho}c_{1}\in\mathcal{R}_{\alpha}$.
\end{proof}
\begin{lem}
\label{lem:Konvekionfinal}Let $\varphi\in L^{\infty}\left(0,T_{\epsilon};H^{1}(\Omega)\right)$
and $\mathbf{w}_{1}^{\epsilon}=\frac{\twe}{\epsilon^{M-\frac{1}{2}}}$.
Then it holds
\[
\int_{0}^{T_{\epsilon}}\left|\int_{\Omega}\epsilon^{M-\frac{1}{2}}\left(\mathbf{w}_{1}^{\epsilon}-\left.\mathbf{w}_{1}^{\epsilon}\right|_{\Gamma}\xi\right)\cdot\nabla c_{A}^{\epsilon}\varphi\d x\right|\d t\leq C(K)C\left(T_{\epsilon},\epsilon\right)\epsilon^{M}\left\Vert \varphi\right\Vert _{L^{\infty}\left(0,T_{\epsilon};H^{1}(\Omega)\right)},
\]
for all $\epsilon\in\left(0,\epsilon_{1}\right)$, where $C\left(T,\epsilon\right)\rightarrow0$
as $\left(T,\epsilon\right)\rightarrow0$.
\end{lem}

\begin{proof}
In $\Omega_{T_{0}}\backslash\Gamma(2\delta)$ we have $\nabla c_{A}^{\epsilon}\in\mathcal{O}\left(\epsilon\right)$
in $L^{\infty}$, thus the estimate in this region is a direct consequence
of Lemma \ref{Wichtig}. Inside $\Gamma\left(2\delta;T_{\epsilon}\right)$
we have $\nabla c_{A}^{\epsilon}=\xi\nabla c_{I}+\xi'\mathbf{n}\left(c_{I}-c_{O,\mathbf{B}}\right)+\left(1-\xi\right)\nabla c_{O,\mathbf{B}}$.
The term involving $\nabla c_{O,\mathbf{B}}$ can be treated as in
the outer region and the estimate for the term $\left(c_{I}-c_{O,\mathbf{B}}\right)$
is a consequence of (\ref{eq:matching}). Now by definition
\[
\nabla c_{I}(x,t)=\sum_{i=0}^{M+1}\epsilon^{i}\left(\partial_{\rho}c_{i}\left(\rho(x,t),x,t\right)\left(\frac{\mathbf{n}\left(S(x,t),t\right)}{\epsilon}-\nabla^{\Gamma}h_{A}^{\epsilon}(x,t)\right)+\nabla_{x}c_{i}\left(\rho(x,t),x,t\right)\right)
\]
for $(x,t)\in\Gamma\left(2\delta;T_{\epsilon}\right)$.
Since $\nabla_{x}c_{0}\equiv0$, we have $\sum_{i=0}^{M+1}\epsilon^{i}\nabla_{x}c_{i}\in\mathcal{O}\left(\epsilon\right)$
in $L^{\infty}\left(\mathbb{R}\times\Gamma(2\delta)\right)$,
allowing for a suitable estimate with the help of Lemma \ref{Wichtig}.
Choosing $\epsilon>0$ small enough, we have $\left|\frac{d_{\Gamma}}{\epsilon}-h_{A}^{\epsilon}\right|\geq\frac{\delta}{2\epsilon}$
in $\Gamma\left(2\delta;T_{\epsilon}\right)\backslash\Gamma\left(\delta;T_{\epsilon}\right)$
and as $\partial_{\rho}c_{i}\in\mathcal{R}_{\alpha}$, this leads
to 
\begin{align*}
\int_{\Gamma\left(2\delta;T_{\epsilon}\right)\backslash\Gamma\left(\delta;T_{\epsilon}\right)} & \epsilon^{M-\frac{1}{2}}\left|\mathbf{w}_{1}^{\epsilon}-\left.\mathbf{w}_{1}^{\epsilon}\right|_{\Gamma}\xi\right|\left|\partial_{\rho}c_{i}\left(\rho(x,t),x,t\right)\left(\frac{\mathbf{n}}{\epsilon}-\nabla^{\Gamma}h_{A}^{\epsilon}\right)\right|\left|\varphi\right|\d(x,t)\\
 & \leq C(K)\left\Vert \twe\right\Vert _{L^{2}\left(0,T_{\epsilon};H^{1}(\Omega)\right)}\left\Vert \varphi\right\Vert _{L^{2}\left(0,T_{\epsilon};H^{1}(\Omega)\right)}\frac{1}{\epsilon}C_{1}e^{-C_{2}\frac{\delta}{2\epsilon}}
\end{align*}
for all $i\in\left\{ 0,\ldots,M+1\right\} $, where we have used $\left\Vert h_{M-\frac{1}{2}}^{\epsilon}\right\Vert _{C^{0}\left(0,T_{\epsilon};C^{1}\left(\mathbb{T}^{1}\right)\right)}\leq C(K)$
due to (\ref{eq:haepglm}). So we only need to show
\begin{align*}
 & \int_{\Gamma\left(\delta;T_{\epsilon}\right)}\epsilon^{M-\frac{1}{2}}\left|\left(\mathbf{w}_{1}^{\epsilon}-\left.\mathbf{w}_{1}^{\epsilon}\right|_{\Gamma}\right)\cdot\left(\epsilon^{i}\partial_{\rho}c_{i}\left(\rho(x,t),x,t\right)\left(\frac{\mathbf{n}}{\epsilon}-\nabla^{\Gamma}h_{A}^{\epsilon}\right)\right)\varphi\right|\d(x,t)\\
 & \leq C\left(T,\epsilon\right)C(K)\epsilon^{M}\left\Vert \varphi\right\Vert _{L^{\infty}\left(0,T_{\epsilon};H^{1}(\Omega)\right)}
\end{align*}
for $i\in\left\{ 0,\ldots,M+1\right\} $, where $C\left(T,\epsilon\right)\rightarrow0$
as $\left(T,\epsilon\right)\rightarrow0$. For $i\in\left\{ 0,1\right\} $
this is a consequence of Lemma \ref{CHS-Konvektion} and for $i\geq2$
this is a consequence of $\partial_{\rho}c_{i}\in L^{\infty}\left(\mathbb{R}\times\Gamma(2\delta)\right)$.
This shows the claim.
\end{proof}

\subsection{The Proof of the Main Result\label{sec:The-Proof}}

Let the assumptions of Theorem \ref{Main} hold true. Moreover, let $c_{A}^{\epsilon},\mu_{A}^{\epsilon},\mathbf{v}_{A}^{\epsilon},p_{A}^{\epsilon},h_{A}^{\epsilon}$
be given as in \cite[Definition~4.1]{NSCH2}, which 
implies in particular
that the properties discussed in Subsection \ref{subsec:The-Approximate-Solutions}
hold. Let $\tilde{\mathbf{w}}_{1}^{\epsilon}$ and $\twz$ be weak
solutions to (\ref{eq:w1})\textendash (\ref{eq:w13}) and (\ref{eq:w2})\textendash (\ref{eq:w23}), resp.,
and let $\overline{\mathbf{v}^{\epsilon}}$ be a strong solution to
(\ref{eq:vbar})\textendash (\ref{eq:vbar3}). We denote $\mathbf{w}_{1}^{\epsilon}=\frac{\twe}{\epsilon^{M-\frac{1}{2}}}$.
Additionally, let $\left(\mathbf{v}^{\epsilon},p^{\epsilon},c^{\epsilon},\mu^{\epsilon}\right)$
be smooth solutions to (\ref{eq:StokesPart})\textendash (\ref{eq:StokesBdry})
such that (\ref{eq:canf}) is satisfied. Note that Proposition \ref{prop:Energyholds}
implies that Lemma \ref{energy} is applicable in this situation.
We define $R:=c^{\epsilon}-c_{A}^{\epsilon}$ in $\Omega_{T_{0}}$
and let $\varphi\left(.,t\right)\in H^{2}(\Omega)\cap H_{0}^{1}(\Omega)$
for $t\in\left[0,T_{0}\right]$ be the unique solution of the problem
\begin{align*}
-\Delta\varphi\left(.,t\right) & =R\left(.,t\right) &  & \text{in }\Omega,\\
\varphi\left(.,t\right) & =0 &  & \text{on }\partial\Omega.
\end{align*}
Then $\varphi$ is smooth and we have $\left\Vert \varphi\left(.,0\right)\right\Vert _{H^{1}(\Omega)}\leq C\Vert R\left(.,0\right)\Vert _{L^{2}(\Omega)}\leq C_{\psi_{0}}\epsilon^{M}$for
all $\epsilon\in\left(0,1\right)$. This implies the existence of
some family $\left(\tau_{\epsilon}\right)_{\epsilon\in\left(0,1\right)}\subset\left(0,T_{0}\right]$
and $K\geq1$ such that Assumption~\ref{assu:Main-est} is satisfied
(and in particular (\ref{eq:Main-est}) holds for $\tau_{\epsilon}$)
and such that 
\begin{equation}
\left\Vert \varphi\left(.,0\right)\right\Vert _{H^{1}(\Omega)}\leq\left\Vert R\left(.,0\right)\right\Vert _{L^{2}(\Omega)}\leq\frac{K}{2}\epsilon^{M}.\label{eq:phi0}
\end{equation}
Moreover, we may choose $\epsilon_{0}\in\left(0,1\right)$ small enough,
such that (\ref{eq:remcahn})\textendash (\ref{eq:rch1-rch2-Linfbdry}),
Lemma \ref{lem:hM-1}, (\ref{eq:w1epsab}) and (\ref{eq:matching})
hold. This implies in particular that Assumption \ref{assu:Auxiliary}
is satisfied and that we may use all the results shown in Section
\ref{sec:Auxiliary-Results}. Now let $T\in\left(0,T_{0}\right]$
and for $\epsilon\in\left(0,\epsilon_{0}\right)$ we set 
\begin{equation}
T_{\epsilon}:=\sup\left\{ \left.t\in\left(0,T\right]\right|\eqref{eq:Main-est}\text{ holds true for \ensuremath{t}}\right\} .\label{eq:Teps}
\end{equation}
We will show in the following that we may choose $T\in\left(0,T_{0}\right]$
(independent of $\epsilon$) and $\epsilon_{0}$ small enough, such
that $T_{\epsilon}=T$ for all $\epsilon\in\left(0,\epsilon_{0}\right)$.

Now let $T'\in\left(0,T_{0}\right]$ be fixed. Multiplying the difference
of the differential equations (\ref{eq:CH-Part1}) and (\ref{eq:Cahnapp})
by $\varphi$ and integrating the result over $\Omega$ yields 
\begin{align}
0 & =\int_{\Omega}\varphi\partial_{t}\left(-\Delta\varphi\right)+\varphi\left(\left(\mathbf{v}^{\epsilon}\cdot\nabla R\right)-\left(\mathbf{v}_{A}^{\epsilon}-\overline{\mathbf{v^{\epsilon}}}\right)\cdot\nabla c_{A}^{\epsilon}+\left(\twe-\twe|_{\Gamma}\xi\left(d_{\Gamma}\right)\right)\cdot\nabla c_{A}^{\epsilon}\right)\d x\nonumber \\
 & \quad+\int_{\Omega}\varphi\left(\mathbf{v}_{err}^{\epsilon}\cdot\nabla c_{A}^{\epsilon}+\twz\cdot\nabla c_{A}^{\epsilon}-\Delta\left(\mu^{\epsilon}-\mu_{A}^{\epsilon}\right)\right)+\varphi\rc\d x\label{eq:CHS-Premaj}
\end{align}
for all $t\in\left(0,T\right).$ Here we used the definition of
$\varphi$ and the identity 
\begin{equation}
\mathbf{v}^{\epsilon}\cdot\nabla c^{\epsilon}-\mathbf{v}_{A}^{\epsilon}\cdot\nabla c_{A}^{\epsilon}=\mathbf{v}^{\epsilon}\cdot\nabla R+\left(\twe+\twz\right)\cdot\nabla c_{A}^{\epsilon}-\left(\mathbf{v}_{A}^{\epsilon}-\overline{\mathbf{v}^{\epsilon}}\right)\cdot\nabla c_{A}^{\epsilon}+\mathbf{v}_{err}^{\epsilon}\cdot\nabla c_{A}^{\epsilon},\label{eq:veps-vaeps}
\end{equation}
which is a consequence of the definition of $\mathbf{v}_{err}^{\epsilon}$
(cf.\ (\ref{eq:vepserrderf})). In order to shorten the notation,
we now write
\begin{align}\nonumber
  \mathcal{E}(R,T')&:=\int_{\Omega_{T'}}\epsilon\left|\nabla R\right|^{2}+\epsilon^{-1}f''\left(c_{A}^{\epsilon}\right)R^{2}\d(x,t),\\
  \mathcal{N}(c_{A}^{\epsilon},R)&:=f'\left(c_{A}^{\epsilon}+R\right)-f'\left(c_{A}^{\epsilon}\right)-f''\left(c_{A}^{\epsilon}\right)R\label{eq:N}
 = f'''(c_A)\frac{R^2}2 + f^{(4)}(c_A) \frac{R^3}6\\\nonumber  
\mathcal{R}^{\epsilon}&:=\left(\epsilon^{M-\frac{1}{2}}\left(-\mathbf{w}_{1}^{\epsilon}+\mathbf{w}_{1}^{\epsilon}|{}_{\Gamma}\xi\left(d_{\Gamma}\right)\right)\cdot\nabla c_{A}^{\epsilon}\right)
\end{align}
which 
leads us to 
\begin{align}
0 & =\frac{1}{2}\frac{\d}{\d t}\int_{\Omega}\left|\nabla\varphi\right|^{2}\d x+\mathcal{E}(R,T')+\int_{\Omega}\varphi\left(\mathbf{v}^{\epsilon}\cdot\nabla R\right)+\epsilon^{-1}\mathcal{N}(c_{A}^{\epsilon},R)R\d x\nonumber \\
 & \quad-\int_{\Omega}\varphi\left(\left(\mathbf{v}_{A}^{\epsilon}-\overline{\mathbf{v^{\epsilon}}}\right)\cdot\nabla c_{A}^{\epsilon}-\twz\cdot\nabla c_{A}^{\epsilon}-\mathbf{v}_{err}^{\epsilon}\cdot\nabla c_{A}^{\epsilon}-\rc+\mathcal{R}^{\epsilon}\right)+R\rh\d x\label{eq:CHS-Majoreq}
\end{align}
for all $t\in\left(0,T'\right)$ because of (\ref{eq:CH-Part2}) and (\ref{eq:Hilliardapp}). We obtained this equality by using
integration by parts in (\ref{eq:CHS-Premaj}) and noting that the
boundary integrals vanish due to the Dirichlet boundary conditions
satisfied by $\varphi$, $\mu_{A}^{\epsilon}$ and $\mu^{\epsilon}$.

Using Theorem \ref{specHillmod}, we obtain 
\begin{align}
&\int_{\Omega}\epsilon\left|\nabla R\right|^{2}+\epsilon^{-1}f''\left(c_{A}^{\epsilon}\right)R^{2}\d x  \geq C_{1}\left(\epsilon\left\Vert R\right\Vert _{L^{2}(\Omega)}^{2}+\epsilon^{-1}\left\Vert R\right\Vert _{L^{2}\left(\Omega\backslash\Gamma_{t}(\delta)\right)}+\epsilon\left\Vert \nabla^{\Gamma}R\right\Vert _{L^{2}\left(\Gamma_{t}(\delta)\right)}^{2}\right)\nonumber \\
 & \qquad+C_{2}\left(\epsilon^{3}\left\Vert \nabla R\right\Vert _{L^{2}(\Omega)}^{2}+\epsilon\left\Vert \nabla R\right\Vert _{L^{2}\left(\Omega\backslash\Gamma_{t}(\delta)\right)}^{2}\right)-C_{3}\left\Vert \nabla\varphi\right\Vert _{L^{2}(\Omega)}^{2}\label{eq:CH-spectral-1}
\end{align}
and due to the assumptions on $f$, \cite[Lemma 2.2]{abc} yields 
\[
\frac{1}{\epsilon}\int_{\Omega}\mathcal{N}\left(c_{A}^{\epsilon},R\right)R\d x\geq-\frac{C}{\epsilon}\int_{\Omega}\left|R\right|^{3}\d x.
\]
Plugging these observations into (\ref{eq:CHS-Majoreq}) enables us
to get 
\begin{align}
 & \frac{1}{2}\frac{\d}{\d t}\int_{\Omega}\left|\nabla\varphi\right|^{2}\d x+C_{1}\left(\left\Vert \left(\epsilon R,\epsilon^{3}\nabla R\right)\right\Vert _{L^{2}(\Omega)}^{2}+\left\Vert \left(\epsilon^{-1}R,\epsilon\nabla R\right)\right\Vert _{L^{2}\left(\Omega\backslash\Gamma_{t}(\delta)\right)}+\epsilon\left\Vert \nabla^{\Gamma}R\right\Vert _{L^{2}\left(\Gamma_{t}(\delta)\right)}^{2}\right)\nonumber \\
 & \leq C_{2}\left\Vert \nabla\varphi\right\Vert _{L^{2}(\Omega)}^{2}+\mathcal{RS},\label{eq:CHS-pr=0000E4gron}
\end{align}
where 
\begin{align*}
\mathcal{RS} & :=\left|\int_{\Omega}\left(\left(\mathbf{v}_{A}^{\epsilon}-\overline{\mathbf{v^{\epsilon}}}\right)\cdot\nabla c_{A}^{\epsilon}+\rc-\twz\cdot\nabla c_{A}^{\epsilon}+\mathcal{R}^{\epsilon}-\mathbf{v}^{\epsilon}\cdot\nabla R-\mathbf{v}_{err}^{\epsilon}\cdot\nabla c_{A}^{\epsilon}\right)\varphi\d x\right|.\\
 & \quad+\frac{C_{3}}{\epsilon}\int_{\Omega}\left|R\right|^{3}\d x+\left|\int_{\Omega}R\rh\d x\right|.
\end{align*}
Integrating (\ref{eq:CHS-pr=0000E4gron}) over $\left(0,T'\right)$
and using Gronwall's inequality, we get

\begin{align}
 & \sup_{0\leq\tau\leq T'}\left\Vert \nabla\varphi\right\Vert _{L^{2}(\Omega)}^{2}+\left\Vert \left(\epsilon R,\epsilon^{3}\left|\nabla R\right|\right)\right\Vert _{L^{2}\left(\Omega_{T'}\right)}^{2}+\left\Vert \left(\epsilon^{-1}R,\epsilon\left|\nabla R\right|\right)\right\Vert _{L^{2}\left(\Omega\backslash\Gamma\left(\delta;T'\right)\right)}^{2}\nonumber \\
 & \quad+\epsilon\left\Vert \nabla^{\Gamma}R\right\Vert _{L^{2}\left(\Gamma\left(\delta;T'\right)\right)}^{2}\le C\left(T_{0}\right)\left(\left\Vert \nabla\varphi\left(.,0\right)\right\Vert _{L^{2}(\Omega)}^{2}+\int_{0}^{T'}\mathcal{RS}\d t\right)\label{eq:CHS-zentral}
\end{align}
for some positive constant $C\left(T_{0}\right)>0$. On the other
hand, (\ref{eq:CHS-Majoreq}) together with Gronwall's inequality
and (\ref{eq:phi0}) also implies 
\begin{equation}
\mathcal{E}\left(R,T'\right)\leq C\left(T_{0}\right)\left(\left\Vert \nabla\varphi\left(.,0\right)\right\Vert _{L^{2}(\Omega)}^{2}+\int_{0}^{T'}\mathcal{RS}\d t\right).\label{eq:CHS-Zusatz}
\end{equation}
The idea now is to show that we may choose $\epsilon_{0}>0$ and $T\in\left(0,T_{0}\right]$
in the definition of $T_{\epsilon}$ so small, that 
\begin{align*}
C\left(T_{0}\right)\left(\left\Vert \nabla\varphi\left(.,0\right)\right\Vert _{L^{2}(\Omega)}^{2}+\int_{0}^{T_{\epsilon}}\mathcal{RS}\d t\right) & <K^{2}\epsilon^{2M}.
\end{align*}
holds for all $\epsilon\in\left(0,\epsilon_{0}\right)$. To this end we have to estimate $\mathcal{RS}$ in the following.

Due to (\ref{eq:remcahn})- (\ref{eq:remhill}) and since
(\ref{eq:Main-est}) holds true for $T_{\epsilon}$, we get 
\[
\int_{0}^{T_{\epsilon}}\left|\int_{\Omega}R\rh\d x\right|\d t+\int_{0}^{T_{\epsilon}}\left|\int_{\Omega}\rc\varphi\d x\right|\d t\leq C(K)C\left(T,\epsilon\right)\epsilon^{2M}.
\]
Moreover, we immediately get
\begin{align*}
\int_{0}^{T_{\epsilon}}\left|\int_{\Omega}\left(\left(\mathbf{v}_{A}^{\epsilon}-\overline{\mathbf{v^{\epsilon}}}\right)\cdot\nabla c_{A}^{\epsilon}+\mathbf{v}_{err}^{\epsilon}\cdot\nabla c_{A}^{\epsilon}+\mathcal{R}^{\epsilon}+\twz\cdot\nabla c_{A}^{\epsilon}\right)\varphi\d x\right|\d t & \leq C(K)C\left(T;\epsilon\right)\epsilon^{2M},
\end{align*}
as a consequence of Corollary \ref{cor:vehler} and Lemmata \ref{lem:Konvekionfinal} and
\ref{w2eps}. Here $C\left(T,\epsilon\right)\rightarrow0$ as $\left(T,\epsilon\right)\rightarrow0$.

Moreover, as a consequence of Lemma \ref{dim2einbett} and H\"older's inequality
we have
\begin{align}
 & \int_{0}^{T_{\epsilon}}\left\Vert R\right\Vert _{L^{3}\left(\gt\right)}^{3}\d t \leq C\left(\left\Vert R\right\Vert _{L^{2}\left(0,T_{\epsilon};L^{2}\left(\gt\right)\right)}+\left\Vert \nabla^{\Gamma}R\right\Vert _{L^{2}\left(0,T_{\epsilon};L^{2}\left(\gt\right)\right)}\right)^{\frac{1}{2}}\nonumber \\
 & \quad\cdot\left(\left\Vert R\right\Vert _{L^{2}\left(0,T_{\epsilon};L^{2}\left(\gt\right)\right)}+\left\Vert \partial_{\mathbf{n}}R\right\Vert _{L^{2}\left(0,T_{\epsilon};L^{2}\left(\gt\right)\right)}\right)^{\frac{1}{2}}\left\Vert R\right\Vert _{L^{4}\left(0,T_{\epsilon};L^{2}\left(\gt\right)\right)}^{2}.\label{eq:linabin-1}
\end{align}
 Since $\left\Vert R\right\Vert _{L^{2}(\Omega)}^{2}\leq\left\Vert \nabla\varphi\right\Vert _{L^{2}(\Omega)}\left\Vert \nabla R\right\Vert _{L^{2}(\Omega)}$, we deduce 
\begin{align}
\left\Vert R\right\Vert _{L^{4}\left(0,T_{\epsilon};L^{2}(\Omega)\right)}^{2} & \leq\sup_{\tau\in\left(0,T_{\epsilon}\right)}\left\Vert \nabla\varphi\right\Vert _{L^{2}(\Omega)}\left\Vert \nabla R\right\Vert _{L^{2}\left(\Omega_{T_{\epsilon}}\right)}.\label{eq:CHS-L4L2-1}
\end{align}
Because of (\ref{eq:Main-est}) and the definition
of $T_{\epsilon}$, this implies 
\[
\frac{1}{\epsilon}\int_{0}^{T_{\epsilon}}\left\Vert R\right\Vert _{L^{3}\left(\gt\right)}^{3}\d t<\frac{1}{\epsilon}CK^{3}\epsilon^{\frac{1}{2}M-\frac{1}{4}}\epsilon^{\frac{1}{2}M-\frac{3}{4}}\epsilon^{M}\epsilon^{M-\frac{3}{2}}=CK^{3}\epsilon^{3M-\frac{7}{2}}\leq CK^{3}\epsilon^{2M+\frac{1}{2}}
\]
since $M\geq4$. On the other hand, we have, for $\epsilon>0$
small enough, 
\begin{align}
\frac{1}{\epsilon}\int_{0}^{T_{\epsilon}}\left\Vert R\right\Vert _{L^{3}\left(\Omega\backslash\Gamma_{t}(\delta)\right)}^{3}\d t & \leq\frac{1}{\epsilon}C\left\Vert R\right\Vert _{L^{2}\left(0,T_{\epsilon};H^{1}\left(\Omega\backslash\gt\right)\right)}\left\Vert R\right\Vert _{L^{4}\left(0,T_{\epsilon};L^{2}\left(\Omega\backslash\gt\right)\right)}^{2}\nonumber \\
 & \leq\frac{1}{\epsilon}CK^{3}\epsilon^{M-\frac{1}{2}}\epsilon^{2M-\frac{3}{2}},\label{eq:Rdreioutside}
\end{align}
where we used the Gagliardo-Nirenberg interpolation theorem, (\ref{eq:CHS-L4L2-1})
and (\ref{eq:Main-est}). As $M\geq4$, the estimate follows.

For the last term in $\mathcal{RS}$ we have 
\begin{align}
\int_{0}^{T_{\epsilon}}\left|\int_{\Omega}\mathbf{v}^{\epsilon}\cdot\nabla R\varphi\d x\right|\d t & =\int_{0}^{T_{\epsilon}}\left|\int_{\Omega}\mathbf{v}^{\epsilon}\cdot\nabla\varphi R\d x\right|\d t\nonumber \\
 & \leq\int_{0}^{T_{\epsilon}}\left|\int_{\Omega}\mathbf{v}_{A}^{\epsilon}\cdot\nabla\varphi R\d x\right|+\left|\int_{\Omega}\left(\mathbf{v}^{\epsilon}-\mathbf{v}_{A}^{\epsilon}\right)\cdot\nabla\varphi R\d x\right|\d t.\label{eq:vepsi}
\end{align}
Before we continue with the estimates, we introduce $\hat{\mathbf{v}}_{A}^{\epsilon}:=\mathbf{v}_{A}^{\epsilon}-\epsilon^{M-\frac{1}{2}}\mathbf{v}_{A,M-\frac{1}{2}}^{\epsilon}\in L^{\infty}\left(\Omega_{T_{0}}\right)$.
First of all we have
\begin{align}\nonumber
  &\int_{0}^{T_{\epsilon}}\left|\int_{\Omega}\!\hat{\mathbf{v}}_{A}^{\epsilon}\cdot\nabla\varphi R\d x\right|\d t\leq \int_{0}^{T_{\epsilon}}\left|\int_{\Omega}\gamma\hat{\mathbf{v}}_{A}^{\epsilon}\cdot\nabla\varphi R\d x\right|\d t  \\
  &\qquad +\int_{0}^{T_{\epsilon}}\!\int_{\Omega}\left|\nabla\left(\left(1-\gamma\right)\hat{\mathbf{v}}_{A}^{\epsilon}\right):\left(\nabla\varphi\otimes\nabla\varphi\right)\right|+\!\left|\left(1-\gamma\right)\hat{\mathbf{v}}_{A}^{\epsilon}\cdot\nabla\left(\frac{\left|\nabla\varphi\right|^{2}}{2}\right)\right|\d x\d t,\label{eq:vepsihilf}
\end{align}
where we used $-\Delta\varphi=R$. We note that we introduced $\gamma$ since $\hat{\mathbf{v}}_{A}^{\epsilon}$
does not satisfy Dirichlet boundary conditions (nor does $\varphi$
satisfy Neumann boundary conditions).

Now $\left|\nabla\hat{\mathbf{v}}_{A}^{\epsilon}(x,t)\right|\leq\left|\xi(d_{\Gamma}(x,t))\partial_{\rho}\mathbf{v}_{0}\left(\rho(x,t),x,t\right)\frac{1}{\epsilon}\right|+C(K)$,
which is a consequence of the uniform boundedness of the terms $\mathbf{v}_{k}$,
$\mathbf{v}_{O,\mathbf{B}}$ and of $\left\Vert h_{A}^{\epsilon}\right\Vert _{C^{0}\left(0,T_{\epsilon};C^{1}(\Gamma_{t}(2\delta))\right)}\leq C(K)$
(see (\ref{eq:haepglm})). Moreover, by (\ref{eq:0teordnung}), and
since $d_{\Gamma}(x,t)=\epsilon\left(\rho(x,t)+h_{A}^{\epsilon}(x,t)\right)$
for $(x,t)\in\Gamma(2\delta)$, we have 
\begin{equation}
\left|\partial_{\rho}\mathbf{v}_{0}\left(\rho(x,t),x,t\right)\right|\leq\epsilon|\eta'(\rho(x,t))|\left|\rho(x,t)+h_{A}^{\epsilon}(x,t)\right|\left|\frac{\mathbf{v}^{+}(x,t)-\mathbf{v}^{-}(x,t)}{d_{\Gamma}(x,t)}\right|\label{eq:delrhov0}
\end{equation}
for all $(x,t)\in\Gamma(2\delta)$.

Due to $\|\eta'(\rho)\rho|<C$ for all $\rho\in\mathbb{R}$
and $\mathbf{v}^{+}=\mathbf{v}^{-}$ on $\Gamma$  this results in
\[
\int_{0}^{T_{\epsilon}}\left|\int_{\Omega}\nabla\left(\left(1-\gamma\right)\hat{\mathbf{v}}_{A}^{\epsilon}\right):\left(\nabla\varphi\otimes\nabla\varphi\right)\d x\right|\d t\leq C(K)T_{\epsilon}\left\Vert \nabla\varphi\right\Vert _{L^{\infty}\left(0,T_{\epsilon};L^{2}(\Omega)\right)}^{2}\leq C(K)T_{\epsilon}\epsilon^{2M}
\]
by (\ref{eq:Main-est-b}) and the facts that $\mathbf{\hat{v}}_{A}^{\epsilon}\in L^{\infty}\left(\Omega_{T_{0}}\right)$
and $\gamma$, $\gamma'$ are bounded.

Concerning the second term on the right-hand side of (\ref{eq:vepsihilf}),
we note that $\left\Vert \operatorname{div}\left(\hat{\mathbf{v}}_{A}^{\epsilon}\right)\right\Vert _{L^{\infty}\left(\Omega_{T_{\epsilon}}\right)}\leq C(K)$
as a consequence of (\ref{eq:delrhov0}) and (\ref{eq:haepglm}).
Thus
\[
\int_{0}^{T_{\epsilon}}\left|\int_{\Omega}\left(\left(1-\gamma\right)\hat{\mathbf{v}}_{A}^{\epsilon}\right)\cdot\nabla\left(\frac{1}{2}\left|\nabla\varphi\right|^{2}\right)\d x\right|\d t\leq C(K)T_{\epsilon}\left\Vert \nabla\varphi\right\Vert _{L^{\infty}\left(0,T_{\epsilon};L^{2}(\Omega)\right)}^{2}\leq C(K)T_{\epsilon}\epsilon^{2M}.
\]
For the third term on the right-hand side of (\ref{eq:vepsihilf}),
we calculate
\begin{align*}
\int_{0}^{T_{\epsilon}}\left|\int_{\Omega}\gamma\hat{\mathbf{v}}_{A}^{\epsilon}\cdot\nabla\varphi R\d x\right|\d t & \leq CT_{\epsilon}^{\frac{1}{2}}\left\Vert \nabla\varphi\right\Vert _{L^{\infty}\left(0,T_{\epsilon};L^{2}(\Omega)\right)}\left\Vert R\right\Vert _{L^{2}\left(\Omega_{T_{\epsilon}}\backslash\Gamma\left(2\delta;T_{\epsilon}\right)\right)},
\end{align*}
so (\ref{eq:Main-est}) implies a suitable estimate. Regarding (\ref{eq:vepsi}),
we have
\begin{align*}
  &\int_{0}^{T_{\epsilon}}\left|\int_{\Omega}\epsilon^{M-\frac{1}{2}}\mathbf{v}_{A,M-\frac{1}{2}}^{\epsilon}\cdot\nabla\varphi R\d x\right|\d t\\
  & \leq\epsilon^{M-\frac{1}{2}}\left\Vert \mathbf{v}_{A,M-\frac{1}{2}}^{\epsilon}\right\Vert _{L^{2}\left(0,T_{\epsilon};L^{\infty}(\Omega)\right)}\left\Vert \nabla\varphi\right\Vert _{L^{\infty}\left(0,T_{\epsilon};L^{2}(\Omega)\right)}\cdot\left\Vert R\right\Vert _{L^{2}\left(0,T_{\epsilon};L^{2}(\Omega)\right)} <C(K)\epsilon^{2M+\frac{1}{2}}
\end{align*}
as $M\geq4$. Here we used that
$\Vert \mathbf{v}_{M-\frac{1}{2}}^{\pm,\epsilon}\Vert _{L^{2}\left(0,T_{\epsilon};L^{\infty}\left(\Omega^{\pm}(t)\cup\Gamma_{t}(2\delta)\right)\right)}\leq C(K)$
due to (\ref{eq:muv0,5est}) and  $H^{2}(\Omega)\hookrightarrow L^{\infty}(\Omega)$.
Hence
\begin{align}
 & \int_{0}^{T_{\epsilon}}\left|\int_{\Omega}\left(\mathbf{v}^{\epsilon}-\mathbf{v}_{A}^{\epsilon}\right)\cdot\nabla\varphi R\d x\right|\d t\nonumber \\
 & \leq\left(\left\Vert \twe\right\Vert _{L^{2}\left(0,T_{\epsilon};L^{4}(\Omega)\right)}+\left\Vert \overline{\mathbf{v}^{\epsilon}}-\mathbf{v}_{A}^{\epsilon}\right\Vert _{L^{2}\left(0,T_{\epsilon};L^{4}(\Omega)\right)}\right)\left\Vert \nabla\varphi\right\Vert _{L^{\infty}\left(0,T_{\epsilon};L^{2}(\Omega)\right)}\left\Vert R\right\Vert _{L^{2}\left(0,T_{\epsilon};L^{4}(\Omega)\right)}\nonumber \\
 & \quad+\int_{0}^{T_{\epsilon}}\left|\int_{\Omega}\twz\cdot\nabla R\varphi\d x\right|\d t+\int_{0}^{T_{\epsilon}}\left|\int_{\Omega}\mathbf{v}_{err}^{\epsilon}\cdot\nabla R\varphi\d x\right|\d t\nonumber \\
 & \leq C(K)\left(\epsilon^{2M+\frac{1}{2}}+\epsilon^{2M+\frac{1}{2}}+C\left(T,\epsilon\right)\epsilon^{2M}\right)+\int_{0}^{T_{\epsilon}}\left|\int_{\Omega}\twz\cdot\nabla R\varphi\d x\right|\d t\label{eq:teil}
\end{align}
because of $\mathbf{v}_{err}^{\epsilon}=\mathbf{v}^{\epsilon}-\left(\overline{\mathbf{v}^{\epsilon}}+\tilde{\mathbf{w}}_{1}^{\epsilon}+\tilde{\mathbf{w}}_{2}^{\epsilon}\right)$, Theorem \ref{vehler} 1), (\ref{eq:w1epsab}),
(\ref{eq:Main-est}), (\ref{eq:vepserr-nablaR-psi}), and $M\geq4$. Regarding the $\twz$ term we first
note that for $\kappa>0$ we have 
\begin{align}
\left\Vert \nabla R\right\Vert _{L^{2}\left(0,T_{\epsilon};L^{2+\kappa}(\Omega)\right)} & \leq C\left(\left\Vert \nabla R\right\Vert _{L^{2}\left(\Omega_{T_{\epsilon}}\right)}^{1-\frac{\kappa}{2+\kappa}}\left\Vert \Delta R\right\Vert _{L^{2}\left(\Omega_{T_{\epsilon}}\right)}^{\frac{\kappa}{2+\kappa}}+\left\Vert \nabla R\right\Vert _{L^{2}\left(\Omega_{T_{\epsilon}}\right)}\right)\nonumber \\
 & \leq C(K)\left(\epsilon^{M-\frac{3}{2}}\epsilon^{-\left(M+2\right)\frac{\kappa}{2+\kappa}}\right)\label{eq:nablaRkappa}
\end{align}
where we used $\Vert R\Vert _{H^{2}(\Omega)}\leq C\left\Vert \Delta R\right\Vert _{L^{2}(\Omega)}$
and $\Vert \Delta R\Vert _{L^{2}\left(\Omega_{T_{\epsilon}}\right)}\leq C(K)\epsilon^{-\frac{7}{2}}$
as in (\ref{eq:laplaceR}). Thus, we may estimate for $\kappa>0$
and $q\in(\frac{2+\kappa}{\left(2+\kappa\right)-1},2)$ 
\begin{align*}
\int_{0}^{T_{\epsilon}}\left|\int_{\Omega}\twz\cdot\nabla R\varphi\d x\right|\d t & \leq\left\Vert \twz\right\Vert _{L^{2}\left(0,T_{\epsilon};L^{q}(\Omega)\right)}\left\Vert \nabla R\right\Vert _{L^{2}\left(0,T_{\epsilon};L^{2+\kappa}(\Omega)\right)}\left\Vert \varphi\right\Vert _{L^{\infty}\left(0,T_{\epsilon};H^{1}(\Omega)\right)}\\
 & \leq C(K)\epsilon^{3M-\frac{5}{2}}\epsilon^{-\left(M+2\right)\frac{\kappa}{2+\kappa}}\leq C(K)\epsilon^{2M+\alpha}
\end{align*}
for some $\alpha>0$, where we used (\ref{eq:w2epsab}), (\ref{eq:Main-est-b}), (\ref{eq:nablaRkappa}), $M\geq4$ and that $\kappa>0$ can be chosen arbitrarily small. 

Because of (\ref{eq:teil}), we get $\int_{0}^{T_{\epsilon}}\left|\int_{\Omega}\left(\mathbf{v}^{\epsilon}-\mathbf{v}_{A}^{\epsilon}\right)\cdot\nabla\varphi R\d x\right|\d t\leq C(K)C\left(T,\epsilon\right)\epsilon^{2M}$,
which concludes the estimates for $\mathcal{RS}$. Since (\ref{eq:CHS-zentral})
and (\ref{eq:CHS-Zusatz}) do not imply estimates of the kind (\ref{eq:Main5})
and (\ref{eq:Main6}), we need to apply another strategy and test with $\gamma^{2}R$ in the following.

Let again $T'\in\left(0,T_{0}\right]$. Multiplying the difference
of the differential equations (\ref{eq:CH-Part1}) and (\ref{eq:Cahnapp})
by $\gamma^{2}R$ and integrating the result over $\Omega$ yields
\begin{align}
0 & =\frac{1}{2}\int_{\Omega}\frac{d}{dt}\left(R^{2}\right)\gamma^{2}\d x+\int_{\Omega_{T}}\gamma^{2}R\left(\mathbf{v}^{\epsilon}\cdot\nabla R+\mathbf{v}_{err}^{\epsilon}\cdot\nabla c_{A}^{\epsilon}+\left(\overline{\mathbf{v}^{\epsilon}}-\mathbf{v}_{A}^{\epsilon}+\twe+\twz\right)\cdot\nabla c_{A}^{\epsilon}\right)\d x\nonumber \\
 & \quad+\int_{\Omega}\gamma^{2}R\rc-\Delta\left(\gamma^{2}R\right)\left(-\epsilon\Delta R+\frac{1}{\epsilon}\left(f''\left(c_{A}^{\epsilon}\right)R+\mathcal{N}\left(c_{A}^{\epsilon},R\right)\right)-\rh\right)\d x,\label{eq:alive}
\end{align}
where we used $\text{supp}\gamma\cap\text{supp}\xi\circ d_{\Gamma_{t}}=\emptyset$
for all $t\in\left[0,T_{0}\right]$, (\ref{eq:veps-vaeps}), integration
by parts and $R=\mu^{\epsilon}=\mu_{A}^{\epsilon}=0$ on $\partial_{T_{0}}\Omega$.

As $c_{A}^{\epsilon}=-1+\mathcal{O}\left(\epsilon\right)$ in $L^{\infty}\left(\partial_{T_{0}}\Omega\left(\frac{\delta}{2}\right)\right)$,
we have $f''\left(c_{A}^{\epsilon}(x,t)\right)=f''\left(-1\right)+\epsilon\tilde{f}(x,t)$
for $(x,t)\in\partial_{T_{0}}\Omega\left(\frac{\delta}{2}\right)$
by a Taylor expansion, where $\tilde{f}\in L^{\infty}\left(\partial_{T_{0}}\Omega\left(\frac{\delta}{2}\right)\right)$.
Moreover, 
\begin{align*}
\nabla\left(\gamma^{2}R\right) & =2\gamma R\nabla\gamma+\gamma^{2}\nabla R,\quad\Delta\left(\gamma^{2}R\right)=\Delta\left(\gamma^{2}\right)R+4\gamma\nabla\gamma\cdot\nabla R+\gamma^{2}\Delta R
\end{align*}
and we find
\begin{align}\nonumber
  &\frac{1}{\epsilon}\int_{\Omega}-\Delta\left(\gamma^{2}R\right)f''\left(c_{A}^{\epsilon}\right)R\d x \\
  & =\frac{1}{\epsilon}f''\left(-1\right)\left\Vert \gamma\nabla R\right\Vert _{L^{2}(\Omega)}^{2}+\frac{1}{\epsilon}\int_{\Omega}f''\left(-1\right)R\nabla\left(\gamma^{2}\right)\cdot\nabla R\d x-\int_{\Omega}\Delta\left(\gamma^{2}R\right)\tilde{f}R\d x,\label{eq:Umformung-f}
\end{align}
where we used $R=0$ on $\partial_{T_{0}}\Omega$.
Moreover, we have
\begin{align*}
\nabla\mathcal{N}\left(c_{A}^{\epsilon},R\right) & =k_{f}\nabla c_{A}^{\epsilon}R^{2}+\left(f^{\left(3\right)}\left(c_{A}^{\epsilon}\right)R+k_{f}R^{2}\right)\nabla R,
\end{align*}
due to \eqref{eq:N} and  $k_{f}=\frac{f^{\left(4\right)}(c_A)}2$. This
yields 
\begin{align}
\int_{\Omega}-\Delta\left(\gamma^{2}R\right)\frac{1}{\epsilon}\mathcal{N}\left(c_{A}^{\epsilon},R\right)\d(x,t) & =\frac{1}{\epsilon}\int_{\Omega}k_{f}\left(\left|\gamma\left(\nabla R\right)R\right|^{2}+\nabla\left(\gamma^{2}\right)R^{3}\cdot\nabla R\right)\d x\nonumber \\
 & \quad+\frac{1}{\epsilon}\int_{\Omega}\nabla\left(\gamma^{2}R\right)\cdot\left(k_{f}\nabla c_{A}^{\epsilon}R^{2}+f^{\left(3\right)}\left(c_{A}^{\epsilon}\right)R\nabla R\right)\d x\nonumber \\
 & =\frac{k_{f}}{\epsilon}\left\Vert \gamma\left|\nabla R\right|R\right\Vert _{L^{2}(\Omega)}^{2}+\frac{1}{\epsilon}\int_{\Omega}\mathcal{N}^{\nabla}\left(c_{A}^{\epsilon},R\right)\d x,\label{eq:Umformung-Nabla}
\end{align}
where the boundary terms due to integration by parts vanish since
$f'\left(-1\right)=R(x,t)=0$ and $c_{A}^{\epsilon}(x,t)=-1$
for $(x,t)\in\partial_{T_{0}}\Omega$. Here we used the
notation 
\begin{equation}
\mathcal{N}^{\nabla}\left(c_{A}^{\epsilon},R\right):=k_{f}\nabla\left(\gamma^{2}\right)R^{3}\cdot\nabla R+\nabla\left(\gamma^{2}R\right)\cdot\left(k_{f}\nabla c_{A}^{\epsilon}R^{2}+f^{\left(3\right)}\left(c_{A}^{\epsilon}\right)R\nabla R\right).\label{eq:N-Nabla}
\end{equation}
Additionally, we compute
\begin{equation}
\int_{\Omega}-\Delta\left(\gamma^{2}R\right)\left(-\epsilon\Delta R\right)\d x=\epsilon\left\Vert \gamma\Delta R\right\Vert _{L^{2}(\Omega)}^{2}+\epsilon\int_{\Omega}4\gamma\nabla\gamma\cdot\nabla R\Delta R+\Delta\left(\gamma^{2}\right)R\Delta R\d x.\label{eq:Umformung Laplace R}
\end{equation}
 Plugging (\ref{eq:Umformung-f}), (\ref{eq:Umformung-Nabla}) and
(\ref{eq:Umformung Laplace R}) (noting that $k_{f},f''\left(-1\right)>0$)
into (\ref{eq:alive}) and integrating in time yields
\begin{align}
 & \sup_{t\in\left(0,T'\right)}\left\Vert \gamma R\left(.,t\right)\right\Vert _{L^{2}(\Omega)}^{2}+\epsilon\left\Vert \gamma\Delta R\right\Vert _{L^{2}\left(\Omega_{T'}\right)}^{2}+\frac{k_{f}}{\epsilon}\left\Vert \left(\gamma\nabla R,\gamma R\nabla R\right)\right\Vert _{L^{2}\left(\Omega_{T'}\right)}^{2}\nonumber \\
 & \leq\left\Vert \gamma R\left(.,0\right)\right\Vert _{L^{2}(\Omega)}^{2}+C_{1}\int_{0}^{T'}\left|\int_{\Omega}\!\gamma^{2}R\left(\mathbf{v}^{\epsilon}\cdot\nabla R+\left(\mathbf{v}_{err}^{\epsilon}+\overline{\mathbf{v}^{\epsilon}}-\mathbf{v}_{A}^{\epsilon}+\twe+\twz\right)\cdot\nabla c_{A}^{\epsilon}\right)\d x\right|\d t\nonumber \\
 & \quad+C_{2}\int_{0}^{T'}\left|\int_{\Omega}\gamma^{2}R\rc+\epsilon\left(\Delta\left(\gamma^{2}\right)R+4\gamma\nabla\gamma\cdot\nabla R\right)\Delta R+\frac{1}{\epsilon}\mathcal{N}^{\nabla}\left(c_{A}^{\epsilon},R\right)\d x\right|\d t\nonumber \\
 & \quad+C_{3}\int_{0}^{T'}\left|\int_{\Omega}\Delta\left(\gamma^{2}R\right)\left(\tilde{f}R-\rh\right)+R\nabla\left(\gamma^{2}\right)\cdot\nabla R\frac{1}{\epsilon}f''\left(-1\right)\d x\right|\d t.\label{eq:Main-New}
\end{align}
If we may now give suitable estimates for the right-hand side of (\ref{eq:Main-New}),
replacing $T'$ by $T_{\epsilon}$, we get (\ref{eq:Main5}) and (\ref{eq:Main6}).

Now we estimate the right-hand side of (\ref{eq:Main-New}). Starting from the last term in (\ref{eq:Main-New}), we have 
\begin{align*}
\int_{0}^{T_{\epsilon}}\left\Vert \nabla\left(\gamma^{2}\right)R\nabla R\frac{f''\left(-1\right)}{\epsilon}\right\Vert _{L^{1}(\Omega)}\d t & \leq\frac{C}{\epsilon}\left\Vert \gamma\nabla R\right\Vert _{L^{2}\left(\Omega_{T_{\epsilon}}\right)}\left\Vert \nabla\gamma R\right\Vert _{L^{2}\left(\Omega_{T_{\epsilon}}\right)}\leq C(K)\epsilon^{2M-\frac{1}{2}}
\end{align*}
due to (\ref{eq:Main-est-a}) and (\ref{eq:Main-est-d}). For the next term, we note that $\rh=\rhb$ in $\partial_{T_{0}}\Omega\left(\frac{\delta}{2}\right)$
and use (\ref{eq:bdry-remainder}) to conclude

\begin{align*}
&\int_{0}^{T_{\epsilon}}\left|\int_{\Omega}\Delta\left(\gamma^{2}R\right)\rh\d x\right|\d t  \le C\left\Vert \left(\gamma\Delta R,\nabla R,R\right)\right\Vert _{L^{2}\left(\partial_{T_{\epsilon}}\Omega\left(\frac{\delta}{2}\right)\right)}\epsilon^{M+1}\\
 &\qquad +C_{2}\epsilon^{M-\frac{1}{2}}\left\Vert \nabla\mu_{M-\frac{1}{2}}^{-}\right\Vert _{L^{2}\left(\Omega_{T_{\epsilon}}^{-}\right)}\left\Vert \left(\gamma\nabla R,R\right)\right\Vert _{L^{2}\left(\partial_{T_{\epsilon}}\Omega\left(\frac{\delta}{2}\right)\right)}\leq C(K)\epsilon^{2M-\frac{1}{2}},
\end{align*}
where we used integration by parts, $\mu_{M-\frac{1}{2}}^{-}=0$ on $\partial_{T_{\epsilon}}\Omega$, (\ref{eq:Main-est-a}), (\ref{eq:Main-est-d})
and (\ref{eq:muv0,5est}). Moreover,
\begin{align*}
  \int_{0}^{T_{\epsilon}}\left|\int_{\Omega}\Delta\left(\gamma^{2}R\right)\tilde{f}R\d x\right|\d t & \leq C\left\Vert \left(\gamma\Delta R,\nabla R,R\right)\right\Vert _{L^{2}\left(\partial_{T_{\epsilon}}\Omega\left(\frac{\delta}{2}\right)\right)}\left\Vert R\right\Vert _{L^{2}\left(\partial_{T_{\epsilon}}\Omega\left(\frac{\delta}{2}\right)\right)}\le C(K)\epsilon^{2M-\frac{1}{2}}.
\end{align*}
Skipping $\mathcal{N}^{\nabla}\left(c_{A}^{\epsilon},R\right)$ for
now, we next estimate 
\begin{align*}
\int_{0}^{T_{\epsilon}}\left|\int_{\Omega}\epsilon4\left(\nabla\gamma\cdot\nabla R\right)\gamma\Delta R\d x\right|\d t & \leq C\epsilon\left\Vert \gamma\Delta R\right\Vert _{L^{2}\left(\Omega_{T_{\epsilon}}\right)}\left\Vert \nabla R\right\Vert _{L^{2}\left(\partial_{T_{\epsilon}}\Omega\left(\frac{\delta}{2}\right)\right)}\leq C_{1}\epsilon^{2M-\frac{1}{2}}
\end{align*}
due to (\ref{eq:Main-est-a}) and (\ref{eq:Main-est-d}). Additionally,
\begin{align*}
&\int_{0}^{T_{\epsilon}}\left|\int_{\Omega}\epsilon\Delta\left(\gamma^{2}\right)R\Delta R\d x\right|\d t  =\int_{0}^{T_{\epsilon}}\left|\int_{\Omega}\epsilon\left(\nabla\Delta\left(\gamma^{2}\right)R+\Delta\left(\gamma^{2}\right)\nabla R\right)\cdot\nabla R\d x\right|\d t\\
 & \qquad \leq C\epsilon\left(\left\Vert R\right\Vert _{L^{2}\left(\partial_{T_{\epsilon}}\Omega\left(\frac{\delta}{2}\right)\right)}\left\Vert \nabla R\right\Vert _{L^{2}\left(\partial_{T_{\epsilon}}\Omega\left(\frac{\delta}{2}\right)\right)}+\left\Vert \nabla R\right\Vert _{L^{2}\left(\partial_{T_{\epsilon}}\Omega\left(\frac{\delta}{2}\right)\right)}^{2}\right)=O(\eps^{2M})
\end{align*}
because of $R|_{\partial\Omega}=0$ and (\ref{eq:Main-est-b}). Now
\begin{align*}
\int_{0}^{T_{\epsilon}}\left|\int_{\Omega}\gamma^{2}Rr_{CH1}^{\epsilon}\d x\right|\d t & \le C\left\Vert R\right\Vert _{L^{2}\left(\partial_{T_{\epsilon}}\Omega\left(\frac{\delta}{2}\right)\right)}\left\Vert \rc\right\Vert _{L^{2}\left(\partial_{T_{\epsilon}}\Omega\left(\frac{\delta}{2}\right)\right)}\le C(K)\epsilon^{2M+\frac{1}{2}}
\end{align*}
due to (\ref{eq:Main-est-a}) and (\ref{eq:rch1-rch2-Linfbdry}),
and
\begin{align*}
 & \int_{0}^{T_{\epsilon}}\left|\int_{\Omega}\gamma^{2}R\left(\overline{\mathbf{v}^{\epsilon}}-\mathbf{v}_{A}^{\epsilon}+\twe+\twz\right)\cdot\nabla c_{A}^{\epsilon}\d x\right|\d t\\
 & \leq C\epsilon\left\Vert R\right\Vert _{L^{2}\left(\partial_{T_{\epsilon}}\Omega\left(\frac{\delta}{2}\right)\right)}\left(\left\Vert \overline{\mathbf{v}^{\epsilon}}-\mathbf{v}_{A}^{\epsilon}\right\Vert _{L^{2}\left(0,T_{\epsilon};H^{1}(\Omega)\right)}+\left\Vert \twe\right\Vert _{L^{2}\left(0,T_{\epsilon};H^{1}(\Omega)\right)}\right)\\
 & \quad+C\epsilon\left\Vert R\right\Vert _{L^{2}\left(0,T_{\epsilon};L^{q'}\left(\partial\Omega\left(\frac{\delta}{2}\right)\right)\right)}\left\Vert \twz\right\Vert _{L^{2}\left(0,T_{\epsilon};L^{q}(\Omega)\right)}
\end{align*}
where $q\in\left(1,2\right)$, $\frac{1}{q'}+\frac{1}{q}=1$ and we used $\nabla c_{A}^{\epsilon}=\mathcal{O}\left(\epsilon\right)$
in $L^{\infty}\left(\partial_{T_{0}}\Omega\left(\frac{\delta}{2}\right)\right)$. Now (\ref{eq:w1epsab}), Theorem \ref{vehler}
1) and (\ref{eq:w2epsab}) together with $H^{1}\left(\partial\Omega\left(\frac{\delta}{2}\right)\right)\hookrightarrow L^{q'}\left(\partial\Omega\left(\frac{\delta}{2}\right)\right)$
and (\ref{eq:Main-est-a}) imply that the term is of order $O(\eps^{2M+\frac12})$. Next,
\begin{align*}
\int_{0}^{T_{\epsilon}}\left|\int_{\Omega}\gamma^{2}R\mathbf{v}_{err}^{\epsilon}\cdot\nabla c_{A}^{\epsilon}\d x\right|\d t & \leq\epsilon\left\Vert \gamma R\right\Vert _{L^{\infty}\left(0,T_{\epsilon};L^{2}(\Omega)\right)}\left\Vert \mathbf{v}_{err}^{\epsilon}\right\Vert _{L^{1}\left(0,T_{\epsilon};H^{1}(\Omega)\right)}\le C(K)\epsilon^{2M+\frac{1}{2}},
\end{align*}
where we again used $\nabla c_{A}^{\epsilon}=\mathcal{O}\left(\epsilon\right)$
in $L^{\infty}\left(\partial_{T_{0}}\Omega\left(\frac{\delta}{2}\right)\right)$
in the first line and (\ref{eq:vepserr-L1}), (\ref{eq:Main-est-d})
in the second line. In view of the above considerations, $\left\Vert \gamma R\left(.,0\right)\right\Vert _{L^{2}\left(\Omega_{T_{\epsilon}}\right)}^{2}\leq\frac{K^{2}}{4}\epsilon^{2M}$
(cf.\ (\ref{eq:phi0})) and (\ref{eq:Main-New}), we have two more
estimates to show:

Using the explicit form of $\mathcal{N}^{\nabla}$ given in (\ref{eq:N-Nabla}),
we calculate
\begin{align}
&\frac{1}{\epsilon}\int_{0}^{T_{\epsilon}}  \left|\int_{\Omega}\mathcal{N}^{\nabla}\left(c_{A}^{\epsilon},R\right)\d x\right|\d t\nonumber \\
 & \le\frac{1}{\epsilon}\int_{0}^{T_{\epsilon}}\left|\int_{\Omega}k_{f}\nabla\left(\gamma^{2}\right)R^{3}\cdot\nabla R\d x\right|\d t+C_{1}\left\Vert R\right\Vert _{L^{3}\left(\Omega_{T_{\epsilon}}\backslash\Gamma\left(2\delta;T_{\epsilon}\right)\right)}^{3}+C_{2}\int_{0}^{T_{\epsilon}}\int_{\Omega}|\gamma^{2}\nabla RR^{2}|\d x\d t\nonumber \\
 & \quad+C_{3}\frac{1}{\epsilon}\int_{0}^{T_{\epsilon}}\int_{\Omega}|\nabla\left(\gamma^{2}R\right)R\nabla R|\d x\d t,\label{eq:nablaNest}
\end{align}
where we again used $\nabla c_{A}^{\epsilon}=\mathcal{O}\left(\epsilon\right)$
in $L^{\infty}\left(\partial_{T_{0}}\Omega\left(\frac{\delta}{2}\right)\right)$
in the last step. Now we have
\begin{align*}
 & \frac{1}{\epsilon}\int_{0}^{T_{\epsilon}}\int_{\Omega}|k_{f}\nabla\left(\gamma^{2}\right)R^{3}\cdot\nabla R|\d x\d t\\
 & \leq\frac{1}{\epsilon}C\left\Vert \gamma R\left|\nabla R\right|\right\Vert _{L^{2}\left(\Omega_{T_{\epsilon}}\right)}\left\Vert R\right\Vert _{L^{\infty}\left(0,T_{\epsilon};L^{2}(\Omega)\right)}\left\Vert R\right\Vert _{L^{2}\left(0,T_{\epsilon};H^{1}\left(\partial\Omega\left(\frac{\delta}{2}\right)\right)\right)}\\
 & \leq C(K)\epsilon^{-1}\epsilon^{M}\epsilon^{\frac{M}{2}-\frac{1}{4}}\epsilon^{M-\frac{1}{2}}
\end{align*}
where we used $\|u\|_{L^4(\Omega)}\leq C \|u\|_{L^2(\Omega)}^{\frac12}\|u\|_{H^1(\Omega)}^{\frac12}$ and (\ref{eq:Main-est}) together with Lemma \ref{lem:R-L2-Linf}
4). The estimate follows since $M\geq4$. Next we have
$\left\Vert R\right\Vert _{L^{3}\left(\Omega_{T_{\epsilon}}\backslash\Gamma\left(2\delta;T_{\epsilon}\right)\right)}\leq C(K)\epsilon^{2M+1}$
due to (\ref{eq:Rdreioutside}) and 
\[
\int_{0}^{T_{\epsilon}}\int_{\Omega}|\gamma^{2}\nabla RR^{2}|\d x\d t\leq C\left\Vert \gamma R\nabla R\right\Vert _{L^{2}\left(\Omega_{T_{\epsilon}}\right)}\left\Vert R\right\Vert _{L^{2}\left(\partial_{T_{\epsilon}}\Omega\left(\frac{\delta}{2}\right)\right)}\leq C(K)\epsilon^{2M+\frac{1}{2}}
\]
due to (\ref{eq:Main-est}). Regarding the last term in (\ref{eq:nablaNest})
we have on the one hand 
\[
\frac{1}{\epsilon}\int_{0}^{T_{\epsilon}}\int_{\Omega}|\left(\nabla\gamma^{2}\right)R^{2}\nabla R|\d x\d t\leq C\frac{1}{\epsilon}\left\Vert \gamma R\nabla R\right\Vert _{L^{2}\left(\Omega_{T_{\epsilon}}\right)}\left\Vert R\right\Vert _{L^{2}\left(\partial_{T_{\epsilon}}\Omega\left(\frac{\delta}{2}\right)\right)}\leq C(K)\epsilon^{2M-\frac{1}{2}}
\]
as before and on the other hand 
\begin{align*}
\frac{1}{\epsilon}\int_{0}^{T_{\epsilon}}\int_{\Omega}|\gamma^{2}\left(\nabla R\right)^{2}R|\d x\d t & \leq C\frac{1}{\epsilon}\left\Vert R\right\Vert _{L^{\infty}\left(0,T_{\epsilon};L^{2}(\Omega)\right)}\left\Vert \gamma\nabla R\right\Vert _{L^{2}\left(\Omega_{T_{\epsilon}}\right)}\left\Vert \gamma\nabla R\right\Vert _{L^{2}\left(0,T_{\epsilon};H^{1}(\Omega)\right)}\\
 & \leq C(K) \epsilon^{\frac{M}{2}-\frac{1}{4}}\epsilon^{M}\epsilon^{M-1}\epsilon^{-1}=C(K)\eps^{\frac52 M-2-\frac14}
\end{align*}
where we use Lemma~\ref{lem:R-L2-Linf} 4) and (\ref{eq:Main-est}). Altogether we have $\frac{1}{\epsilon}\int_{0}^{T_{\epsilon}}\left|\int_{\Omega}\mathcal{N}^{\nabla}(c_{A}^{\epsilon},R)\d x\right|\d t\leq\epsilon^{2M-\frac{1}{2}}.$

Finally, we estimate
\begin{align*}
  & \int_{0}^{T_{\epsilon}}\left|\int_{\Omega}\gamma^{2}R\left(\mathbf{v}^{\epsilon}\cdot\nabla R\right)\d x\right|\d t \\
  &\leq\int_{0}^{T_{\epsilon}}\left|\int_{\Omega}\gamma^{2}R\left(\mathbf{v}_{err}^{\epsilon}\cdot\nabla R\right)\d x\right|
 +\int_{0}^{T_{\epsilon}}\left|\int_{\Omega}\left(\mathbf{v}^{\epsilon}-\mathbf{v}_{err}^{\epsilon}\right)\cdot\frac{1}{2}\nabla\left(\gamma^{2}\right)R^{2}\d x\right|\d t\\
 & \leq C(K)C\left(\epsilon,T_{\epsilon}\right)\epsilon^{2M-1}+\frac{1}{2}\int_{\Omega_{T_{\epsilon}}}\!\!\left|\left(\overline{\mathbf{v}^{\epsilon}}-\mathbf{v}_{A}^{\epsilon}+\twe+\twz\right)\cdot\nabla\left(\gamma^{2}\right)R^{2}\right|+\left|\mathbf{v}_{A}^{\epsilon}\cdot\nabla\left(\gamma^{2}\right)R^{2}\right|\!\d(x,t)\\
 & \leq C(K)C\left(\epsilon,T_{\epsilon}\right)\epsilon^{2M-1}+\frac{1}{2}\int_{\Omega_{T_{\epsilon}}}\left|\left(\overline{\mathbf{v}^{\epsilon}}-\mathbf{v}_{A}^{\epsilon}+\twe+\twz\right)\cdot\nabla\left(\gamma^{2}\right)R^{2}\right|\d(x,t)\\
 & \quad+\epsilon^{M-\frac{1}{2}}\frac{1}{2}\int_{0}^{T_{\epsilon}}\int_{\Omega}\left|\mathbf{v}_{M-\frac{1}{2}}^{-,\epsilon}\cdot\nabla\left(\gamma^{2}\right)R^{2}\right|\d t,
\end{align*}
where we used that $\mathbf{v}^{\epsilon}-\mathbf{v}_{err}^{\epsilon}$
is divergence free and $R|_{\partial\Omega}=0$, as well
as (\ref{eq:vepserr-R-nablaR}) and the definition of $\mathbf{v}_{err}^{\epsilon}$
in (\ref{eq:vepserrderf}). Furthermore,
we used $\mathbf{v}_{A}^{\epsilon}-\epsilon^{M-\frac{1}{2}}\mathbf{v}_{A,M-\frac{1}{2}}^{\epsilon}\in L^{\infty}\left(\Omega_{T_{0}}\right)$
and (\ref{eq:Main-est-a}). Note that $\mathbf{v}_{A,M-\frac{1}{2}}^{\epsilon}=\mathbf{v}_{M-\frac{1}{2}}^{-,\epsilon}$
in $\partial_{T_{0}}\Omega\left(\frac{\delta}{2}\right)$. We may
continue estimating
\begin{align*}
  &\int_{\Omega_{T_{\epsilon}}}\left|\left(\overline{\mathbf{v}^{\epsilon}}-\mathbf{v}_{A}^{\epsilon}+\twe\right)\cdot\nabla\left(\gamma^{2}\right)R^{2}\right|\d(x,t)\\
  & \le\left(\left\Vert \overline{\mathbf{v}^{\epsilon}}-\mathbf{v}_{A}^{\epsilon}\right\Vert _{L^{2}\left(0,T_{\epsilon};H^{1}(\Omega)\right)}+\left\Vert \twe\right\Vert _{L^{2}\left(0,T_{\epsilon};H^{1}(\Omega)\right)}\right)\cdot\left\Vert \gamma R\right\Vert _{L^{\infty}\left(0,T_{\epsilon};L^{2+\kappa}(\Omega)\right)}\left\Vert R\right\Vert _{L^{2}\left(\partial_{T_{\epsilon}}\Omega\left(\frac{\delta}{2}\right)\right)}\\
 & \leq C(K)\left(\epsilon^{M}+\epsilon^{M-\frac{1}{2}}\right)\epsilon^{M-\frac{1}{2}-\frac{\kappa}{2+\kappa}M}\epsilon^{M+\frac{1}{2}}
\end{align*}
where we used $H^{1}(\Omega)\hookrightarrow L^{s}(\Omega)$
for all $1\leq s<\infty$ in the first inequality, Theorem \ref{vehler} 1),
Lemma \ref{Wichtig} (in particular (\ref{eq:w1epsab})), Lemma \ref{lem:R-L2-Linf}
3) and (\ref{eq:Main-est-a}) in the second inequality. A suitable
estimate follows since $M\geq4$ and we may choose $\kappa>0$ arbitrarily
small. Regarding $\twz$, we choose $\kappa>0$ and $q=\frac{2+\kappa}{\left(2+\kappa\right)-1}$
and estimate
\begin{align*}
\int_{\Omega_{T_{\epsilon}}}\left|\twz\cdot\nabla\left(\gamma^{2}\right)R^{2}\right|\d(x,t) & \leq\left\Vert \twz\right\Vert _{L^{2}\left(0,T_{\epsilon};L^{q}(\Omega)\right)}\left\Vert \gamma R\right\Vert _{L^{\infty}\left(0,T_{\epsilon};L^{2+\kappa}(\Omega)\right)}\left\Vert R\right\Vert _{L^{2}\left(0,T_{\epsilon};L^{\infty}(\Omega)\right)}\\
 & \leq C\left(K,\alpha\right)\epsilon^{M-1}\epsilon^{M-\frac{1}{2}}\epsilon^{-\frac{\kappa}{2+\kappa}M}\epsilon^{M-\frac{3}{2}}\epsilon^{-\left(M+2\right)\alpha}
\end{align*}
for $\alpha>0$ , where we used (\ref{eq:w2epsab}), (\ref{eq:Main-est-d})
and Lemma~\ref{lem:R-L2-Linf} 1).
Again $M\geq 4$ and a suitable choice of $\alpha>0$ and $\kappa>0$
yield the final estimate by $C(T,\eps)\eps^{2M-1}$. For the term involving $\mathbf{v}_{M-\frac{1}{2}}^{-,\epsilon}$
we obtain
\begin{align*}
  &\epsilon^{M-\frac{1}{2}}\int_{\Omega_{T_{\epsilon}}}\left|\mathbf{v}_{M-\frac{1}{2}}^{-,\epsilon}\cdot\nabla\left(\gamma^{2}\right)R^{2}\right|\d(x,t)\\
  & \leq C\epsilon^{M-\frac{1}{2}}\left\Vert \mathbf{v}_{M-\frac{1}{2}}^{-,\epsilon}\right\Vert _{L^{2}\left(0,T_{\epsilon};L^{\infty}\left(\Omega^{-}(t)\right)\right)}\left\Vert \gamma R\right\Vert _{L^{\infty}\left(0,T_{\epsilon};L^{2}(\Omega)\right)}\cdot\left\Vert R\right\Vert _{L^{2}\left(\partial_{T_{\epsilon}}\Omega\left(\frac{\delta}{2}\right)\right)}\le C(K)\epsilon^{3M-\frac{1}{2}}
\end{align*}
where we used (\ref{eq:muv0,5est}) (together with $H^{2}\left(\Omega^{-}(t)\right)\hookrightarrow L^{\infty}\left(\Omega^{-}(t)\right)$)
and (\ref{eq:Main-est}) in the second estimate.

Thus, we have shown 
\[
\int_{0}^{T_{\epsilon}}\left|\int_{\Omega}\gamma^{2}R\left(\mathbf{v}^{\epsilon}\cdot\nabla R\right)\d x\right|\d t\leq C(K)C\left(T_{\epsilon},\epsilon\right)\epsilon^{2M-1}
\]
 and with that may conclude using (\ref{eq:Main-New}) that
\begin{align}
\sup_{t\in\left(0,T_{\epsilon}\right)}\left\Vert \gamma R\left(.,t\right)\right\Vert _{L^{2}(\Omega)}^{2}+\epsilon\left\Vert \gamma\Delta R\right\Vert _{L^{2}\left(\Omega_{T_{\epsilon}}\right)}^{2}\nonumber \\
+\frac{1}{\epsilon}\left\Vert \left(\gamma\nabla R,\gamma R\nabla R\right)\right\Vert _{L^{2}\left(\Omega_{T_{\epsilon}}\right)}^{2} & \leq C(K)C\left(T,\epsilon\right)\epsilon^{2M-1}.\label{eq:mainestzweiterteil}
\end{align}
where $C(T,\eps)\to_{(T,\eps)\to 0} 0$.

Altogether we may now choose $\epsilon_{0}>0$ and $T\in\left(0,T_{0}\right]$
so small that (\ref{eq:Main-est-a})\textendash (\ref{eq:Main-est-c})
follow for $T_{\epsilon}=T$ from (\ref{eq:CHS-zentral}) and (\ref{eq:CHS-Zusatz})
as a consequence of the estimates for $\mathcal{RS}$. (\ref{eq:Main-est-d})
follows for $T_{\epsilon}=T$ from (\ref{eq:mainestzweiterteil}).
This shows (\ref{eq:Main}). Regarding (\ref{eq:Mainv}), we have
by the definition of $\mathbf{v}_{err}^{\epsilon}$ in (\ref{eq:vepserrderf})
for $q\in\left(1,2\right)$
\begin{align*}
\left\Vert \mathbf{v}^{\epsilon}-\mathbf{v}_{A}^{\epsilon}\right\Vert _{L^{1}\left(0,T;L^{q}(\Omega)\right)} & \leq\left\Vert \mathbf{v}_{err}^{\epsilon}+\twe+\twz\right\Vert _{L^{1}\left(0,T;L^{q}(\Omega)\right)}+C\left\Vert \overline{\mathbf{v}^{\epsilon}}-\mathbf{v}_{A}^{\epsilon}\right\Vert _{L^{1}\left(0,T;L^{2}(\Omega)\right)}\\
 & \leq C\left(K,q\right)\epsilon^{M-\frac{1}{2}}
\end{align*}
by (\ref{eq:w1epsab}), (\ref{eq:w2epsab}) and Theorem \ref{vehler}.
The convergence results (\ref{eq:Maincconverge}) and (\ref{eq:Mainvconverge})
are then due to the construction of $c_{A}^{\epsilon}$ and $\mathbf{v}_{A}^{\epsilon}$,
more precisely to the discussed form of the zero-th order terms, where
it is important to note (\ref{eq:muv0,5est}) for $\mathbf{v}_{M-\frac{1}{2}}^{\pm,\epsilon}$.
This finishes the proof of Theorem~\ref{Main}.
\begin{rem}
\label{rem:nodirnosol} In this final remark, we want to discuss the
consequences of considering Neumann boundary conditions $\partial_{\mathbf{n}_{\partial\Omega}}\mu^{\epsilon}=0$
on $\partial_{T_{0}}\Omega$ instead of $\mu^\epsilon=0$. Of
course, in this case we would construct $\mu_{A}^{\epsilon}$ such
that $\partial_{\mathbf{n}_{\partial\Omega}}\mu_{A}^{\epsilon}=0$
is satisfied on $\partial_{T_{0}}\Omega$. To gain (\ref{eq:CHS-Majoreq}),
which is a vital point of the proof, we need to ensure that 
\[
\int_{\Omega}\varphi\Delta\left(\mu^{\epsilon}-\mu_{A}^{\epsilon}\right)\d x=\int_{\Omega}\Delta\varphi\left(\mu^{\epsilon}-\mu_{A}^{\epsilon}\right)\d x
\]
holds, which is satisfied if we choose Neumann boundary conditions
for $\varphi$. In particular, $\varphi$ should be the solution to
\begin{equation}
-\Delta\varphi\left(.,t\right)=R\left(.,t\right)\text{ in }\Omega,\;\partial_{\mathbf{n}_{\partial\Omega}}\varphi=0\text{ on }\partial\Omega,\label{eq:neumannphi}
\end{equation}
together with $\int_{\Omega}\varphi\left(.,t\right)\d x=0$.
However, in order for (\ref{eq:neumannphi}) to be well-posed, $\int_{\Omega}R\left(.,t\right)\d x=0$
needs to be satisfied, where
\begin{align*}
\int_{\Omega}R(x,t)\d x & =\int_{0}^{t}\int_{\Omega}\partial_{t}\left(c^{\epsilon}-c_{A}^{\epsilon}\right)\d x\d\tau+\int_{\Omega}c_{0}^{\epsilon}-c_{A}^{\epsilon}|_{t=0}\d x\\
 & =\int_{0}^{t}\int_{\Omega}\operatorname{div}\left(\mathbf{v}_{A}^{\epsilon}\right)c_{A}^{\epsilon}+\left.\twe\right|_{\Gamma}\cdot\nabla c_{A}^{\epsilon}\xi-\rc\d x\d\tau+\int_{\Omega}c_{0}^{\epsilon}-c_{A}^{\epsilon}|_{t=0}\d x
\end{align*}
in the case of no-slip boundary conditions for $\mathbf{v}^{\epsilon}$.
This expression does not vanish and we are not able to estimate it to a high enough power of $\eps$. A similar problem
arises in the case of periodic boundary conditions. To circumvent
this difficulty, we decided to stick to Dirichlet boundary values for $\mu$.
\end{rem}

\section*{Acknowledgement} 
The results of this paper are part of the second author's PhD Thesis. The authors acknowledge support by the SPP 1506 ``Transport Processes
at Fluidic Interfaces'' of the German Science Foundation (DFG) through the grant AB285/4-2.

\providecommand{\bysame}{\leavevmode\hbox to3em{\hrulefill}\thinspace}
\providecommand{\MR}{\relax\ifhmode\unskip\space\fi MR }
\providecommand{\MRhref}[2]{%
  \href{http://www.ams.org/mathscinet-getitem?mr=#1}{#2}
}
\providecommand{\href}[2]{#2}

\end{document}